\documentclass[11pt,reqno]{amsart}
\usepackage{type1ec}
\usepackage[utf8]{inputenc}
\usepackage[T1]{fontenc}
\usepackage{comment}
\usepackage{enumitem}
\usepackage{amssymb,amsfonts,amsthm}

\usepackage{enumitem}
\usepackage{amsthm}
\usepackage{amssymb}
\usepackage[french,english]{babel}
\usepackage{enumerate}
\usepackage{setspace}

\usepackage{hyperref}

\usepackage{color}
\definecolor{Red}{cmyk}{0,1,1,0.2}

\usepackage{amsfonts,amsmath,layout}

\newcommand{\Q}{\mathbb Q}

\newcommand{\R}{\mathbb R}

\newcommand{\W}{\mathbb{W}}

\def\R{\mathbb R}
\def\N{\mathbb N}

\def\E{\mathbb E}
\def\P{\mathbb P}
\def\B{\mathbb{B}}

\def\J{{\mathcal J}}
\def\ind{\mathbf{1}_}
\def\e{\varepsilon}
\newcommand{\be}{\begin{equation}}
\newcommand{\ee}{\end{equation}}

\def\1{{\bf 1}}

\def\sgn{{\rm sgn}}

\def\ds{\displaystyle}

\newcommand{\pare}[1]{\left (#1\right )}
\newcommand{\croc}[1]{\left [#1\right ]}

\newtheorem{Theorem}{Theorem}[section]
\newtheorem{Definition}[Theorem]{Definition}
\newtheorem{Proposition}[Theorem]{Proposition}
\newtheorem{Lemma}[Theorem]{Lemma}
\newtheorem{Corollary}[Theorem]{Corollary}
\newtheorem{Remark}[Theorem]{Remark}

\newtheorem{Sketch of proof}[Theorem]{Sketch of proof}

\begin{document}
\title[Walsh spider process with spinning measure selected from its own local time]{Martingale problem for a Walsh Spider process with spinning measure selected from its own local time}
\author[Miguel Martinez \& Isaac Ohavi]{Miguel Martinez$^{\dagger}$, Isaac Ohavi$^{\star}$\\
{$^\dagger$ Universit\'e Gustave Eiffel, LAMA UMR 8050, France}\\{$^\star$ Hebrew University of
Jerusalem, Department of Statistics, Israël}}
\email{miguel.martinez@univ-eiffel.fr}
\email{isaac.ohavi@mail.huji.ac.il \& isaac.ohavi@gmail.com}
\thanks{Research partly funded by the Bézout Labex, funded by ANR, reference ANR-10-LABX-58\\Research partially supported by the GIF grant 1489-304.6/2019}
\dedicatory{Version: \today}
\maketitle

\begin{abstract}
The objective of this article is to prove existence and weak uniqueness of a Walsh spider diffusion process, whose spinning measure and coefficients are allowed to depend on the local time spent at the junction vertex. The methodology is to show carefully that an effectively designed martingale problem is well-posed. 
Exploiting fully the results coming from the pioneering work of \cite{Freidlin-Wentzell-2}, the construction of the solution is performed using a concatenation procedure, as introduced in the seminal reference \cite{Stroock}. Uniqueness is shown by making use of the recent results obtained in \cite{Martinez-Ohavi EDP} for the solution of the corresponding parabolic PDE that involves a new class of transmission condition called {\it local time Kirchhoff's transmission condition}. As a byproduct of our main result, we manage to compute the explicit law of the diffusion when it behaves as a standard Brownian motion on each branch. The case $I=2$ permits us to derive also that there is existence and uniqueness for solutions of generalized SDE on the real line that involve the local time of the unknown process in all its coefficients. 
\end{abstract}

\section{Introduction}
\subsection{Introduction}
\label{sec:intro}
This article should be regarded as the first part of a two papers series concerning Walsh spider processes whose spinning measure and coefficients are allowed to depend on the local time at the junction vertex. In this first contribution, we construct the process and prove that there is uniqueness in the weak sense of its law on the joint spider's path space-local time space. Because this contribution is already quite long, we have decided to push back in a subsequent work the further study of some crucial properties, such as the derivation of Markov's property, It\^o's formula, and quadratic approximation of the local time.  

Up to our knowledge, this is the first result in literature that provides existence and uniqueness of a Walsh spider's diffusion with non constant spinning measure. We believe that this contribution can lead to the formulation of new original problems in the field of stochastic analysis and partial differential equations. For example, we think of the study of viscosity solutions for non degenerate HJB systems posed on a star-shaped network with non linear Kirchhoff's boundary condition (currently investigated by the second author) and the rigorous formulation of a stochastic scattering control problem for a Walsh spider diffusion, with optimal directions selected from its own local time. 

We are interested in the construction - in a weak sense - of a particular type of singular processes in the plane. The idea was first introduced by Walsh in the epilogue of \cite{Walsh}. As explained in \cite{Barlow-Pitman-Yor}: {\it (...) Started at a point $z$ in the plane away from the origin $0$, such a process moves like a one-dimensional Brownian motion along the ray joining $z$ and $0$ until it reaches $0$. Then the process is kicked away from $0$ by an entrance law which makes the radial part of the diffusion a reflecting Brownian motion, while randomizing the angular part"}. The difficulty arises since, as Walsh explains with certain sense of humor (quoted from \cite{Walsh} and \cite{Barlow-Pitman-Yor}):
{\it (...) It is a diffusion which, when away from
the origin, is a Brownian motion along a ray, but which has what might be called a round-house singularity at the origin: when the process enters it, it, like Stephen Leacock’s hero, immediately rides off in all directions at once.}
Rigorous construction of Walsh's Brownian motion may be found in \cite{Barlow-Pitman-Yor} (see also the references given therein).

Generalizing the idea of Walsh, diffusions on graphs were introduced in the seminal works of Freidlin and Wentzell \cite{Freidlin-Wentzell-2} and Freidlin and Sheu \cite{freidlinS} for a star-shaped network (and afterwards for open books in \cite{Freidlin}). 
For the sake of conciseness and also because of the difficulty of the subject, in this contribution we will focus on the case where the diffusion lives on a star-shaped network:~these diffusions are called {\it spiders} and, sticking with this metaphor, the state-space junction network may be referred as {\it the spider's web}. 

Walsh spider diffusion processes are currently being thoroughly studied and extended to various settings. Let us mention the following recent articles amongst the vast literature on the subject: in \cite{Ichiba} the authors propose the construction of stochastic integral equations related to Walsh semimartingales, in \cite{Ichiba-2} the authors compute the possible stationary distributions, in \cite{Karatzas-Yan} the authors investigate stopping control problems involving Walsh semimartingales, in \cite{Atar} the authors study related queuing networks, whereas \cite{Bayraktar} addresses the problem of finding related stopping distributions. We refer to the introduction of \cite{Ichiba} for a comprehensive survey, the reader may also find therein many older references on the subject.
\medskip

The difficulty when constructing spider processes comes from the fact that the natural filtration generated by a spider process is not a Brownian filtration as soon as the underlying spider web possesses three branches or more. Although the motion behaves as a one dimensional semimartingale during its stay on a particular branch and is driven by a Brownian motion along the branch, the fact that the process {\it 'has more than two directions to choose infinitesimally when moving apart from the junction point'} makes it impossible for a spider diffusion to be adapted to a Brownian filtration (whatever the dimension of the underlying Brownian motion). This crucial fact imposes that it is not possible to apply directly It\^o's stochastic calculus theory to both components $(x,i)$ of the spider process (where $x$ stands for the distance to the vertex junction and $i$ stands for the label of the branch), but only to the distance component $x$. As mentioned in \cite{Lejay}: {\it ''this fact has given rise to an abundant literature on Brownian filtrations''}. We mention \cite{DeMeyer} for a comprehensive argument that the natural filtration cannot be generated by a Brownian motion when $I\geq 3$ and also the unavoidable reference \cite{Barlow} (see also \cite{tsirel}) for a theoretical study of these questions.

The relationship between Walsh diffusion processes and skew diffusions (that correspond to the particular case where $I=2$) may be found in the survey \cite{Lejay}. It is notable that skew diffusions solve stochastic differential equations that involve the local time of the unknown process (in this case the natural filtration of the process is a Brownian filtration) . 

Although difficult, several constructions of Walsh's diffusions have been proposed in the literature, see  for e.g. \cite{Barlow-Pitman-Yor} for a construction based on Feller’s semigroup theory, \cite{Salisbury} for a construction using the excursion theory for right processes, and also the very recent preprint \cite{Bayraktar-2} that proposes a new construction of Walsh diffusions using time changes of multi-parameter processes. 

Note that in all these constructions, the spinning measure of the process -- that is strongly related somehow to {\it 'the way of selecting infinitesimally the different branches from the junction vertex'} -- remains constant through time.

In this article, we will exploit the results coming from \cite{Freidlin-Wentzell-2}. Before explaining the whys and wherefores of our contribution, we need to explain briefly the results contained therein, since they constitute the starting point of our investigations.

\medskip 
Let $I$ a positive integer ($I\geq 2$). Denote $[I] := \{1,\dots, I\}$ and
\begin{align*}
\mathcal{J}~~:=~~\{{\bf 0}\}\cup \pare{(0,\infty)\times [I]}.
\end{align*} 
where ${\bf 0} = \{(0,j), j\in [I]\}$ is the junction vertex equivalence class. In the sequel, we will often denote this equivalence class by one of its representatives.

Assume we are given $I$ pairs $(\sigma_{i},b_{i})_{i\in [I]}$ of mild coefficients from $[0,+\infty)$ to $\R$ satisfying the following condition of ellipticity: $\forall i\in [I],~\sigma_i>0$. We are also given $\big(\alpha_1,\ldots,\alpha_I)$ positive constants satisfying $\displaystyle \sum_{i=1}^I \alpha_i=1$. It is proved in \cite{Freidlin-Wentzell-2} that there exists a continuous Feller Markov process $\big(x(\cdot),i(\cdot)\big)$ valued in ${\mathcal{J}}$, whose generator is given by the following operator:
$$\mathcal{L}:\begin{cases}\mathcal{C}^2(\mathcal{J})\to \mathcal{C}(\mathcal{J}),\\
f=f_i(x)\mapsto 
b_i(x)\partial_xf_i(x)+\displaystyle \frac{\sigma_i^2(x)}{2}\partial_x^2f_i(x)\end{cases},$$
with domain
$$D(\mathcal{L}):=\Big\{ f \in \mathcal{C}^2(\mathcal{J}),~~\displaystyle \sum_{i=1}^I\alpha_i\partial_xf_i(0)=0\Big\}.$$
In the above, for $k=0,1,2,\dots$ the $k$-th order continuous class space on the junction network $\mathcal{C}^k(\mathcal{J})$ is defined as 
\[
\left \{f:~\mathcal{J}\rightarrow \R,\,\,(x,i)\mapsto f_i(x)\;\;\text{s.t.}\;\;\;\forall (i,j)\in [I]^2,\;\;\;f_i(0) = f_j(0),\;\;\;f_i\in {\mathcal C}^k([0,\infty))\right \}.
\]
The network $\mathcal{J}$ and the generator $(\mathcal{L}, D(\mathcal{L}))$ arise naturally as some kind of limit when applying an averaging principle to a family of diffusion processes living in the Euclidean space $\R^2$ exhibiting a "fast" and "slow" component. The characteristics of the limiting
process are obtained by averaging the characteristics of the slow process along the "fast" directions with respect to the stationary distribution of the "fast" Markov process. When the fast component tends to concentrate to $\mathcal{J}$ we may obtain a Markov process with a generator of type $(\mathcal{L}, D(\mathcal{L}))$. The junction point then appears as a point of equilibrium of the fast component and the positive constants $\big(\alpha_1,\ldots,\alpha_I)$ as a stationary distribution for each of the selected directions (see \cite{Freidlin-Wentzell-2} for details).
\medskip

\medskip

Our objective in this paper is to extend the previous mentioned existence result obtained by Freidlin and Wentzell in \cite{Freidlin-Wentzell-2} for spider motions by allowing all coefficients of the process - including the spinning measure - to depend on the own local time of the process spent at the junction, together with the current running time. 

With this perspective in mind, we felt that we could not follow the lines of construction of \cite{Freidlin} and \cite{freidlinS}. 
Instead, we take the results stated in \cite{Freidlin-Wentzell-2} as a starting point and we construct 'by hand' a solution of a martingale problem that is purposely designed in order to take the presence of the local time in all the leading coefficients into account. Of course, since we add the local time in the picture, the canonical space has to be extended accordingly.

In order to exhibit a martingale $\big(\mathcal{S}_{pi}-\mathcal{M}_{ar}\big)$ problem adapted to our expectations, we need to compute the generator of our potential spider motion which possesses now three components $(x(t),i(t),\ell^0_t)_{t\in [0,T]}$. Here $(\ell_t^0)_{t\in [0,T]}$ stands for the local time of the expected spider process at the junction.

It appears that the following {\it local time Kirchhoff's transmission condition}:
\begin{equation}
\label{eq:condition-transmission-intro}
\partial_l f(t,0,l)+\displaystyle \sum_{i=1}^I \alpha_i(t,l)\partial_x f_i(t,0,l)=0,\;\;\;(t,l)\in(0,T)\times(0,+\infty)
\end{equation}
must hold for any $f:~[0,T]\times \mathcal{J}\times \R^+\rightarrow \R$, $(t,x,i,l)\mapsto f_i(t,x,l)$ satisfying the continuity condition $f_i(t,0,l)=f_j(t,0,l)=f(t,0,l),\;\;(t,l)\in[0,T]\times[0,+\infty)^2,\;(i,j)\in[I]^2$. In the above, the variable $l$ corresponds to the local time $(\ell_t^0)_{t\in [0,T]}$. For an explanation on the reasons why this {\it local time Kirchhoff's transmission condition} \eqref{eq:condition-transmission-intro} arises naturally in our problem, the reader might want to have a look at the paragraph \ref{subsection-generator} in the present paper.

The parabolic differential equation corresponding to 
the generator with condition transmission \eqref{eq:condition-transmission-intro} has been studied in our paper \cite{Martinez-Ohavi EDP}. The results contained therein are of crucial importance when turning to the difficult problem of uniqueness for the martingale problem $\big(\mathcal{S}_{pi}-\mathcal{M}_{ar}\big)$. 

The statement of the martingale problem $\big(\mathcal{S}_{pi}-\mathcal{M}_{ar}\big)$ studied in this paper is given in Section \ref{subsec: Main results}. Notably, remark that the tests functions do not depend on the local time variable: our result implies that whenever we are looking at some continuous process $(x(t),i(t),l(t))_{t\in [0,T]}$ with state space $\mathcal{J}\times \R^+$ that is starting from $(x_\ast,i_\ast,l_\ast)$ and such that for any $f\in\mathcal{C}^{1,2}_b([0,T]\times \mathcal{J})$
\begin{align}
\label{eq:def-V-intro}
&\Bigg(f_{i(s)}(s,x(s))- f_{i}(0,x)\displaystyle -\int_{0}^{s}\pare{\partial_tf_{i(u)}(u,x(u))+\displaystyle\frac{1}{2}\sigma_{i(u)}^2(u,x(u),l(u))\partial_{xx}^2f_{i(u)}(u,x(u))}du\nonumber\\
&\displaystyle-\int_{0}^{s}\pare{b_{i(u)}(u,x(u),l(u))\partial_xf_{i(u)}(u,x(u))}du-\displaystyle\sum_{j=1}^{I}\int_{0}^{s}\alpha_j(u,l(u))\partial_{x}f_{j}(u,0)dl(u)~\Bigg)_{0\leq s\leq T}\nonumber,
\end{align} 
is a martingale under some probability measure $\P$, we are allowed to deduce that $(\ell^0_t:= l(t) - l_\ast)_{t\in [0,T]}$ is in fact the local time of the spider $(x(t),i(t))_{t\in [0,T]}$ as long as it is $\P$-almost surely increasing only on the times where the component $(x(t))_{t\in [0,T]}$ equals $0$. 
\medskip

Our main theorem - Theorem \ref{th: exist Spider } - states that the martingale problem $\big(\mathcal{S}_{pi}-\mathcal{M}_{ar}\big)$ is well-posed on the corresponding canonical space.
The existence proof relies on a careful adaptation of the seminal construction for solutions of classical martingale problems that have $\R^d$ as the underlying state space, and uses concatenation of probability measures, such as performed in the essential reference book \cite{Stroock}. The uniqueness proof is more involved and, as already mentioned, comes essentially from the results obtained  in \cite{Martinez-Ohavi EDP} for the associated parabolic PDE with a {\it local time Kirchhoff's transmission condition} given by \eqref{eq:condition-transmission-intro}.

We derive two results from this main result. First, we manage to compute the transition density kernel of $(x(t),i(t),\ell^0_t)_{t\in [0,T]}$ when the spider behaves as a Brownian motion on each branch, but with a spinning measure that is now allowed to depend on the local time.

Another byproduct of our results -- when turning to the particular case $I=2$ and making an analogy with $\R^{-}$ and $ \R^+$ for the branches $1$ and $2$ -- implies that there exists a unique weak solution of stochastic differential equations with real state space of type
\begin{equation}
\label{eq:sde-I-equals-2-intro}
y_0 = y,\hspace{0,4 cm}dy(s) = \tilde{\sigma}(s,y(s),\ell_s^0(y))dW_s + \tilde{b}(s,y(s),\ell_s^0(y))ds + \beta(s,\ell^0_s(y))d\ell_s^0(y)\;\;s\in [0,T]
\end{equation}
where $(\ell^0_t(y))_{t\in [0,T]}$ stands now for the local time at zero of the unknown solution process $(y_t)_{t\in [0,T]}$ that takes now values on the real line. These equations generalize the time inhomogeneous stochastic differential equations involving the local time of the unknown process studied in \cite{Etore-Martinez-2} and also the {\it variably skewed Brownian motion} introduced in \cite{Barlow-Burdzy} (see also \cite{Cox}). Note that the solution of \eqref{eq:sde-I-equals-2-intro} gives a new example of a process that belongs to the class $(\Sigma)$ for semimartingales: for an account on the class $(\Sigma)$ and the origins of the class $(\Sigma)$, we refer the reader to \cite{Eyi-Obiang-al}, \cite{Eyi-Obiang-al-2} and the references therein.

\subsection{Organization of the paper}

In Section \ref{sec: main assumptions} we introduce the main notations and assumptions that will be used throughout the paper. We follow by providing a remainder of the principal results concerning spider diffusions with homogeneous coefficients and constant spinning measure constructed in the seminal papers \cite{Freidlin-Wentzell-2} and \cite{freidlinS} that will be of constant use in the sequel in a form that is adapted for our purposes.
Section \ref{sec:main-results} is devoted exclusively to the formulation of the martingale problem $\big(\mathcal{S}_{pi}-\mathcal{M}_{ar}\big)$ and the statement of our main result  Theorem \ref{th: exist Spider  }, which asserts that this martingale problem is well-posed (in the classical sense given in \cite{Stroock}). 

Before moving to the existence proof, we need to supply the main ingredients that will be used for our construction: this is done in Section \ref{sec: concatenation}, where we impart the necessary material for the concatenation of probability measures on the spider's web-local time path space. 

Once all ingredients for the proof have been properly introduced, we perform the construction in Section \ref{sec: preuve et paths properties} (Theorem \ref{th: exist Spider  } - Existence) with the use of Prokhorov's theorem. We end this section by giving the natural representation of the first component of the spider with $(x(t))_{t\in [0,T]}$ described as an It\^o reflected process.

Next, we move on to the problem of the uniqueness in Section \ref{sec:uniqueness}: we start by introducing the backward parabolic differential equation associated to the martingale problem $\big(\mathcal{S}_{pi}-\mathcal{M}_{ar}\big)$ (that we studied in \cite{Martinez-Ohavi EDP}) and further use this PDE to prove the uniqueness result {\it via} a regularization procedure (Theorem \ref{th: exist Spider  } - uniqueness).

We end the paper with Section \ref{sec:deriving-results} by bringing forward two results that can be deduced from our study and that we believe are interesting {\it per se} (see above the last paragraph of our introduction).
\section{Notations - Main assumptions - Remainder on spider diffusions}\label{sec: main assumptions}
\subsection{Notations and assumptions}
\label{subsec: Notations and preliminary results}
Fix $I\geq 2$ an integer. We denote $[I] = \{1,\dots, I\}$ and consider $\mathcal{J}$ a junction space with $I$ edges defined by
\begin{align*}
\mathcal{J}~~:=~~\{{\bf 0}\}\cup \pare{(0,\infty)\times [I]}.
\end{align*}

All the points of $\mathcal{J}$ are described by couples $(x,i)\in [0,\infty)\times [I]$  with the junction point ${\bf 0}$ identified with the equivalent class $\{(0,i)~:~i\in [I]\}$.

Thus and with a slight abuse of notation, the common junction point ${\bf 0}$ of the $I$ edges will be often denoted be $0$ and we will also often identify the space $\mathcal{J}$ with a union of $I$ edges
$J_i=[0,+\infty)$ satisfying $J_i\cap J_j=\{0\}$ whenever $(i,j)\in [I]^2$ with $i\neq j$. With these notations $(x,i)\in {\mathcal{J}}$ is equivalent to asserting that $x\in J_i$.
\vspace{0,3 cm}

We endow naturally ${\mathcal{J}}$ with the distance $d^{\mathcal{J}}$ defined by
\begin{equation*}
\forall \Big((x,i),(y,j)\Big)\in \mathcal{J}^2,~~d^\mathcal{J}\Big((x,i),(y,j) \Big)  := \left\{
\begin{array}{ccc}
 |x-y| & \mbox{if }  & i=j\;,\\ 
x+y & \mbox{if }  & {i\neq j},\;
\end{array}\right.
\end{equation*}
so that $\pare{\mathcal{J}, d^{\mathcal{J}}}$ is a Polish space.
\vspace{0,3 cm}

For $T>0$, we introduce the time-space domain $\mathcal{J}_T$ defined by
\begin{align*}\mathcal{J}_T~~&:=~~[0,T]\times\mathcal{J},
\end{align*}
and consider $\mathcal{C}^{\mathcal{J}}[0,T]$ the Polish space of maps defined from $[0,T]$ onto the junction space $\mathcal{J}$ that are continuous w.r.t. the metric $d^{\mathcal{J}}$. The space $\mathcal{C}^{\mathcal{J}}[0,T]$ is naturally endowed with the uniform metric $d^\mathcal{J}_{[0,T]}$ defined by: 
\begin{equation*}
\forall \Big((x,i),(y,j)\Big)\in \pare{\mathcal{C}^{\mathcal{J}}[0,T]}^2,~~d^\mathcal{J}_{[0,T]}:=\sup_{t \in [0,T]}d^\mathcal{J}\Big((x(t),i(t)),(y(t),j(t)) \Big).
\end{equation*}
Together with $\mathcal{C}^{\mathcal{J}}[0,T]$, we introduce
\begin{eqnarray*}
\mathcal{L}[0,T]~~:=~~\Big\{l:[0,T]\to \R^+,\text{ continuous non decreasing}\Big\}
\end{eqnarray*} 
endowed with the usual uniform distance $|\,.\,|_{(0,T)}$.

The modulus of continuity on $\mathcal{C}^{\mathcal{J}}[0,T]$ and $\mathcal{L}[0,T]$ are naturally defined for any $\theta\in (0, T]$ as
\begin{align*}
&\forall X=\big(x,i\big)\in \mathcal{C}^{\mathcal{J}}[0,T],\\
&\hspace{1,3 cm}\omega\pare{X,\theta}= \sup\Big\{d^{\mathcal{J}}\pare{(x(s),i(s)),(x(u),i(u))} ~\big\vert~(u,s)\in [0,T]^2,~~|u-s|\leq \theta\Big\};\\ 
&\forall f\in {\mathcal{L}}[0,T],\\
&\hspace{1,3 cm}\omega\pare{f,\theta}= \sup \Big\{|f(u)-f(s)|~\big\vert~(u,s)\in [0,T]^2,~~|u-s|\leq \theta\Big\}.
\end{align*}

We then form the product space 
\begin{align*}
\Phi~~=~~\mathcal{C}^{\mathcal{J}}[0,T]\times \mathcal{L}[0,T]
\end{align*}
 considered as a measurable Polish space equipped with its Borel $\sigma$-algebra $\mathbb{B}(\Phi)$ generated by the open sets relative to the metric $d^{\Phi}:=d^{\mathcal{J}}_{[0,T]} + |\,.\,|_{(0,T)}$. 

The canonical process $\tilde{X}$ on $(\Omega, \mathcal{F}) := (\Phi, \mathbb{B}(\Phi))$ is defined as
\[\tilde{X}:\begin{array}{cll}
[0,T]\times \Omega&\to\;\;{\mathcal J}\times \R^+\\
(t,\omega) &\mapsto\;\;\tilde{X}(t,\omega) := \omega(t).
\end{array}
\]
We denote by $(\Psi_t:=\sigma(\tilde{X}(s), 0\leq s \leq t))_{0\leq t\leq T}$ the canonical filtration on $\pare{\Phi, \mathbb{B}(\Phi)}$.

Note that we will often write $(x_\ast,i_\ast,l_\ast)$ instead of $((x_\ast,i_\ast),l_\ast)$ an element of ${\mathcal J}\times \R^+$.
Also we will sometimes need to separate the coordinates of $\tilde{X}$ and we will denote by $X$ (without the '\;$\tilde{}$\;' superscript) the process formed by the first two coordinates of $\tilde{X}$, namely
\[
X(t) = (x(t),i(t)),\hspace{0,3 cm}\forall t\in [0,T],
\]
corresponding to $\tilde{X} = (x,i,l)\in \Phi$. 

Moreover, in order to emphasize the fact that the canonical process is an application on $\Phi$, we will sometimes have to abusively write $\tilde{X}_t(w)$ or $\pare{(x_t(w), i_t(w), l_t(w)}$ instead of $\tilde{X}(t,\omega)$: the reader should be aware from the context that the argument $w$ stands here for an arbitrary element of $\Phi$. Most times, we will simply write $\omega = (x,i,l)$ to denote the elements of $(\Phi,\mathbb{B}(\Phi))$ and 
\[\tilde{X} = (\tilde{X}(t, (x,i,l)))_{t\in [0,T]} = (x(t),i(t),l(t))_{t\in [0,T]} = (x,i,l)\] when no confusion is possible.

The modulus of continuity on $\Phi$ is naturally defined for any $\theta\in (0, T]$ as
\begin{align*}
&\forall \tilde{X}=\big(x,i,l\big)\in \mathcal{C}\pare{[0,T];{\mathcal{J}\times \R^+}},\\
&\hspace{0,5 cm}\tilde{\omega}(\tilde{X},\theta)= \sup\Big\{d^{\mathcal{J}}\pare{(x(s),i(s)),(x(u),i(u))} + |l(u) - l(s)| ~\big\vert~(u,s)\in [0,T]^2,~~|u-s|\leq \theta\Big\}.
\end{align*}
The set of the probability measures on the measurable space $(\Phi,\mathbb{B}(\Phi))$ is denoted by $\mathcal{P}(\Phi,\mathbb{B}(\Phi))$. We endow this space with the topology of weak convergence of probability measures. Since $\Phi$ is a Polish space, this topology is induced by the Prokhorov metric $\pi$ on $\mathcal{P}(\Phi,\mathbb{B}(\Phi))$. Recall that the Prokhorov distance $\pi(P,Q)$ between two probability measures on $(\Phi,\mathbb{B}(\Phi))$ is defined as
\begin{align*}
&\forall P,Q \in \mathcal{P}(\Phi,\mathbb{B}(\Phi)),\\
&\hspace{1,3 cm}\pi(P,Q)~~=~~\inf\left \{\varepsilon>0~\vert~\forall A\in \mathbb{B}(\Phi),\;\; P(A)\leq Q(A^{\varepsilon}) + \varepsilon \text{ and }Q(A)\leq P(A^{\varepsilon}) + \varepsilon\right \}
\end{align*}
where $A^{\varepsilon}$ denotes the $\varepsilon$-dilation of the set $A\in \Phi$ ($A^{\varepsilon}=\{\tilde{X}\in \Phi,~d^{\Phi}(\tilde{X}, A)\leq \varepsilon\}$).

We introduce the set $\mathcal{P}([I])\subset [0,1]^I$   
\begin{align*}
    \mathcal{P}([I]):=\Big\{(\alpha_{i})\in[0,1]^{I}~\big\vert~ \sum_{i=1}^{I}\alpha_i=1\Big\}
\end{align*}
the simplex set giving all probability measures on $[I]$.

\vspace{0,7 cm}

We introduce the following data  
$$\begin{cases}
(\sigma_i)_{i \in \{1 \ldots I\}} \in \pare{\mathcal{C}_b([0,T]\times[0,+\infty)\times [0,+\infty);\R)}^I\\
(b_i)_{i \in \{1 \ldots I\}} \in \pare{\mathcal{C}_b([0,T]\times[0,+\infty)\times [0,+\infty);\R)}^I\\
\alpha=(\alpha_i)_{i \in \{1 \ldots I\}} \in \mathcal{C}([0,T] \times [0,+\infty);\, \mathcal{P}([I]))\\
\end{cases}
$$
satisfying the following assumption $(\mathcal{H})$ (where $({\bf A})$ stands for alpha, $({\bf E})$ for ellipticity, and $({\bf R})$ for regularity):

$$\textbf{Assumption } (\mathcal{H})$$
\begin{align*}
&({\bf A})~~\exists\;\ds \underline{a}\in \left (0,{1}/{I}\right ],~~\forall i \in \{1    \ldots    I\}, ~~\forall (t,l)\in [0,T]\times[0,+\infty),~~\alpha_i(t,l)~~\ge~~\underline{a}.\\
&({\bf E})~~\exists\,\,\underline{\sigma} >0,~~\forall i \in \{1    \ldots    I\}, ~~\forall (t,x,l)\in [0,T]\times[0,+\infty)\times [0,+\infty),~~\sigma_i(t,x,l)~~\ge~~\underline{\sigma}.\\
&({\bf R})~~\exists (|b|,|\sigma|,|\overline{a}|)\in (0,+\infty)^3,~~\forall i \in \{1 \ldots I\},\\
&({\bf R} - i)\hspace{0,5 cm}\displaystyle\sup_{t,x,l}|b_i(t,x,l)|\;\;\,+\sup_{t,\ell}\sup_{(x,y),\;x\neq y} \frac{|b_i(t,x,l)-b_i(t,y,l)|}{|x-y|}\\
&\displaystyle\hspace{3,5cm}+\sup_{t,x}\sup_{(l,l'),\;l\neq l'} \frac{|b_i(t,x,l)-b_i(t,x,l')|}{|l-l'|}\\
&\displaystyle
\hspace{3,5 cm}+\sup_{x,l}\sup_{(t,s),\;t\neq s} \frac{|b_i(t,x,l)-b_i(s,x,l)|}{|t-s|}~~\leq~~|b|,\\
&({\bf R} - ii)\hspace{0,5 cm}\displaystyle\sup_{t,x,l}|\sigma_i(t,x,l)|\;\;\,+\sup_{t,l}\sup_{(x,y),\;x\neq y} \frac{|\sigma_i(t,x,l)-\sigma_i(t,y,l)|}{|x-y|}\\
&\displaystyle\hspace{3,5cm}+\sup_{t,x}\sup_{(l,l'),\;l\neq l'} \frac{|\sigma_i(t,x,l)-\sigma_i(t,x,l')|}{|l-l'|}\\
&\displaystyle
\hspace{3,5 cm}+\sup_{x,l}\sup_{(t,s),\;t\neq s} \frac{|\sigma_i(t,x,l)-\sigma_i(s,x,l)|}{|t-s|}~~\leq ~~|\sigma|,\\
&({\bf R}- iii)\hspace{0,5 cm}\displaystyle \sup_t\sup_{(l,l'),\;l\neq l'} \frac{|\alpha_i(t,l)-\alpha_i(t,l')|}{|l-l'|}+\displaystyle\sup_l\sup_{(t,s),\;t\neq s} \frac{|\alpha_i(t,l)-\alpha_i(s,l)|}{|t-s|}~~\leq ~~|\overline{a}|.
\end{align*}
\vspace{0,3 cm}

Let us introduce $\mathcal{C}^{1,2}_b(\mathcal{J}_T)$ the class of function defined on $\mathcal{J}_T$ with regularity $\mathcal{C}^{1,2}_b([0,T]\times [0,\infty))$ on each edge, namely
\begin{align*}
&\mathcal{C}^{1,2}_b(\mathcal{J}_T):=\Big\{f:\mathcal{J}_T\to\R,~~(t,(x,i))\mapsto f_i(t,x)~\Big\vert\\
&\hspace{3,4 cm}~\forall i\in [I], \;f_i:[0,T]\times {J}_i\to\R,\,(t,x)\mapsto f_i(t,x)\in \mathcal{C}^{1,2}_b([0,T]\times J_i),\\
&\hspace{5,4 cm}~ \forall (t,(i,j))\in [0,T]\times [I]^{2}, \,f_i(t,0)=f_j(t,0)\Big\}.
\end{align*}

\subsection{A remainder of results regarding spider diffusion processes}

The seeds of our study come from the results of the seminal paper \cite{freidlinS} that we now restate as a single theorem adapted to our purposes.

\begin{Theorem}\label{th: exi Freidlin}(see Theorem 2.1, Lemma 2.2 and Lemma 2.3 in \cite{freidlinS})

Assume $\pare{\mathcal{H}}$.

Fix $(t,x_\ast,i_\ast,l_\ast)\in [0,T]\times \mathcal{J}\times [0,+\infty)$ and remember our convention that $\tilde{X} = (x,i,l)$ stands for the canonical process on $(\Phi, \B(\Phi))$. 

Then, there exists a unique probability measure $P_{t}^{x_\ast,i_\ast,l_\ast} \in \mathcal{P}(\Phi,\mathbb{B}(\Phi))$, satisfying the following characterization conditions $(\mathcal{S})$:
$$\textbf{Characterization conditions } (\mathcal{S})$$
-\;($\mathcal{S}$-i) For each $u\leq t$, $(x(u), i(u), l(u))=(x_\ast,i_\ast,l_\ast)$, $P_{t}^{x_\ast,i_\ast,l_\ast}$ a.s.\\ 
-\;($\mathcal{S}$-ii) 
There exists an adapted $\pare{P^{x_\ast,i_\ast,l_\ast}_t,\;(\Psi_s)_{0 \leq s\leq T}}$ increasing process $(\ell_s)_{0\leq s\leq T}$ satisfying $\ell_s = 0$ for any $s\in [0,t)$ $P^{x_\ast,i_\ast,l_\ast}_t-{\rm a. s.}$ and
constructed such that
\begin{equation}
\label{pte:fonctionnelle-additive}
l(s) = l_\ast + \ell_s\hspace{0,5 cm}\forall s\in [t,T],\hspace{1,0 cm}P_{t}^{x_\ast,i_\ast,l_\ast}-{\rm a.s.}
\end{equation}
The process $(\ell_s)_{t\leq s\leq t}$ \emph{is a progressively measurable process w.r.t. $(\Psi_s)_{0 \leq s\leq T}$ after time $t$} (in the sense defined by \cite{Stroock} Chapter I, Section 1.2 p.19).

The process $(\ell_s)_{s\in [0,T]}$ increases only on the set $\{s\in [t,T]~\vert~ x(s) = 0\}$, namely
\begin{eqnarray}
\label{eq:tps-local}
\displaystyle\displaystyle\int_{t}^{T}\mathbf{1}_{\{x(u)>0\}}d\ell_u=0,~~P_{t}^{x_\ast\,i_\ast,l_\ast} \text{ a.s.}
\end{eqnarray} 
-\;($\mathcal{S}$-iii) For any $f\in\mathcal{C}^{1,2}_b(\mathcal{J}_T)$, the process defined by

$\forall s\in[t,T]:$
\begin{align}
\label{eq:martingale}
&~~f_{i(s)}(s,x(s)) - f_{i_\ast}(t,x_\ast)\nonumber\\
&\hspace{0.3 cm}-\displaystyle\int_{t}^{s}\pare{\partial_tf_{i(u)}(u,x(u))+\displaystyle\frac{1}{2}\sigma_{i(u)}^2(t,x(u),l_\ast)\partial_{xx}^2f_{i(u)}(u,x(u))}du\nonumber\\
&\hspace{0.3 cm}-\displaystyle\int_t^s \pare{b_{i(u)}(t,x(u),l_\ast)\partial_xf_{i(u)}(u,x(u))}du\nonumber\\
&\hspace{0.3 cm}-\displaystyle\sum_{j=1}^{I}\int_{t}^{s}\alpha_j(t,l_\ast)\partial_{x}f_{j}(u,0)d\ell_u,
\end{align}
is a $(\Psi_s)_{t \leq s \leq T}$ martingale after time $t$ under $P_t^{x_\ast,i_\ast,l_\ast}$.
\medskip
\medskip

Moreover, $P_{t}^{x_\ast,i_\ast,l_\ast} \in \mathcal{P}(\Phi,\mathbb{B}(\Phi))$ satisfies the following properties:
$$\textbf{Properties } (\mathcal{M})$$

-\;($\mathcal{M}$-i) There exists a $\pare{P^{x_\ast,i_\ast,l_\ast}_t,\;(\Psi_s)_{t \leq s\leq T}}$ standard one dimensional Brownian $W^{(t,l_\ast)}$ such that for any $f\in\mathcal{C}^{1,2}_b(\mathcal{J}_T)$, for any $s\in[t,T]$,
\begin{align}\
\label{eq:martingale-mb}
&~~f_{i(s)}(s,x(s))- f_{i_\ast}(t,x_\ast)=\nonumber\\
&\hspace{1.0 cm}\int_t^s\sigma_{i(u)}(t,x(u),l_\ast)dW_u^{(t,l_\ast)}\nonumber\\
&\hspace{0.3 cm}~+\displaystyle\int_{t}^{s}\pare{\partial_tf_{i(u)}(u,x(u))+\displaystyle\frac{1}{2}\sigma_{i(u)}^2(t,x(u),l_\ast)\partial_{xx}^2f_{i(u)}(u,x(u))}du\nonumber\\
&\hspace{0.3 cm}~+\displaystyle\int_t^s b_{i(u)}(t,x(u),l_\ast)\partial_xf_{i(u)}(u,x(u))du\nonumber\\
&\hspace{0.3 cm}~+\displaystyle\sum_{j=1}^{I}\int_{t}^{s}\alpha_j(t,l_\ast)\partial_{x}f_{j}(u,0)d\ell_u.
\end{align} 

--\;($\mathcal{M}$-ii) The process $(x(s),i(s),\ell(s))_{t\leq s\leq T}$ satisfies
\begin{equation}
\label{eq:equation-freidlin-sheu}
dx(s) = \sigma_{i(s)}(t,x(s),l_\ast)dW^{(t,l_\ast)}_s + b_{i(s)}(t,x(s),l_\ast)ds + d\ell_s\;\;\;\;s\in [t,T]
\end{equation}
$P_t^{x_\ast,i_\ast,l_\ast}$ - a.s.

-\;($\mathcal{M}$-iii) Moreover, (cf. (2.10) p.185 in \cite{freidlinS})
\begin{equation}
\label{eq:non-stickness-fs}
\lim_{{\varepsilon}\searrow 0}E_t^{x_\ast,i_\ast,l_\ast}\croc{\int_{t}^T \ind{x(s)\leq \varepsilon}ds} = 0.
\end{equation}
\end{Theorem}
\begin{Remark}
The careful reader might notice that our canonical space $(\Phi, \B(\Phi))$ is not the canonical space used in \cite{freidlinS}, which is given in \cite{freidlinS} by $(\mathcal{J}_T, \B(\mathcal{J}_T))$. Let us explain this issue that might be considered as an abuse at first glance. Take $t$ and $l_\ast$ above as fixed extrinsic parameters and let $P^{x_\ast,i_\ast}$ denote the law on $(\mathcal{C}^{\mathcal{J}}[0,T], \B(\mathcal{C}^{\mathcal{J}}[0,T]))$ of the Feller process $X$ of Theorem 2.1 in \cite{freidlinS} starting from $(x_\ast,i_\ast)$ with coefficients $(\sigma_i, b_i,\alpha_i) := (\sigma_i(t,.), b_i(t,.),\alpha(t,l_\ast))$. Then the probability $P_{t}^{x_\ast,i_\ast,l_\ast}$ in the statement of Theorem \ref{th: exi Freidlin} above is nothing but the image law on $(\Phi, \B(\Phi))$ of $P^{x_\ast,i_\ast}$ under the measurable map from $(\mathcal{C}^{\mathcal{J}}[0,T], \B(\mathcal{C}^{\mathcal{J}}[0,T]))$ to $(\Phi, \B(\Phi))$ defined by
\[
\psi : X\mapsto \pare{(x_\ast,i_\ast,l_\ast){\bf 1}_{s\leq t} + (X(s), l_\ast + \ell_s^0(X)){\bf 1}_{s>t}}_{0\leq s\leq T}
\]
where $(\ell_s^0(X))_{0\leq s\leq T}$ stands for the local time process given by Lemma 2.2 in \cite{freidlinS} which is a functional of $X$ (by Lemma 2.2  eq. (2.7) in \cite{freidlinS}). We have $P_{t}^{x_\ast,i_\ast,l_\ast} = P^{x_\ast,i_\ast}\circ \psi^{-1}$: as such $P_{t}^{x_\ast,i_\ast,l_\ast}$ is uniquely determined and satisfies all requirements of Theorem \ref{th: exi Freidlin} above, whose statement is now adapted for our future purposes.
\end{Remark}
\begin{Remark}
\label{rem:importante-construction-canonique}
From \eqref{pte:fonctionnelle-additive}, note that the above construction and the properties of $(\ell_s)_{t\leq s\leq T}$ (Condition ($\mathcal{S}$-ii)) permit to assert that for all $j\in [I]$
\[\displaystyle\int_{t}^{s}\alpha_j(t,l_\ast)\partial_{x}f_{j}(u,0)d\ell_u = \displaystyle\int_{t}^{s}\alpha_j(t,l_\ast)\partial_{x}f_{j}(u,0)dl(u) = \displaystyle\int_{t}^{s}\alpha_j(t,l(t))\partial_{x}f_{j}(u,0)dl(u)\] holds $P_{t}^{x_\ast,i_\ast,l_\ast}$ a.s. So we may substitute this equality to write \eqref{eq:martingale} and \eqref{eq:martingale-mb} accordingly.
In the same manner, we could have written 
\begin{eqnarray}
\label{eq:tps-local-2}
\displaystyle\displaystyle\int_{t}^{T}\mathbf{1}_{\{x(u)>0\}}dl(u)=0,~~P_{t}^{x_\ast\,i_\ast,l_\ast} \text{ a.s.}
\end{eqnarray} 
instead of \eqref{eq:tps-local}
and written that for any $f\in\mathcal{C}^{1,2}_b(\mathcal{J}_T)$, the process defined by
$\forall s\in[t,T]:$
\begin{align}
\label{eq:martingale-2}
&~~f_{i(s)}(s,x(s)) - f_{i_\ast}(t,x_\ast)\nonumber\\
&\hspace{0.3 cm}-\displaystyle\int_{t}^{s}\pare{\partial_tf_{i(u)}(u,x(u))+\displaystyle\frac{1}{2}\sigma_{i(u)}^2(t,x(u),l_\ast)\partial_{xx}^2f_{i(u)}(u,x(u))}du\nonumber\\
&\hspace{0.3 cm}-\displaystyle\int_t^s \pare{b_{i(u)}(t,x(u),l_\ast)\partial_xf_{i(u)}(u,x(u))}du\nonumber\\
&\hspace{0.3 cm}-\displaystyle\sum_{j=1}^{I}\int_{t}^{s}\alpha_j(t,l_\ast)\partial_{x}f_{j}(u,0)dl(u),
\end{align}
is a $(\Psi_s)_{t \leq s \leq T}$ martingale after time $t$ under $P_t^{x_\ast,i_\ast,l_\ast}$, instead of ($\mathcal{S}$-iii).

\end{Remark}
\medskip
\medskip

\section{Main result}
\label{sec:main-results}

\subsection{Main result - The martingale problem \emph{$\big(\mathcal{S}_{pi}-\mathcal{M}_{ar}\big)$}
for a spider diffusion whose coefficients and spinning measure depend on its own local time}
\label{subsec: Main results}
\medskip
\medskip

We now state our main result, related to the existence and uniqueness of weak solutions for a class of spider diffusions with random selections depending on the own local time of the process at the junction point.

We define the following martingale problem of $(\Phi,\mathbb{B}(\Phi))$:
\medskip
\begin{center}
\emph{$\big(\mathcal{S}_{pi}-\mathcal{M}_{ar}\big)$}
\end{center}
\medskip
As usual denote $\tilde{X}=(x,i,l)$ the canonical process on $(\Phi,\mathbb{B}(\Phi))$.

{\it For $(x_\ast,i_\ast,l_\ast)\in \mathcal{J}\times \R^+$, find a probability $\P^{x_\ast,i_\ast,l_\ast}$ defined on the measurable space $(\Phi,\mathbb{B}(\Phi))$ such that:\\
-(i) $\tilde{X}(0)=\big(x_\ast,i_\ast,l_\ast\big)$, $\P^{x_\ast,i_\ast,l_\ast}$-a.s.\\ 
-(ii) For each $s\in [0,T]$: 
\begin{eqnarray*}
\displaystyle\displaystyle\int_{0}^{s}\mathbf{1}_{\{x(u)>0\}}dl(u)=0,~~\P^{x_\ast,i_\ast,l_\ast}-\text{ a.s.}
\end{eqnarray*}
and $(l(u))_{u\in [0,T]}$ has increasing paths $\P^{x_\ast,i_\ast,l_\ast}$-almost surely.\\
-(iii) For any $f\in\mathcal{C}^{1,2}_b(\mathcal{J}_T)$, the following process:
\begin{align}
\label{eq:def-V}
&\Bigg(~V^f(s):=f_{i(s)}(s,x(s))- f_{i_\ast}(0,x_\ast)\displaystyle\\
&\displaystyle -\int_{0}^{s}\pare{\partial_tf_{i(u)}(u,x(u))+\displaystyle\frac{1}{2}\sigma_{i(u)}^2(u,x(u),l(u))\partial_{xx}^2f_{i(u)}(u,x(u))}du\nonumber\\
&\displaystyle-\int_{0}^{s}\pare{b_{i(u)}(u,x(u),l(u))\partial_xf_{i(u)}(u,x(u))}du-\displaystyle\sum_{j=1}^{I}\int_{0}^{s}\alpha_j(u,l(u))\partial_{x}f_{j}(u,0)dl(u)~\Bigg)_{0\leq s\leq T}\nonumber,
\end{align} 
is a $(\Psi_s)_{0 \leq s \leq T}$ martingale under the measure of probability $\P^{x_\ast,i_\ast,l_\ast}$.}
\medskip

We now state the main result of this paper.
\begin{Theorem}\label{th: exist Spider }
Assume assumption $(\mathcal{H})$. Then, the martingale problem $\big(\mathcal{S}_{pi}-\mathcal{M}_{ar}\big)$ is well-posed.
\end{Theorem}

\medskip

\section{Concatenation of probability measures}\label{sec: concatenation}
This section is dedicated to obtain the main tools that will permit to the prove the existence part of Theorem \ref{th: exist Spider }.

\subsection{First technical results}
We begin our study by the following result.
\begin{Proposition}\label{pr : lower semi continuity}
The following map:
\begin{align}
\label{map:rho}
\rho:\begin{cases}
\mathcal{C}^{\mathcal{J}}[0,T]\times \mathcal{L}[0,T]&\to\mathbb{R}\\
\Big(\big (x,i),\hspace{0,3 cm}l\Big)&\mapsto \displaystyle\int_{0}^T\mathbf{1}_{\{t\in [0,T]\vert x(t)>0\}}(u)dl(u)
\end{cases}
\end{align}
is lower semi continuous.
\end{Proposition}
\begin{proof}
Let $\Big(\big(x^n,i^n\big),l^n\Big)$ in $\mathcal{C}^{\mathcal{J}}[0,T]\times \mathcal{L}[0,T]$
converging to $\Big(\big(x,i\big),l\Big)$.
The uniform convergence of $(x^n)$ to $x$ implies that the sequence $(x^n)$ is equicontinuous on $\mathcal{C}[0,T]$ (converse of Ascoli-Arzela's theorem) and it is easy to check the lower semicontinuity of the real map $[0,T]\rightarrow \R$, $u\mapsto \mathbf{1}_{\{t\in [0,T]\vert x(t)>0\}}(u)$. Hence, the sequence of real maps $\pare{f_n: [0,T]\rightarrow \R, u\mapsto \mathbf{1}_{\{t\in [0,T]~\vert x^n(t)>0\}}(u)}$ is lower semiequicontinuous.
On the other hand,
$$\forall h\in \mathcal{C}[0,T],~~\Big|\int_0^Th(t)dl^n(t)-\int_0^Th(t)dl(t)\Big|\leq 2|h|_{(0,T)}\,\,|l^n-l|_{(0,T)},$$
and $dl^n \overset{*}{\rightharpoonup} dl$ in $*\sigma\Big(\mathcal{C}[0,T]^{'},\mathcal{C}[0,T]\Big)$. The uniform asymptotic integrability w.r.t the sequence $(dl^n)$ of the sequence of indicators $(f_n)$ is straightforward and we are in position to use Fatou's Lemma for weakly convergent sequences of measures (recalled in Corollary \ref{cr: Fatou}), which proves that
$$\int_0^T\mathbf{1}_{\{t\in [0,T]\vert x(t)>0\}}(u)dl(u)\leq \liminf\limits_{n\rightarrow +\infty}\int_0^T\mathbf{1}_{\{t\in [0,T]\vert x^n(t)>0\}}(u)dl^n(u),$$
and completes the proof.
\end{proof}

Let us now state a very useful result.
\begin{Lemma}
\label{lem:non-stickness}
Fix $(t,(x_\ast,i_\ast),l_\ast) \in [0,T]\times \mathcal{J}\times [0,+\infty)$ and let $(t_n,(x_n,i_n),l_n)$ denote a sequence converging to $(t,(x_\ast,i_\ast),l_\ast)$ such that $\sup_{n\geq 0}(|x_n - x_\ast| + |l_n - l_\ast|)\leq 1$.

Then, there exists a constant $C>0$, depending only on the data $\Big(T,|b|,|\sigma|,c\Big)$ introduced in assumption $(\mathcal{H})$ such that:
\begin{eqnarray}\label{eq:majorenzero generale}  \forall \varepsilon>0,~~\sup\limits_{n\geq 0}{E^{x_n,i_n,l_n}_{t_n}}\croc{\int_{0}^{T}\mathbf{1}_{\{x(s)<\varepsilon\}}ds}\leq C\varepsilon.
\end{eqnarray}
\end{Lemma}
\begin{proof}
\noindent

\textbf{Step 1} Let $M>0$. From equation \eqref{eq:equation-freidlin-sheu}, it is standard to prove (combining It\^o's formula applied to $\phi:[0,+\infty)\to\;\R$, 
$x\mapsto\;\exp(Mx) - Mx - 1$ and Gr\"onwall's lemma - that there exists $C>0$, depending only on $M$ and the data $(T,|b|,|\sigma|)$ introduced in assumption $(\mathcal{H})$, such that:
\begin{eqnarray}
\sup_{n\geq 0}{E^{x_n,i_n,l_n}_{t_n}}\Big[\exp(Mx(T))\Big]~~\leq~~C.
\end{eqnarray}

\textbf{Step 2} Fix $n\in \N$. Let $\varepsilon>0$, and $\beta^\varepsilon \in \mathcal{C}([0,+\infty),\R_+)$ satisfying:
\begin{eqnarray}\label{eq : fonct test estim temps 0}
\forall x\ge2\varepsilon,~~\beta_\varepsilon(x)=0\,\text{ and }\, \forall x\ge 0,~\mathbf{1}_{\{x<\varepsilon\}}\leq\beta_\varepsilon(x)\leq 1.
\end{eqnarray}
We define $u^\varepsilon\in \mathcal{C}^{2}([0,+\infty))$ as the unique solution of the following ordinary second order differential equation
\begin{eqnarray}\label{eq :pde parab estima}\begin{cases}\displaystyle\partial_{x}^2u^\varepsilon(x)-M\partial_xu^\varepsilon(x)=2\beta^\varepsilon(x)/\underline{\sigma}^2,~~ \text{ if }  x\in(0,+\infty),\\
\partial_xu^\varepsilon(0)=0,\\
u^\varepsilon(0)=0,\\
\end{cases}
\end{eqnarray}
where $\underline{\sigma}$ is the constant of ellipticity defined in assumption $(\mathcal{H})-({\bf E})$, and $M$ is chosen by setting
\begin{eqnarray*}
M=\displaystyle\frac{2|b|}{\underline{\sigma}^2}.
\end{eqnarray*}
The solution is:
\begin{eqnarray*}
u^\varepsilon(x)=\int_0^x\exp(Mz)\int_0^z\frac{2\beta^\varepsilon(u)}{\underline{\sigma}^2}\exp(-Mu)dudz.
\end{eqnarray*}
By the assumption on $\beta_\varepsilon$ and assumption $(\mathcal{H})$, we get: 
\begin{eqnarray}\label{eq : conditions0 fonction test}
\forall x\ge 0,~~0\leq\partial_xu^\varepsilon(x)\leq 4\varepsilon/\underline{\sigma}^2\exp(Mx),~~0\leq u^\varepsilon(x)\leq \frac{4\varepsilon}{M\underline{\sigma}^2} (\exp(Mx)-1).
\end{eqnarray}
Hence using the characterization condition $(\mathcal{S}-iii)$ (with $f=u^\varepsilon $ and after a localization argument), we get using \eqref{eq : fonct test estim temps 0}, \eqref{eq :pde parab estima} and \eqref{eq : conditions0 fonction test}:
\begin{equation*}
\begin{split}
&{E^{x_n,i_n,l_n}_{t_n}}\Big[u^\varepsilon(x(T))- u^\varepsilon(x_n)\Big]\\
&=
{E^{x_n,i_n,l_n}_{t_n}}\Big[\int_{t_n}^T\Big(\displaystyle\frac{1}{2}\sigma_{i(u)}^2(t,x(u))\partial_{x}^2u^\varepsilon(x(u))+ b_{i(u)}(t,x(u)))\partial_xu^\varepsilon(x(u))\Big)du\Big]\\
&={E^{x_n,i_n,l_n}_{t_n}}\Big[\int_t^T\displaystyle\frac{1}{2}\sigma_{i(u)}^2(t,x(u))\Big(\partial_{x}^2u^\varepsilon(x(u))+\displaystyle\frac{b_{i(u)}(t,x(u)))}{\frac{1}{2}\sigma_{i(u)}^2(t,x(u))}\partial_xu^\varepsilon(x(u))\Big)du\Big ]\\
&\ge
{E^{x_n,i_n,l_n}_{t_n}}\Big[\int_t^T\displaystyle\frac{1}{2}\sigma_{i(u)}^2(t,x(u))\Big(\partial_{x}^2u^\varepsilon(x(u))-M\partial_xu^\varepsilon(x(u))\Big)du\Big]\\
&\ge {E^{x_n,i_n,l_n}_{t_n}}\Big[\int_t^T\displaystyle\frac{1}{2}\underline{\sigma}^2\Big(2\beta^\varepsilon(x(u))/\underline{\sigma}^2)\Big)du\Big]\\
&\ge {E^{x_n,i_n,l_n}_{t_n}}\Big[\displaystyle\int_{t}^{T}\beta^\varepsilon(x(u))du\Big]\,\ge {E^{x_n,i_n,l_n}_{t_n}}\Big[\int_{0}^{T}\mathbf{1}_{\{x(u)\leq \varepsilon\}}du\Big].
\end{split}
\end{equation*}
Hence, using \eqref{eq : conditions0 fonction test} we get :
\begin{eqnarray*} {E^{x_n,i_n,l_n}_{t_n}}\Big [\int_{t}^{T}\mathbf{1}_{\{x(s)\leq \varepsilon\}}ds \Big] \leq \frac{4\varepsilon}{M\underline{\sigma}^2}{E^{x_n,i_n,l_n}_{t_n}}\Big[\exp(Mx(T))-1\Big]. \end{eqnarray*}
The conclusion follows using \textbf{Step 1}.
\end{proof}

\begin{Proposition}
\label{pr: conti noyau}
The following map:
\begin{align*}[0,T]\times \mathcal{J}\times [0,+\infty) &\to \mathcal{P}\pare{\Phi, \mathbb{B}(\Phi)}\\
(t,(x_\ast,i_\ast),l_\ast)&\mapsto P_t^{x_\ast,i_\ast,l_\ast} \end{align*}
is continuous for the weak topology on $\mathcal{P}(\Phi,\mathbb{B}(\Phi))$.
\end{Proposition}

The proof of Proposition \ref{pr: conti noyau} relies on the following  technical lemma:
\begin{Lemma}
\label{lem:modules-cont}
Fix $(t,(x_\ast,i_\ast),l_\ast) \in [0,T]\times \mathcal{J}\times [0,+\infty)$ and let $(t_n,(x_n,i_n),l_n)$ denote a sequence converging to $(t,(x_\ast,i_\ast),l_\ast)$ such that $\sup_{n\geq 0}(|x_n - x_\ast| + |l_n - l_\ast|)\leq 1$.
Then, 
\begin{itemize}[label = $-$]
\item For all integer $m$, there exists a constant $C_m$ depending only on $m$ and the data $(T,|b|,|\sigma|)$ such that
\begin{align}
\label{eq:majoration-moments-ordre-deux}
&\sup_{n\ge 0}~~{E^{x_n,i_n,l_n}_{t_n}}\croc{|x|^{2m}_{(0,T)}+|l|^{2m}_{(0,T)}+ |\ell|^{2m}_{(0,T)}}~~\leq~~C_m(1+x_\ast^{2m}+l_\ast^{2m});
\end{align}
\item There exists a constant $C$, depending only on the data $(T,|b|,|\sigma|, \underline{\sigma})$, such that:
\begin{align}
\label{eq:modulus-continuity-x-l}
&\forall \theta\in (0,1),~~\sup_{n\ge 0}E^{x_n,i_n,l_n}_{t_n}\croc{\omega\pare{x,\theta}^2 +\omega(l,\theta)^{2} + \omega(\ell,\theta)^{2}}~~\leq~~C\theta \ln\pare{\displaystyle{2T}/{\theta}},
\end{align}
\begin{align}
\label{eq:modulus-continuity-X}
&\forall \theta\in (0,1),\;\;\;\;\sup_{n\ge 0}{E^{x_n,i_n,l_n}_{t_n}}\Big [\tilde{\omega}(\tilde{X},\theta)\Big ]\leq C\sqrt{\theta\ln(2T/\theta)}.
\end{align}
\end{itemize}
\end{Lemma}

\begin{Remark}
Note that the result of Lemma \ref{lem:modules-cont} is by all means similar to known results that hold for real valued It\^o processes (see for e.g. \cite{Fischer}). Our problems come from the fact that the process $(x(t),i(t))$ lives on a star-shaped network and also from the presence of the additional increasing process $(l(t))$ in the parameters. However, these difficulties may be overcome by relying on \eqref{eq:equation-freidlin-sheu}; in order not to overload our exposition we have decided to postpone the proof of Lemma \ref{lem:modules-cont} in Appendix. 
\end{Remark}
\medskip

We are now in position to prove Proposition \ref{pr: conti noyau}.
\begin{proof} {\it (of Proposition \ref{pr: conti noyau} using Lemma \ref{lem:modules-cont}).}

Fix $(t,(x_\ast,i_\ast),l_\ast) \in [0,T]\times \mathcal{J}\times [0,+\infty)$ and let $(t_n,(x_n,i_n),l_n)$ denote a sequence converging to $(t,(x_\ast,i_\ast),l_\ast)$.

Denote for a moment $P^n := P_{t_n}^{x_n,i_n,l_n}$ $(n\in \N)$ to simplify the notations and $E^n$ its associated expectation.

From the estimates of Lemma \ref{lem:modules-cont} using Ascoli-Arzela's theorem and as a consequence of Prokhorov's theorem, we see that the sequence $(P^n)$ is weakly relatively compact  and converges (up to a subsequence) to a probability $P\in \mathcal{P}\pare{\Phi,\B(\Phi)}$. Indeed, the formal proof of this essential fact uses exactly the same arguments as the proof of the forthcoming Corollary \ref{corollary-relat-compactness} (see also Remark \ref{rem:tension-spider}). since the statement of Corollary \ref{corollary-relat-compactness} is more central for the proof of our main theorem (Theorem \ref{th: exist Spider }), for the convenience of the reader and for conciseness, we decide not to repeat it here.
\vspace{0,3 cm}

Let us now turn to show that $P$ satisfies the conditions $(\mathcal{S})$ stated in Theorem \ref{th: exi Freidlin} that characterize the probability measure in a unique way. 

We start with condition i) of $(\mathcal{S})$. Let $A:=\{\tilde{X}_u = \pare{x_\ast,i_\ast,l_\ast},\,\forall u\leq t\}$ and for any $n\in \N$,
$A_n:=\{\tilde{X}_u = \pare{x_n,i_n,l_n},\,\forall u\leq t_n\}$.
Observe that, from the definition of $P^n$, $P^n(A_n) = 1$ (for any $n\in\N$).

Fix $\varepsilon\in (0,1)$. Since $\pi(P^n, P)$ tends to $0$ and $(t_n,(x_n,i_n),l_n)$ converges to $(t,(x_\ast,i_\ast),l_\ast)$, there exists $N_{\varepsilon}\in \N$ \, s.t. for any $n\geq N_{\varepsilon}$, $|t_n-t_\ast| + d^{\mathcal{J}}((x_n,i_n),(x_\ast,i_\ast)) + |l_n-l_\ast| <\varepsilon$ and for any set $H\in \mathbb{B}(\Phi)$, $P\pare{H^{\varepsilon}}\geq P^{n}\pare{H} - \varepsilon$.

Let $n\geq N_{\varepsilon}$. Then,\\
a) either $t_n\geq t$ and in this case $A_n\subset A^{\varepsilon}$ so that (taking $H = A^{\varepsilon}$) we have
$$P((A^\varepsilon)^{\varepsilon})\geq P^n(A^\varepsilon) - \varepsilon \geq P^n(A_n) - \varepsilon = 1 - \varepsilon\;;$$
b) or $t_n\leq t$ and in this case things are a bit more intricate since we do not have $A_n\subset A^{\varepsilon}$ anymore. Still, let $\delta>>\sqrt{\varepsilon}$ a parameter to be fixed later and let $$B_n^{\delta}:=\{\sup_{s\in [t_n,t]}d^{\mathcal{J}}((x(s),i(s)), (x_n,i_n))>\delta/4\}.$$ From the result of Lemma \eqref{lem:modules-cont} and Markov's inequality, we ensure that there exists some constant $C>0$ s.t.
$\ds P^n(B_n^{\delta})\leq \frac{C(\varepsilon\ln(1/\varepsilon))^{1/2}}{\delta}$. Observe that $A_n\cap (B_n^{\delta})^c\subset A^{\delta}$ (on the set $(B_n^{\delta})^c$ the canonical process cannot have increased by more than $\delta/4$ units during the time interval $[t_n,t]$). Thus, we have 
\begin{align*}
P((A^\delta)^{\varepsilon})&\geq P^n(A^\delta) - \varepsilon\geq P^n(A_n \cap (B_n^{\delta})^c) - \varepsilon\\
&\geq P^n(A_n) -  P^n(A_n \cap B_n^{\delta}) - \varepsilon\geq P^n(A_n) -  P^n(B_n^{\delta}) - \varepsilon\\
&\geq 1 - \frac{C(\varepsilon \ln(1/\varepsilon))^{1/2}}{\delta}  - \varepsilon.
\end{align*}
Choose now $\delta = {\varepsilon}^{1/4}$. The previous inequality implies
\begin{align*}
P\pare{\pare{A^{\varepsilon^{1/4}}}^{\varepsilon}}\geq 1 - O({\varepsilon}^{1/4}\ln(1/\varepsilon)^{1/2}).
\end{align*}

Hence, for any $\varepsilon>0$, 
$P((A^\varepsilon)^{\varepsilon})\geq 1-\varepsilon$ or $P\pare{\pare{A^{\varepsilon^{1/4}}}^{\varepsilon}}\geq 1 - O({\varepsilon}^{1/4}\ln(1/\varepsilon)^{1/2})$. From the previous a) and b) we have either
$\lim_{\varepsilon \searrow 0}P((A^{\varepsilon})^{\varepsilon}) = 1$ or $\lim_{\varepsilon \searrow 0} P\pare{\pare{A^{\varepsilon^{1/4}}}^{\varepsilon}} = 1$. But the dilation sets $(A^\varepsilon)^{\varepsilon}$ and $\pare{A^{\varepsilon^{1/4}}}^{\varepsilon}$ decrease both to $A$ as $\varepsilon$ tends to zero and the monotone convergence theorem ensures that $P(A) = 1$ s.t. $P$ satisfies condition $(\mathcal{S})$ i).

Let us now turn to the proof of condition $(\mathcal{S})$ (ii) for $P$. From the result of Proposition \ref{pr : lower semi continuity}, the map $\rho$ defined in \eqref{map:rho} is lower semi continuous.
Consequently, $O :=\rho^{-1}\pare{(0,\infty)}$ is open in $\mathcal{C}^{\mathcal{J}}[0,T]\times \mathcal{L}[0,T]$ and from the weak convergence of $(P_n)$ to $P$ we have
$$ 0=\liminf_{n \to +\infty}P^n(O)\ge P(O)$$
yielding precisely that $P$ satisfies condition (ii) of $(\mathcal{S})$.

We finish by proving that $P$ satisfies condition (iii) of $(\mathcal{S})$. 

By the result of Lemma \ref{lem:non-stickness}, we know that there exists a constant $C>0$, depending only on the data $\Big(T,|b|,|\sigma|,\underline{\sigma}\Big)$ (see assumption $(\mathcal{H})$) and $x_\ast$, such that
\begin{align*}\forall \varepsilon>0,~~ \sup_n {E^n}\croc{\int_{t_{n}}^{T}\mathbf{1}_{\{x(s)<\varepsilon\}}ds}\leq C\varepsilon.
\end{align*}
Fix $\varepsilon>0$. 
Note that the following functional $\beta_\varepsilon:\mathcal{C}[0,T]\to \R,~~x\mapsto \displaystyle \int_t^T\mathbf{1}_{\{x(s)<\varepsilon\}}ds$ is lower semi continuous so that
from the weak convergence of $P_n$ to $P$
\[
E\Big[\int_{t}^{T}\mathbf{1}_{\{x(s)<\varepsilon\}}ds\Big]\leq\liminf_{n \to +\infty}E^n\Big[\int_{t}^{T}\mathbf{1}_{\{x(s)<\varepsilon\}}ds\Big].
\]

Writing
\begin{align*}E^n\Big[\int_{t_{n}}^{T}\mathbf{1}_{\{x(s)<\varepsilon\}}ds\Big]~~=~~E^n\Big[\int_{t}^{T}\mathbf{1}_{\{x(s)<\varepsilon\}}+\int_{t_{n}}^{t}\mathbf{1}_{\{x(s)<\varepsilon\}}ds\Big],
\end{align*}
gives easily that
\begin{align*}
E\Big[\int_{t}^{T}\mathbf{1}_{\{x(s)<\varepsilon\}}ds\Big]\leq\liminf_{n \to +\infty}E^n\Big[\int_{t_{n}}^{T}\mathbf{1}_{\{x(s)<\varepsilon\}}ds~~\Big]~~\leq~~(C+1)\varepsilon.
\end{align*}
Making $\varepsilon$ tend to $0$ and using Lebesgue's convergence theorem ensures
\begin{equation}
\label{eq:zero-lebeque-measure}
E\Big[\int_t^T\mathbf{1}_{\{x(s)=0\}}ds\Big]=0,~~\text{and}~~ \int_t^T\mathbf{1}_{\{x(s)=0\}}ds,~~P-\text{a.s.}
\end{equation}

\vspace{0,3 cm}
Let $f\in\mathcal{C}^{1,2}_b(\mathcal{J}_T)$. Because of the continuity condition at the junction point $0$ for $(f_i)_{i\in [I]}$ and since $(t_n,i_n)$ converges to $(t,i_\ast)$, the sequence $(f_{i_n}(t_{n},x(t_{n})))$ admits the limit
$$\lim\limits_{n\rightarrow + \infty} f_{i_n}(t_{n},x(t_{n})) = f_{i_\ast}(t,x(t))$$ even in the case where $x(t) = 0$.

For any fixed $n\in \N$, we define for $s\in [t_n, T]$:
\begin{align*}
M^f_n(s):=&\;\;\;f_{i(s)}(s,x(s))- f_{i_n}(t_{n},x(t_{n}))\\
&-\displaystyle\int_{t_{n}}^{s}\pare{\partial_tf_{i(u)}(u,x(u))+\displaystyle\frac{1}{2}\sigma_{i(u)}^2(t_n,x(u),l_n)\partial_{xx}^2f_{i(u)}(u,x(u))}du\\
& - \int_{t_{n}}^{s} \pare{b_{i(u)}(t_n,x(u),l_n)\partial_xf_{i(u)}(u,x(u))}du\\
&-\displaystyle\sum_{j=1}^{I}\int_{t_n}^{s}\alpha_{j}(t_{n},l_{n})\partial_{x}f_{j}(u,0)d\ell_u.
\end{align*}

Fix $s\in (t, T]$. Observe from \eqref{eq:zero-lebeque-measure} that the equality
\begin{align*}
M^f_n(s)=&\;\;\;f_{i(s)}(s,x(s))- f_{i_n}(t_{n},x(t_{n}))\\
&-\displaystyle\int_{t_{n}}^{s}\pare{\partial_tf_{i(u)}(u,x(u))+\displaystyle\frac{1}{2}\sigma_{i(u)}^2(t_n,x(u),l_n)\partial_{xx}^2f_{i(u)}(u,x(u))}\mathbf{1}_{\{x(u)\neq 0\}}du\\
& - \int_{t_{n}}^{s} \pare{b_{i(u)}(t_n,x(u),l_n)\partial_xf_{i(u)}(u,x(u))}\mathbf{1}_{\{x(u)\neq 0\}}du\\
&-\displaystyle\sum_{j=1}^{I}\int_{t_n}^{s}\alpha_{j}(t_{n},l_{n})\partial_{x}f_{j}(u,0)d\ell_u
\end{align*}
holds almost surely under $P$. Recall also that $P$ satisfies condition $(\mathcal{S})$ i), so that $P(x(t) = x_\ast) = 1$. Since $(t_n,i_n,l_n)$ converges to $(t,i_\ast,l_\ast)$, our assumptions $({\mathcal H})$ allow to apply Lebesgue's convergence theorem and we get that the sequence
$(M^{f}_n(s))$ converges $P$ a.s. to 
\begin{align*}
M^f(s)&:=f_{i(s)}(s,x(s))- f_{i_\ast}(t,x_\ast)\\
&\hspace{0,3 cm}-\displaystyle\int_{t}^{s}\pare{\partial_tf_{i(u)}(u,x(u))+\displaystyle\frac{1}{2}\sigma_{i(u)}^2(t,x(u),l_\ast)\partial_{xx}^2f_{i(u)}(u,x(u))}du\\
&\hspace{0,3 cm}-\int_{t}^{s} \pare{b_{i(u)}(t,x(u),l_\ast)\partial_xf_{i(u)}(u,x(u))}du~~
\\&\hspace{0,3 cm}-\displaystyle\sum_{j=1}^{I}\int_{t}^{s}\alpha_j(t,l_\ast)\partial_{x}f_{j}(u,0)d\ell_u.
\end{align*}
Now let $\varphi_s\in \mathcal{C}_b(\Phi,\R)$ $\Psi_s$ measurable and fix $u\in [s,T]$.

On the one hand, using Lebesgue's theorem for weakly convergence measures (see Corollary \ref{cr Lebesgue faible}), we get that
\begin{align*}
\lim_{n \to +\infty}{E^n}\Big[\varphi_s (M^{f}_n(u)-M^{f}_n(s))\Big]~~=~~{E}\Big[\varphi_s(M^{f}(u)-M^{f}(s))\Big].
\end{align*}
On the other hand, using the martingale property $({\mathcal{S}})$ iii) under the measure $P^n$, we see that
\begin{align*}
&\Big|{E^n} \Big[\varphi_s(M^{f}_n(u)-M^{f}_n(s))\Big]\Big|\\
&=\Big|E^n \Big[\varphi_s(M^{f}_n(u)-M^{f}_n(s))\Big]\mathbf{1}_{\{s\ge t_{n}\}}+E^n \Big[\varphi_s(M^{f}_n(u)-M^{f}_n(s))\Big]\mathbf{1}_{\{t<s\leq t_{n}\}}\Big|\\
&=\Big|E^n \Big[\varphi_s(M^{f}_n(u)-M^{f}_n(s))\Big]\Big|\mathbf{1}_{\{t<s\leq t_{n}\}}.
\end{align*}
It is easy to see that our assumptions $(\mathcal{H})$ imply $\sup\limits_{n\in \N}\Big|E^n \croc{\varphi_s(M^{f}_n(u)-M^{f}_n(s))}\Big |<+\infty$, and since time $s$ has been fixed so that $t<s$, the indicator $\mathbf{1}_{\{t<s\leq t_{n}\}}$ is easily seen to converge to $0$ as $n$ tends to $+\infty$. 
Gathering both facts ensures 
\[E\Big[\varphi_s(M^{f}(u)-M^{f}(s))\Big]=0\]
which holds true for any $s>t$ and $u\in [s,T]$ and $\varphi_s\in \Psi_s$.
Thus, we deduce $E[M^{f}(u)|\Psi_s] = M^{f}(s)$ for any $s>t$ and $u\in [s,T]$.

For $s=t$, we let $\varepsilon >0$ and write that for any $u\in [t+\varepsilon,T]$
\[E[M^{f}(u)|\Psi_t]={E}[E[M^{f}(u)|\Psi_{t+\varepsilon}]|\Psi_t]=E[M^{f}(t+\varepsilon)|\Psi_t]\]
and in this particular case we retrieve the martingale property by using Lebesgue's dominated convergence theorem for conditional expectations when letting $\varepsilon$ tend to zero.

Finally, using the condition i) of $\mathcal{S}$ at time $t$, namely that $P(\tilde{X}_t = ((x_\ast,i_\ast),l_\ast)) = 1$, we conclude that the process:
\begin{align*}
&\Bigg(~f_{i(s)}(s,x(s))- f_{i_\ast}(t,x(t))~\displaystyle-\int_{t}^{s}\pare{\partial_tf_{i(u)}(u,x(u))+\displaystyle\frac{1}{2}\sigma_{i(u)}^2(t,x(u),\ell(t))\partial_{xx}^2f_{i(u)}(u,x(u))}du\\
&\displaystyle-\int_{t}^{s}\pare{b_{i(u)}(t,x(u),\ell(t))\partial_xf_{i(u)}(u,x(u))}du~-\displaystyle\sum_{j=1}^{I}\int_{t}^{s}\alpha_j(t,\ell(t))\partial_{x}f_{j}(u,0)d\ell_u~\Bigg)_{t\leq s\leq T},
\end{align*} 
is a continuous $\pare{P, (\Psi_s)_{t\leq s\leq T}}$ martingale after time $t$. 

Applying Theorem \ref{th: exi Freidlin}, we are allowed to identify the limit $P$ to $P^{x_\ast,i_\ast,l_\ast}_t$.

We conclude that the map $(t,(x_\ast,i_\ast),l_\ast)\to P^{x_\ast,i_\ast,l_\ast}_t$ is continuous for the weak topology on $\mathcal{P}(\Phi,\mathbb{B}(\Phi))$.
\end{proof}

\subsection{One step concatenation of probability measures}
Let us first recall the fundamentals for the construction of concatenated probability measures as presented in the seminal book by Stroock and Varadhan \cite{Stroock}, which is our main tool for our construction. The proof of the following statements may be directly adapted from the proofs of Lemma 6.1.1 and Theorem 6.1.2 in \cite{Stroock}: we replace $(\Omega, \mathcal{B}(\Omega))$ by $(\Phi, \B(\Phi))$ and $\mathcal{C}([0,s],\R^d)$ and $\mathcal{C}([s,\infty),\R^d)$ by $\mathcal{C}([0,s],\mathcal{J}\times \R^+)$ and $\mathcal{C}([s,T],\mathcal{J}\times \R^+)$ accordingly.

\begin{Proposition}\label{pr: exist proba conca}
(Concatenation of probability measures - from Lemma 6.1.1 and Theorem 6.1.2 in \cite{Stroock})

\noindent a) 
Let $s\in [0,T)$. Let $(Q^{Y})_{Y \in \Phi}$ a transition probability kernel from $(\Phi,\Psi_{s})$ to $(\Phi,\Psi_{T})$ (this means that the mapping $Y\mapsto Q^Y(A)$ is $\Psi_s$-measurable for all $A\in \Psi_T$) that satisfies:
$$\forall {Y}\in \Phi,~~Q^{{Y}}\Big(\tilde{X}_{s}(\cdot)=\tilde{X}_{s}({Y}) = Y_s\Big)=1.$$

There exists a unique transition probability kernel from $(\Phi,\Psi_{s})$ to $(\Phi,\Psi_{T})$ denoted by $$(\Pi_Y\otimes_{s}Q^Y)_{Y\in \Phi}$$ such that:

$\forall Y\in \Phi$,
\begin{align}
\label{def-Pi-Proba}
&\Pi_Y\otimes_{s}Q^Y\Big(\tilde{X}_t(\cdot)=\tilde{X}_{t}(Y) = Y_t,\;\forall t\in [0,s]\Big)=1,\\
\label{def-Pi-Proba-2}
&\forall A\in \sigma(\tilde{X}_{(s+t)\wedge T},s\geq 0),~~\Pi_Y\otimes_{s}Q^Y(A)=Q^Y(A).
\end{align}
b) Moreover, let $\tau$ a $(\Psi_{s})_{t\leq s \leq T}$ stopping time satisfying that $\tau(Y)\leq T$ for all $Y\in \Phi$ and let $(Q^{Y})_{Y \in \Phi}$ a transition probability kernel from $(\Phi,\Psi_{\tau})$ to $(\Phi,\Psi_{T})$ (the mapping $Y\mapsto Q^Y(A)$ is $\Psi_\tau$-measurable for all $A\in \Psi_T$) satisfying
$$\forall {Y}\in \Phi,\;\;Q^{{Y}}\Big(\tilde{X}_{\tau(Y)}(.)=\tilde{X}_{\tau(Y)}({Y}) = Y_{\tau(Y)}\Big)=1.$$

To $P\in {\mathcal P}\pare{\Phi, {\mathbb B}(\Phi)}$ we associate a unique probability measure on $\pare{\Phi, {\mathbb B}(\Phi)}$ denoted by $P\otimes_\tau Q$ satisfying that \\
(i) the restriction of $P\otimes_\tau Q$ with respect to $\Psi_{\tau}$ is equal to $P$;\\
(ii) a r.c.p.d (regular conditional probability distribution) of $P\otimes_\tau Q$  with respect to $\Psi_{\tau}$ is equal to $(\Pi_Y\otimes_{\tau(Y)}Q^Y)_{Y\in \Phi}$: for all
$F\in \B(\Phi)$,
\begin{equation}
\label{eq:prop-rcpd}
P\otimes_\tau Q\,(F\;|\;\Psi_\tau)(Y) = \Pi_Y\otimes_{\tau(Y)}Q^Y(F)
,
\end{equation}
$P\otimes_\tau Q$ almost surely (w.r.t\;\;$Y$).
\end{Proposition}
\medskip

We are now going to make use of Proposition \ref{pr: exist proba conca} for the construction of our spider diffusion. The basic idea is to concatenate probability measures that come from Theorem \ref{th: exi Freidlin}. This is possible thanks to the result of Proposition \ref{pr: conti noyau} that ensures that the corresponding transition probability kernels are measurable. With this in mind, we need the following ingredients.
\medskip

Fix $(x_\ast,i_\ast,l_\ast)\in \mathcal{J}\times \R^+$ and let $P^{\ast}:=P_0^{x_\ast,i_\ast,l_\ast}$ stand for the unique probability measure on $(\Phi,\mathbb{B}(\Phi))$ constructed in Theorem \ref{th: exi Freidlin}. The conditions $(\mathcal{S})$ read in this case:\\
-(i) $\tilde{X}_0 = \big(x(0),i(0),l(0)\big) =(x_\ast,i_\ast,l_\ast)$, $P^{\ast}$-a.s. (where as usual with our convention, $\tilde{X} = (x,i,l)$ stands for the canonical process of $(\Phi,\mathbb{B}(\Phi))$)\\ 
-(ii)
\begin{eqnarray*}
\displaystyle\displaystyle\int_{0}^{T}\mathbf{1}_{\{x(u)>0\}}d\ell_u=0,~~P^{\ast}-\text{ a.s.} \;\;\;\big (\textrm{or equivalently } \int_{0}^{T}\mathbf{1}_{\{x(u)>0\}}dl(u)=0,~~P^{\ast}-\text{ a.s.}\\
\textrm{see Remark \ref{rem:importante-construction-canonique}}\big ).
\end{eqnarray*}
-(iii) For any $f\in\mathcal{C}^{1,2}_b(\mathcal{J}_T)$, the following process: 
\begin{align}
\label{eq:def-Mftilde}
&\Bigg (\tilde{M}^f(s):=f_{i(s)}(s,x(s)) - f_{i_\ast}(0,x_\ast)\nonumber\\
&\hspace{0.3 cm}-\displaystyle\int_{0}^{s}\pare{\partial_tf_{i(u)}(u,x(u))+\displaystyle\frac{1}{2}\sigma_{i(u)}^2(0,x(u),0)\partial_{xx}^2f_{i(u)}(u,x(u))}du\nonumber\\
&\hspace{0.3 cm}-\displaystyle\int_0^s \pare{b_{i(u)}(0,x(u),0)\partial_xf_{i(u)}(u,x(u))}du\nonumber\\
&\hspace{0.3 cm}-\displaystyle\sum_{i=1}^{I}\int_{0}^{s}\alpha_j(0,0)\partial_{x}f_{j}(u,0)dl(u)\bigg)_{0\leq s\leq T}
\end{align}
is a $\pare{P^{\ast}, (\Psi_s)_{0 \leq s \leq T}}$ martingale.  
Note that we made use of Remark \ref{rem:importante-construction-canonique} in order to replace the local time by the increasing component $l$ of the canonical process).
\medskip

Let $\tau$ stand for a $(\Psi_{s})_{t\leq s \leq T}$ stopping time satisfying that $\tau(Y)\leq T$ for any $Y\in \Phi$. 
In the sequel, we will use the notation $Y:=\Big(\big(\hat{x}(\cdot),\hat{i}(\cdot)\big),\hat{\ell}(\cdot)\Big) \in \Phi$.

Once again applying Theorem \ref{th: exi Freidlin},  we may define the following family of probability measures on $\pare{\Phi, {\mathbb B}(\Phi)}$
\begin{equation}
\label{defeq: transition-probability-kernel}
\Big(Q^Y:=P_{\tau(Y)}^{\hat{x}(\tau(Y)),\hat{i}(\tau(Y)),\hat{l}(\tau(Y))}\Big)_{Y:=\pare{\hat{x}(\cdot),\hat{i}(\cdot),\hat{l}(\cdot)} \in \Phi}
\end{equation}
satisfying the following conditions (i)-(ii)-(iii) that hold for any $Y\in \Phi$:\\
-(i) For each $s\leq \tau(Y)$, $\tilde{X}_s=\Big(\hat{x}(\tau(Y)),\hat{i}(\tau(Y)),\hat{l}(\tau(Y))\Big)$,\;$Q^Y$-a.s. \\ 
-(ii) 
\begin{eqnarray}
\label{eq:zero-acc-tl}
\displaystyle
\int_{\tau(Y)}^{T}\mathbf{1}_{\{x(u)>0\}}dl(u)=0,~~Q^{Y}-\text{ a.s.}\;\;\;\big (\textrm{see Remark \ref{rem:importante-construction-canonique}}\big )\nonumber.
\end{eqnarray}
-(iii) For any $f\in\mathcal{C}^{1,2}_b(\mathcal{J}_T)$, the following process: 
 \begin{align}
\label{eq:def-hatM}
&\Bigg (\hat{M}^f(s):= f_{i(s)}(s,x(s)) - f_{\hat{i}(\tau(Y))}(\tau(Y),\hat{x}(\tau(Y))\nonumber\\
&\hspace{0.3 cm}-\displaystyle\int_{\tau(Y)}^{s}\pare{\partial_tf_{i(u)}(u,x(u))+\displaystyle\frac{1}{2}\sigma_{i(u)}^2(\tau(Y),x(u),\hat{l}(\tau(Y)))\partial_{xx}^2f_{i(u)}(u,x(u))}du\nonumber\\
&\hspace{0.3 cm}-\displaystyle\int_{\tau(Y)}^s \pare{b_{i(u)}(\tau(Y),x(u),\hat{l}(\tau(Y)))\partial_xf_{i(u)}(u,x(u))}du\nonumber\\
&\hspace{0.3 cm}-\displaystyle\sum_{i=1}^{I}\int_{\tau(Y)}^{s}\alpha_j(\tau(Y),\hat{l}(\tau(Y)))\partial_{x}f_{j}(u,0)dl(u)\Bigg )_{\tau(Y)\leq s\leq T}
\end{align}
is a $\pare{Q^Y, (\Psi_s)_{\tau(Y)\wedge T\leq s\leq T}}$ martingale after the $Q^Y$ almost surely deterministic time $\tau(Y)\wedge T$ (in the sense given in \cite{Stroock}, Chapter I) (once again see Remark \ref{rem:importante-construction-canonique}).
\medskip
\medskip

We are now in position to state the following lemma, which comes as a consequence of Propositions \ref{pr: conti noyau} and \ref{pr: exist proba conca} and describes the first step of our construction.
\begin{Lemma}(One step concatenation of spider probability measures)\label{lm : construction par morceaux}

As usual let $(x,i,l)$ stand for the canonical process of $(\Phi, \B(\Phi))$.

Let $(Q^{Y})_{Y \in \Phi}$ the family of probability measures defined in \eqref{defeq: transition-probability-kernel}.
Then,

-(a) The family $(Q^Y)_{Y \in \Phi}$ is a transition probability kernel from $(\Phi,\Psi_{\tau})$ to $(\Phi,\Psi_{T})$ that satisfies:
$$\forall Y\in \Phi,~~Q^Y\Big(\tilde{X}_{\tau(Y)}(\cdot)=\tilde{X}_{\tau(Y)}(Y)\Big)=1.$$
-(b) Moreover, if $(x_\ast,i_\ast,l_\ast)\in \mathcal{J}\times \R^+$ and $P^{\ast}:=P_0^{(x_\ast,i_\ast),l_\ast}$ stands for the unique probability measure on $(\Phi,\mathbb{B}(\Phi))$ of Theorem \ref{th: exi Freidlin}, then the concatenated probability measure $P^{\ast}\otimes_\tau Q\in {\mathcal P}\pare{\Phi, {\mathbb B}(\Phi)}$ constructed in Proposition \ref{pr: exist proba conca} has the following properties:\\
-(i) $\big(x(0),i(0)),l(0)\big)=\big((x_\ast,i_\ast),l_\ast\big)$, $P^{\ast}\otimes_\tau Q$ a.s.\\ 
-(ii) 
\begin{eqnarray*}
\displaystyle\int_{0}^{T}\mathbf{1}_{\{x(u)>0\}}dl(u)=0,~~P^{\ast}\otimes_\tau Q \text{ a.s.}
\end{eqnarray*}
-(iii) For any $f\in\mathcal{C}^{1,2}_b(\mathcal{J}_T)$, the following process:
\begin{align}
\label{eq:def-Mf}
&\Bigg(M^f(s):= f_{i(s)}(s,x(s)) - f_{i_\ast}(0,x_\ast)\nonumber\\
&\hspace{0.3 cm}-\displaystyle\int_{0}^{\tau\wedge s}\pare{\partial_tf_{i(u)}(u,x(u))+\displaystyle\frac{1}{2}\sigma_{i(u)}^2(0,x(u),l_\ast)\partial_{xx}^2f_{i(u)}(u,x(u))}du\nonumber\\
&\hspace{0.3 cm}-\displaystyle\int_0^{\tau\wedge s} \pare{b_{i(u)}(0,x(u),l_\ast)\partial_xf_{i(u)}(u,x(u))}du\nonumber\\
&\hspace{0.3 cm}-\displaystyle\int_{\tau\wedge s}^s\pare{\partial_tf_{i(u)}(u,x(u))+\displaystyle\frac{1}{2}\sigma_{i(u)}^2(\tau,x(u),l(\tau))\partial_{xx}^2f_{i(u)}(u,x(u))}du\nonumber\\
&\hspace{0.3 cm}-\displaystyle\int_{\tau\wedge s}^s \pare{b_{i(u)}(\tau,x(u),l(\tau))\partial_xf_{i(u)}(u,x(u))}du\nonumber\\
&\hspace{0.3 cm}-\displaystyle\sum_{i=1}^{I}\int_{0}^{\tau \wedge s}\alpha_j(0,l_\ast)\partial_{x}f_{j}(u,0)dl(u) - \sum_{i=1}^{I}\int_{\tau \wedge s}^s\alpha_j(\tau,l(\tau))\partial_{x}f_{j}(u,0)dl(u)
\Bigg)_{0\leq s\leq T}
\end{align}
 is a $\pare{P^{\ast}\otimes_\tau Q, (\Psi_s)_{0 \leq s\leq T}}$ martingale.
\medskip

 \end{Lemma}
\begin{proof}
The fact that $(Q^Y)_{Y\in \Phi}$ is a measurable kernel from $(\Phi,\Psi_{\tau})$ to $(\Phi,\Psi_{T})$ is a straightforward consequence of Proposition \ref{pr: conti noyau} and the point $(a)$ appears as a consequence of the definition of $Q^Y$.

Using the properties of $P^{\ast}\otimes_\tau Q$ that derive from Proposition \ref{pr: exist proba conca}, we recover easily that
\[P^{\ast}\otimes_\tau Q\Big(\big(x(0),i(0), l(0)\big)=\big(x_\ast,i_\ast,l_\ast\big)\Big)=P^{\ast}\Big(\big (x(0),i(0),l(0)\big)=\big(x_\ast,i_\ast,l_\ast\big)\Big)=1.\]
Since $\ds \int_{0}^{\tau}\mathbf{1}_{\{x(u)>0\}}dl(u)$ is $\Psi_\tau$ measurable,
\begin{align}
\label{eq:tl-cond-det}
 P^{\ast}\otimes_\tau Q\Big(\int_{0}^{\tau}\mathbf{1}_{\{x(u)>0\}}dl(u)=0\Big)=P^{\ast}\Big(\int_{0}^{\tau}\mathbf{1}_{\{x(u)>0\}}dl(u)=0\Big)=1.  
\end{align}
Conditioning w.r.t $\Psi_\tau$, we see that
\begin{align}
\label{eq:tl-cond-tau}
&P^{\ast}\otimes_\tau Q\Big(\int_{\tau}^T\mathbf{1}_{\{x(u)>0\}}dl(u)=0\Big)\nonumber\\
&=P^{\ast}\otimes_\tau Q\pare{P^{\ast}\otimes_\tau Q\Big(\int_{\tau}^T\mathbf{1}_{\{x(u)>0\}}dl(u)=0\Big \vert \Psi_\tau\Big)}\nonumber\\
&=P^{\ast}\otimes_\tau Q\pare{\Pi_Y\otimes_{\tau(Y)} Q^Y\Big(\int_{\tau}^T\mathbf{1}_{\{x(u)>0\}}dl(u)=0\Big)}\nonumber\\
&=P^{\ast}\otimes_\tau Q\pare{\Pi_Y\otimes_{\tau(Y)} Q^Y\Big(\int_{\tau(Y)}^T\mathbf{1}_{\{x(u)>0\}}dl(u)=0\Big)},
\end{align}
where we have used \eqref{def-Pi-Proba} for the last line to replace $\tau$ by $\tau(Y)$, which is justified because $\tau$ is a stopping time.

Note that by construction $(l(s))_{\tau(Y)\leq s\leq T}$ is an $(\Psi_s)_{\tau(Y)\leq s\leq T}$ additive functional under\\$P^{\ast}\otimes_\tau Q$. In particular, from the $\Pi_Y\otimes_{\tau(Y)} Q^Y$ a.s. convergence of the Stieljes sums corresponding to the integral $\displaystyle \int_{\tau(Y)}^T\mathbf{1}_{\{x(u)>0\}}dl(u)$, we observe that $$\left \{\tilde{X} = (x,i,l) \in  \Phi~;~\ds \int_{\tau(Y)}^T\mathbf{1}_{\{x(u)>0\}}dl(u)=0\right \}$$ belongs to the $\sigma$-field of canonical events occurring between $\tau(Y)$ and $T$\\
namely $\sigma\pare{\tilde{X}_{(\tau(Y) + s)\wedge T} ,s\geq 0}$.

Thus, we may apply the results of Proposition \ref{pr: exist proba conca} \eqref{def-Pi-Proba-2} to derive from \eqref{eq:tl-cond-tau} that
\begin{align*}
&P^{\ast}\otimes_\tau Q\Big(\int_{\tau}^T\mathbf{1}_{\{x(u)>0\}}dl(u)=0\Big)\\
&={\mathbb E}^{P^{\ast}\otimes_\tau Q}\croc{\Pi_Y\otimes_{\tau(Y)} Q^{Y} \pare{\int_{\tau(Y)}^T\mathbf{1}_{\{x(u)>0\}}dl(u)=0}}\\
&={\mathbb E}^{P^{\ast}\otimes_\tau Q}\croc{Q^{Y} \pare{\int_{\tau(Y)}^T\mathbf{1}_{\{x(u)>0\}}dl(u)=0}}\\
&={\mathbb E}^{P^{\ast}}\croc{P^{\pare{\hat{x}(\tau(Y)), \hat{i}(\tau(Y))}, \hat{\ell}(\tau(Y))}_{\tau(Y)} \pare{\int_{\tau(Y)}^T\mathbf{1}_{\{x(u)>0\}}dl(u)=0}}.
\end{align*}
Recall that for any $Y\in \Phi$, the process $(l(u))$ is $P^{\pare{\hat{x}(\tau(Y)), \hat{i}(\tau(Y))}, \hat{\ell}(\tau(Y))}_{\tau(Y)}$ -a.s constant over $[0,\tau(Y)]$ so that the contribution of the Stieljes measure $dl(u)$ is reduced to zero on this interval. Remember also that $\ds \int_{\tau(Y)}^T\mathbf{1}_{\{x(u)>0\}}dl(u)=0$ under $P^{\pare{\hat{x}(\tau(Y)), \hat{i}(\tau(Y))}, \hat{\ell}(\tau(Y))}_{\tau(Y)}$ by \eqref{eq:zero-acc-tl}. Gathering these facts, we find
\[
P^{\ast}\otimes_\tau Q\Big(\int_{\tau}^T\mathbf{1}_{\{x(u)>0\}}dl(u)=0\Big)=1.
\]
Combining with \eqref{eq:tl-cond-det} ensures
\[P^{\ast}\otimes_\tau Q\Big(\forall s\in [0,T],\;\;\int_{0}^s\mathbf{1}_{\{x(u)>0\}}dl(u)=0\Big)=1.\]
\medskip

Fix $(s,t)\in [0,T]^2$ with $t\leq s$ and remember the definition of $(M^f(s))_{s\in [0,T]}$ in \eqref{eq:def-Mf}.
Then, we have that for any $\Psi_t$ measurable map $\varphi_t\in \mathcal{C}_b(\Phi,\R)$
\begin{align*}
   &\mathbb{E}^{P^{x_\ast}\otimes_\tau Q}[\varphi_t(M^f(s)-M^f(t))]\\
   &=
   \mathbb{E}^{P^{\ast}\otimes_\tau Q}[\varphi_t(M^f(s)-M^f(t))\mathbf{1}_{\{t\leq s< \tau\}}]\\
   &\hspace{0,4 cm}+ \mathbb{E}^{P^{\ast}\otimes_\tau Q}[\varphi_t(M^f(s)-M^f(t))\mathbf{1}_{\{t\leq \tau\leq s\}}] + \mathbb{E}^{P^{\ast}\otimes_\tau Q}[\varphi_t(M^f(s)-M^f(t))\mathbf{1}_{\{\tau< t\leq s\}}] \\
   &:=I_1 + I_2 + I_3.
\end{align*}
For the term $I_1$, observe that the set $\{s< \tau\}$ is $\Psi_\tau$ measurable. Since $t\leq s$ by assumption, this is also the case for the random variable
\[
\varphi_t(M^f(s)-M^f(t))\mathbf{1}_{\{s< \tau\}}.
\]
Since $P^{\ast}\otimes_\tau Q$ restricted to $\Psi_\tau$ is equal to $P^{\ast}$, we get
\begin{align*}
\mathbb{E}^{P^{\ast}\otimes_\tau P_{\tau}}[\varphi_t(M^f(s)-M^f(t))\mathbf{1}_{\{s< \tau\}}]&=
\mathbb{E}^{P^{\ast}}[\varphi_t(M^f(s)-M^f(t))\mathbf{1}_{\{s< \tau\}}].
\end{align*}
Conditioning now w.r.t. $\Psi_s$ and using the fact that $\mathbf{1}_{\{s\leq \tau\}}$ and $\varphi_t$ are $\Psi_s$ measurable r.v. gives
\begin{align*}
I_1&=\mathbb{E}^{P^{\ast}\otimes_\tau Q}[\varphi_t(M^f(s)-M^f(t))\mathbf{1}_{\{s< \tau\}}]\\&=
\mathbb{E}^{P^{\ast}}\croc{\mathbb{E}^{P^{\ast}}\croc{(M^f(s)-M^f(t))\mathbf{1}_{\{s< \tau\}}\big \vert\Psi_s}\varphi_t}\\
&=
\mathbb{E}^{P^{\ast}}\croc{\mathbb{E}^{P^{\ast}}\croc{(\tilde{M}^f(s)-\tilde{M}^f(t))\big \vert\Psi_s}\varphi_t\mathbf{1}_{\{s< \tau\}}}\\
&=0,
\end{align*}
where we have used the facts that $({M}^f(s))_{s\in [0,T]}$ equals $(\tilde{M}^f(s))_{s\in [0,T]}$ (defined in \eqref{eq:def-Mftilde}) on the set $\{s< \tau\}$ and that $(\tilde{M}^f(s))_{s\in [0,T]}$ is a $(\Psi_u)_{u\in [0,T]}$ martingale under $P^{\ast}$. 

The terms $I_2$ and $I_3$ are treated almost in a similar way so we only concentrate on $I_2$.

We now apply the results of Proposition \ref{pr: exist proba conca}. Conditioning in the expectation w.r.t $\Psi_\tau$, we derive
\begin{align*}
\begin{split}
I_2 &= \mathbb{E}^{P^{\ast}\otimes_\tau Q}[\varphi_t(M^f(s)-M^f(t))\mathbf{1}_{\{s\geq \tau\geq t\}}]\\
&=
\mathbb{E}^{P^{\ast}\otimes_\tau Q}[\varphi_t(M^f(s)-M^f(\tau))\mathbf{1}_{\{s\geq \tau\geq t\}}]\\&\hspace{1,4 cm}+\mathbb{E}^{P^{\ast}\otimes_\tau Q}[\varphi_t(M^f(\tau)-M^f(t))\mathbf{1}_{\{s\geq \tau\geq t\}}]\\
&=
\mathbb{E}^{P^{\ast}\otimes_\tau Q}\croc{\mathbb{E}^{P^{\ast}\otimes_\tau Q}\croc{\varphi_t(M^f(s)-M^f(\tau))\mathbf{1}_{\{s\geq \tau\geq t\}}|\Psi_\tau}}\\&\hspace{1,4 cm}+\mathbb{E}^{P^{\ast}}[\varphi_t(M^f(\tau)-M^f(t))\mathbf{1}_{\{s\geq \tau\geq t\}}].
\end{split}
\end{align*}
So that
\begin{align*}
\begin{split}
I_2&=\mathbb{E}^{P^{\ast}\otimes_\tau Q}
\croc{{{\Pi_Y\otimes_{\tau(Y)} Q^Y}}
\croc{(M^f(s)-M^f(\tau))\varphi_t\mathbf{1}_{\{s\geq \tau\geq t\}}}}\\
&\hspace{1,4 cm}+\mathbb{E}^{P^{\ast}}[\varphi_t(M^f(\tau)-M^f(t))\mathbf{1}_{\{s\geq \tau\geq t\}}]\\
&=\mathbb{E}^{P^{\ast}\otimes_\tau Q}
\croc{{{\Pi_Y\otimes_{\tau(Y)} Q^Y}}
\croc{(M^f(s)-M^f(\tau(Y)))\varphi_t(Y)\mathbf{1}_{\{s\geq \tau(Y)\geq t\}}}}\\
&\hspace{1,4 cm}+\mathbb{E}^{P^{\ast}}[(M^f(\tau)-M^f(t))\varphi_t\mathbf{1}_{\{s\geq \tau\geq t\}}]
\end{split}
\end{align*}
where, once again, we have used \eqref{def-Pi-Proba} for the last line allowing us to replace $\tau$ by $\tau(Y)$, which is justified because $\tau$ is a stopping time.

From the definition of $(\hat{M}^f(s))_{s\in [\tau(Y),T]}$\eqref{eq:def-hatM} and $(\tilde{M}^f(s))_{s\in [0,T]}$\eqref{eq:def-Mftilde}, we find
\begin{align*}
\begin{split}
I_2&=\mathbb{E}^{P^{\ast}\otimes_\tau Q}\croc{{{\Pi_Y\otimes_{\tau(Y)} Q^Y}}\croc{(\hat{M}^f(s)-\hat{M}^f(\tau(Y)))\varphi_t(Y)\mathbf{1}_{\{s\geq \tau(Y)\geq t\}}}}\\
&\hspace{1,4 cm}+\mathbb{E}^{P^{\ast}}[(\tilde{M}^f(\tau)-\tilde{M}^f(t))\varphi_t\mathbf{1}_{\{s\geq \tau\geq t\}}]\\
&=\mathbb{E}^{P^{\ast}\otimes_\tau Q}\croc{{{\Pi_Y\otimes_{\tau(Y)} Q^Y}}\croc{(\hat{M}^f(s)-\hat{M}^f(\tau(Y)))}\varphi_t(Y)\mathbf{1}_{\{s\geq \tau(Y)\geq t\}}}\\
&\hspace{1,4 cm}+\mathbb{E}^{P^{\ast}}[(\tilde{M}^f(\tau)-\tilde{M}^f(t))\varphi_t\mathbf{1}_{\{s\geq \tau\geq t\}}]\\
&=\mathbb{E}^{P^{\ast}\otimes_\tau Q}\croc{\E^{Q^Y}\croc{\hat{M}^f(s)-\hat{M}^f(\tau(Y))}\varphi_t(Y)\mathbf{1}_{\{s\geq \tau(Y)\geq t\}}}\\
&\hspace{1,4 cm}+\mathbb{E}^{P^\ast}[(\tilde{M}^f(\tau)-\tilde{M}^f(t))\varphi_t\mathbf{1}_{\{s\geq \tau\geq t\}}]
\end{split}    
\end{align*}
where these equalities are justified by \eqref{def-Pi-Proba} and \eqref{def-Pi-Proba-2} together with the fact that $\hat{M}^f(s)-\hat{M}^f(\tau(Y))$ is $\sigma(\tilde{X}_{(\tau(Y)+t)\wedge T},s\geq 0)$ measurable on the set where $s\geq \tau(Y)$.

We now make use of the fact that $(\tilde{M}^f(s))_{s\in [0,T]}$ is a $(\Psi_u)_{u\in [0,T]}$ martingale under $P^{(x_\ast,i_\ast)}$ and that $(\hat{M}^f(s))_{s\in [\tau(Y),T]}$ is a $(\Psi_u)_{u\in [0,T]}$ martingale under $Q^Y$ after time $\tau(Y)$. 

Since $\tau$ is a bounded stopping time, we are in position to apply the optional sampling theorem for martingales and we get that
$I_2 = 0$.

The term $I_3$ is analyzed similarly and we retrieve $I_3 = 0$, which achieves the proof.
\end{proof}

\subsection{Finite order step concatenation of spider probability measures and tightness}

Let $n\in \mathbb{N}^*$. Consider the following subdivision $(t_j=(jT/n)\big)_{0\leq j \leq n}$ of $[0,T]$. We attach to $(t_j)_{0\leq j \leq n}$ the discretizing step function
$$
\eta_n~:~[0,T]\rightarrow [0,T],\;\;\;u\mapsto \eta_n(u) = \sum_{j=0}^n\,t_j{\bf 1}_{[t_j,t_{j+1})}(u).
$$

Let $(x_\ast,i_\ast,l_\ast)\in \mathcal{J}\times \R^+$. Using Proposition \ref{pr: exist proba conca}, we may define recursively a finite sequence of probability measure $(P^n_j)_{0\leq j \leq n}$ defined on $(\Phi,\mathbb{B}(\Phi))$ by setting inductively
$P^n_0 = P^{x_\ast,i_\ast,l_\ast}$
and for all $j\in [\![0,n-1]\!]$
$$
P^n_{j+1} = P^n_{j}\otimes_{t_{j+1}} P_{t_{j+1}}^{{x}(t_{j+1}), {i}(t_{j+1}),l(t_{j+1})}.
$$
We denote
\[
P^n:= P_n^n.
\]
We leave it to the reader to check inductively -- by mimicking the proof of Lemma 6.1.5 in \cite{Stroock} and using the same arguments as those used for the proof of Lemma \ref{lm : construction par morceaux} -- that the probability $P^n$ satisfies the following properties:\\
-(i) $\big(x(0),i(0)),l(0)\big)=\big(x_\ast,i_\ast,l_\ast\big)$, $P^n-\text{a.s}$.
\\ 
-(ii) For each $s \in [0,T]$: 
\begin{eqnarray*}
\displaystyle\displaystyle\int_{0}^{s}\mathbf{1}_{\{x(u)>0\}}dl(u)=0,~~P^{n}-\text{a.s.}
\end{eqnarray*}
-(iii) For any $f\in\mathcal{C}^{1,2}_b(\mathcal{J}_T)$, the following process: 
\begin{equation}
\label{eq:def-Vnf}
\begin{split}
&\Big (V_n^f(s):= f_{i(s)}(s,x(s))- f_{i}(0,x_\ast)\\
&\hspace{0,3 cm}-\displaystyle\int_0^{s}\big (\partial_tf_{i(u)}(u,x(u))+\displaystyle\frac{1}{2}\sigma_{i(u)}^2(\eta_n(u),x(u),l(\eta_n(u)))\partial_{x}^2f_{i(u)}(u,x(u))\\
&\hspace{1,3 cm}+b_{i(u)}(\eta_n(u),x(u),l(\eta_n(u)))\partial_xf_{i(u)}(u,x(u))\big )du~~
\\
&\hspace{0,3 cm}-~~\displaystyle\sum_{i=1}^{I}\int_0^{s}\alpha_i(\eta_n(u),l(\eta_n(u)))\partial_{x}f_{i}(u,0)dl(u)\Big )_{s\in [0,T]},
\end{split}
\end{equation} 
is a $(\Psi_s)_{0 \leq s \leq T}$ martingale under the probability measure $P^n$.
\medskip
\begin{Proposition}\label{pr: existence Brow-concatenation}
We have
\begin{eqnarray*}
&\forall s\in[0,T],~~d\langle x(\cdot)\rangle_{s}=\sigma^2_{i(s)}(\eta_n(s),x(s), l(\eta_n(s)))ds,\;\;~~P^{n}-\text{ a.s.}
\end{eqnarray*}
Moreover, there exists a $(\Psi_{s})_{0\leq s\leq T}$ standard one dimensional Brownian motion $W^n$, such that $P^{n} - \text{ a.s.}$: 

$\forall s\in[0,T],$
\begin{align}\label{eq : diff x-concatenation}
x(s)=x_\ast+\displaystyle\int_{0}^{s}\sigma_{i(u)}(\eta_n(u),x(u),l(\eta_n(u)))dW^n_u + \displaystyle\int_{0}^{s}b_{i(u)}(\eta_n(u),x(u),l(\eta_n(u)))du+l(s).
\end{align}
\end{Proposition}
\begin{proof}
Observe that the coordinate function $\mathcal{J}_T\rightarrow \R^+$, $(t,x,i)\mapsto x$ belongs to $\mathcal{C}^{1,2}(\mathcal{J}_T)$. Hence, by (iii), we know that
\[
\pare{\mathfrak{m}_s:= x(s) - x_\ast - \int_0^s b_{i(u)}(\eta_n(u),x(u),l({\eta_n}(u)))du - l(s)}_{s\in [0,T]}
\]
is a $(\Psi_s)_{0 \leq s \leq T}$ local martingale under the probability measure $P^n$.
For simplicity, set for any $s\in [0,T]$
\[
a_s := x_\ast + \int_0^s b_{i(u)}(\eta_n(u),x(u),l({\eta_n}(u)))du + l(s).
\]
Observe that $(a_s)_{s\in [0,T]}$ is of finite variation so that necessarily $\langle \mathfrak{m}\rangle_t = \langle x(.)\rangle_t$ ($t\in [0,T]$). 
Then, applying It\^o's formula to $\pare{\mathfrak{m}_s}_{s\in [0,T]}$, we have for $t\in [0,T]$:
\begin{align}
\mathfrak{m}^2_t &= 2\int_0^t \mathfrak{m}_sd\mathfrak{m}_s + \langle \mathfrak{m}\rangle_t\nonumber\\
&=2\int_0^t (x(s) -a_s)d\pare{x(s) - a_s} + \langle\mathfrak{m}\rangle_t\\
&=2\int_0^t x(s)dx(s) - 2\int_0^t x(s)da_s - 2\int_0^t a(s)d\mathfrak{m}_s + \langle\mathfrak{m}\rangle_t\nonumber\\
&=x(t)^2 - x_\ast^2 - \langle x\rangle_t - 2\int_0^t x(s)da_s - 2\int_0^t a(s)d\mathfrak{m}_s + \langle\mathfrak{m}\rangle_t\nonumber\\
\label{eq:calcul-crochet}&=x(t)^2 - x_\ast^2- 2\int_0^t x(s)da_s - 2\int_0^t a(s)d\mathfrak{m}_s.
\end{align}
Observe now that the quadratic coordinate function $\mathcal{J}_T\rightarrow \R^+$, $(t,x,i)\mapsto x^2$ belongs also to $\mathcal{C}^{1,2}(\mathcal{J}_T)$. Hence, by (iii) and from the definition of $(a_s)_{s\in [0,T]}$, we have that
\[
\pare{\tilde{\mathfrak{m}}_s:= x^2(s) - x_\ast^2 - \int_0^s \sigma^2_{i(u)}(\eta_n(u),x(u),l({\eta_n}(u)))du -\int_0^s x(s)da_s}_{s\in [0,T]}
\]
is also a $(\Psi_s)_{0 \leq s \leq T}$ local martingale under the probability measure $P^n$. So that from \eqref{eq:calcul-crochet}
\begin{align*}
\mathfrak{m}^2_t 
&=\mathfrak{\tilde{m}}_t + \pare{\int_0^t \sigma^2_{i(s)}(\eta_n(s),x(s),l({\eta_n}(s)))ds  +  2\int_0^t x(s)da_s} - 2\int_0^t x(s)da_s - 2\int_0^t a(s)d\mathfrak{m}_s\\
&=\mathfrak{\tilde{m}}_t - 2\int_0^t a(s)d\mathfrak{m}_s + \int_0^t \sigma^2_{i(s)}(\eta_n(s),x(s),l({\eta_n}(s)))ds  
\end{align*}
which permits to identify
$$
\langle \mathfrak{m}\rangle_t = \langle x\rangle_t = \int_0^t \sigma^2_{i(s)}(\eta_n(s),x(s),l({\eta_n}(s)))ds
$$
since $\pare{\int_0^t \sigma^2_{i(s)}(\eta_n(s),x(s),l({\eta_n}(s)))ds}_{t\in [0,T]}$ is an increasing continuous process and\newline $\pare{\mathfrak{m}^2_t - \int_0^t \sigma^2_{i(s)}(\eta_n(s),x(s),l({\eta_n}(s)))ds}_{t\in [0,T]}$ is the $(\Psi_s)_{0 \leq s \leq T}$ local martingale given by \newline $\pare{\tilde{m}_t - 2\int_0^ta_sd{m}_s}_{t\in [0,T]}$.

The conclusion follows then from assumption $(\bf E)$ by applying the representation theorem for continuous local martingales stated in Proposition 3.8 Chapter V in \cite{Revuz-Yor} (p. 202) to  $(\mathfrak{m}_t)_{t\in [0,T]}$.
\end{proof}
\begin{Lemma}
\label{lemma:tightness-concatenation}
There exists a constant $C$, depending only on the data $(T,|b|,|\sigma|)$, such that:
\begin{align}
&\label{maj : moment-ordre-deux}\sup_{n\ge 0}\mathbb{E}^{P^n}\croc{|x|^2_{(0,T)}+ |l|^2_{(0,T)}}\leq C(1+x_\ast^{2}),\\
&\label{maj : modul de continuité}\forall \theta\in (0,1),\;\;\;\;\sup_{n\ge 0}\mathbb{E}^{P^n}\croc{\omega(X,\theta)^{2}+\omega(l,\theta)^{2}}\leq C\theta \ln(\displaystyle{2T}/{\theta}).
\end{align}
and
\begin{align}
\label{eq:modulus-continuity-X-concatenation}
&\forall \theta\in (0,1),\;\;\;\;\sup_{n\ge 0}{E^{P^n}\Big [\tilde{\omega}(\tilde{X},\theta)\Big ]\leq C\sqrt{\theta\ln(2T/\theta)}}.
\end{align}
\end{Lemma}
\begin{proof}
Using the description \eqref{eq : diff x-concatenation} for the process $(x(t))_{t\geq 0}$, we claim that the inequalities \eqref{maj : moment-ordre-deux} and \eqref{maj : modul de continuité} can be retrieved repeating the same arguments as those used for the derivation of \eqref{eq:majoration-moments-ordre-deux} and \eqref{eq:modulus-continuity-x-l} (see the Appendix \ref{sk: ineg modulus inequality} for the proof of \eqref{eq:majoration-moments-ordre-deux} and \eqref{eq:modulus-continuity-x-l}).

\end{proof}
\begin{Corollary}
\label{corollary-relat-compactness}
The sequence $(P^n)$ is relatively compact for the weak topology on $\mathcal{P}\pare{\Phi,\B(\Phi)}$.
\end{Corollary}
\begin{Remark}
\label{rem:tension-spider}
This result cannot be deduced directly from Lemma \ref{lemma:tightness-concatenation} by using standard results (such as Theorem 3.21 in \cite{J-S} Chapter VI). Up to our knowledge there is no tightness criteria for sequences of probability measures on a general Polish space. For an account on these problems, we mention \cite{Meziani} when the metric space under consideration is a Banach space (which is not our case here since the spider web space is not even a vector field). 
\medskip

However, observe that Prokhorov's theorem is valid for general Polish spaces (see Theorem 3.5 \cite{J-S} Chapter VI p.311). Thus, we can prove below Corollary \ref{corollary-relat-compactness} by following classical arguments (see for e.g. the proof of the characterization Theorem 3.21 in \cite{J-S} Chapter VI p. 314):  the key fact here is that we are working on a Polish space whose topology remains induced by a uniform metric (the distance $d^{\Phi}$), so that we are still in position to apply Ascoli-Arzela's theorem.
\end{Remark}
\begin{proof} {\it (of Corollary \ref{corollary-relat-compactness})}

Fix $\varepsilon>0$. From the result of Lemma \ref{lemma:tightness-concatenation}, using Markov's inequality, it is not hard to see that there exists a constant $K_{\varepsilon} < \infty$ and $\theta_{\varepsilon,\,k}>0$ satisfying
\begin{align*}
&\sup_{n\geq 0} P^n\pare{\sup_{0\leq t\leq T} (|x(t)| + |i(t)| + |l(t)|) > K_{\varepsilon}}\leq \frac{\varepsilon}{2}\\
&\sup_{n\geq 0} P^n\pare{\tilde{\omega}(\tilde{X},\theta_{\varepsilon,\,k})>\frac{1}{k}}\leq \frac{\varepsilon}{2^k},
\end{align*}
(take $\theta_{\varepsilon,\,k}>0$ so small that $\theta_{n,k}\ln(2T/\theta_{n,k})\leq \varepsilon^2/(C^2k^22^{2k})$ where $C$ stands for the constant appearing in Lemma \ref{lemma:tightness-concatenation}).

Then 
\[
A_{\varepsilon} = \left \{\gamma\in \Phi~:~\sup_{0\leq t\leq T}|\gamma(t)|\leq K_{\varepsilon},\;\;\tilde{\omega}(\gamma,\theta_{\varepsilon,\,k})\leq \frac{1}{k}\;\;\text{ for all }k\in \N^\ast\right \}.
\]
Therefore,
\[
P^n\pare{\tilde{X}\notin A_{\varepsilon}}\leq P^n\pare{\sup_{0\leq t\leq T} |\tilde{X}(t)| > K_{\varepsilon}} + \sum_{k=1}^\infty P^n\pare{\tilde{\omega}(\tilde{X},\theta_{\varepsilon,\,k})>\frac{1}{k}}\leq \varepsilon.
\]
We claim that $A_{\varepsilon}$ is relatively compact in $\Phi$. Indeed, this may be seen as a consequence of Ascoli-Arzela's theorem applied to the vector space of continuous functions on the compact set $[0,T]$ with values in the Polish space $\pare{\mathcal{J}\times \R^+, d^{\mathcal{J}}((a,i),(b,j)) + |l-l'|}$ endowed with the corresponding uniform metric. Hence, we have
\[
P^n\pare{\tilde{X}\notin \overline{A_{\varepsilon}}}\leq \varepsilon.
\]
The sequence $(P^n)_{n\in \N}$ is tight and the result is a consequence of Prokhorov's theorem applied on $(\Phi, \B(\Phi))$ (see for e.g. Theorem 3.5 \cite{J-S} Chapter VI p.311).
\end{proof}
From the result of the above Corollary \ref{corollary-relat-compactness}, the sequence $(P^n)$ converges weakly - up to a sub sequence $(n_k)$ - to a probability measure $P\in \mathcal{P}(\Phi, \B(\Phi))$.

We end this subsection with the following crucial result:
\begin{Proposition}('non-stickiness' of $P$)
\label{prop:non-stickness-1}

Assume assumption $(\mathcal{H})$ and let $P$ constructed as above. Then, there exists a constant $C>0$, depending only on the data $\Big(T,|b|,|\sigma|,c,x_\ast\Big)$ such that
\begin{eqnarray*}\forall \varepsilon>0,~~\mathbb{E}^{P}\croc{\int_{0}^{T}\mathbf{1}_{\{x(s)<\varepsilon\}}ds}\leq C\varepsilon.
\end{eqnarray*}
In particular,
\begin{equation*}
\mathbb{E}^{P}\croc{\int_0^T\mathbf{1}_{\{x(s)=0\}}ds}=0\;\;\text{\;and\;}~~ \int_0^T\mathbf{1}_{\{x(s)=0\}}ds~~P-\text{a.s.}
\end{equation*}
\end{Proposition}
\begin{proof}
By applying the same arguments as those given in the proof of Lemma \ref{lem:non-stickness} (using in force the result of Proposition \ref{pr: existence Brow-concatenation} - \eqref{eq : diff x-concatenation}), we prove that there exists a constant $C>0$, depending only on the data $\Big(T,|b|,|\sigma|,c,x_\ast\Big)$, such that
\begin{eqnarray*}\forall \varepsilon>0,~~\forall k\ge 0,~~\mathbb{E}^{P^{n_k}}\croc{\int_{0}^{T}\mathbf{1}_{\{x(s)<\varepsilon\}}ds}~~\leq~~C\varepsilon.
\end{eqnarray*}
Making use of the fact that the following functional: $\beta:\mathcal{C}[0,T]\to \R,~~x(.)\mapsto \displaystyle \int_t^T\mathbf{1}_{\{x(s)<\varepsilon\}}ds$ is lower semi continuous, we get that:
\begin{eqnarray*}\forall \varepsilon>0,~~\mathbb{E}^{P}\Big[~~\int_{0}^{T}\mathbf{1}_{\{x(s)<\varepsilon\}}ds~~\Big]\leq\liminf_{k \to +\infty}\mathbb{E}^{P^{n_k}}\Big[~~\int_{0}^{T}\mathbf{1}_{\{x(s)<\varepsilon\}}ds~~\Big]~~\leq~~C\varepsilon.
\end{eqnarray*}
\end{proof}

\section{Existence: construction of a Walsh spider diffusion with dependent local time spinning measure}\label{sec: preuve et paths properties}
This section is dedicated to the proof of the existence part of Theorem \ref{th: exist Spider }.
\subsection{Proof of Theorem \ref{th: exist Spider } I - Existence}
\begin{proof} {\it{Theorem \ref{th: exist Spider } I - Existence}}

Let $(x_\ast,i_\ast,l_\ast)\in \mathcal{J}\times \R^+$. Let us assume assumption $(\mathcal{H})$ and denote by $P$ a probability on $(\Phi, \B(\Phi))$ constructed as above. 

We claim that the first point (i) is satisfied by construction.
Indeed, using Portmanteau's theorem to the closed set
$\{\tilde{X}\in \Phi, \tilde{X}(0)=((x(0),i(0),l(0)) = (x_\ast,i_\ast,l_\ast)\}$, we have
\begin{align}
P((x(0),i(0),l(0))=(x_\ast,i_\ast,l_\ast))\geq \limsup P^n((x(0),i(0),l(0))=(x_\ast,i_\ast,l_\ast)) = 1.
\end{align}.
For (ii), we have from Proposition \ref{pr : lower semi continuity}, that the following map:
\begin{eqnarray*}
\rho:\begin{cases}
\mathcal{C}^{\mathcal{J}}[0,T]\times \mathcal{L}[0,T]\to\mathbb{R}\\
\Big(\big(x(\cdot),i(\cdot)\big),l(\cdot)\Big)\mapsto \displaystyle\int_{0}^T\mathbf{1}_{\{x(u)>0\}}l(du)
\end{cases}
\end{eqnarray*}
is lower semi continuous. Consequently, the following set $O$ defined by
\begin{eqnarray*}
&O:=\Big\{\Big(\big(x,i\big),l\Big)\in\mathcal{C}^{\mathcal{J}}[0,T]\times \mathcal{L}[0,T],~~\displaystyle\int_{0}^T\mathbf{1}_{\{x(u)>0\}}l(du)>0\Big\},
\end{eqnarray*}
is open in $\mathcal{C}^{\mathcal{J}}[0,T]\times \mathcal{L}[0,T]$. Hence, using Portmanteau's theorem for open sets, we have:
$$ 0=\liminf_{k \to +\infty}P^{n_k}(O)\ge P(O),$$
which means that (ii) holds true.

Finally, let us show point (iii). 

Recall the notations \eqref{eq:def-V} and \eqref{eq:def-Vnf}. Using assumption $(\mathcal{H})$, Lebesgue's theorem with the non stickiness result of Proposition \ref{prop:non-stickness-1}, it is not hard to deduce that $P$ almost surely:
\begin{eqnarray*}
&\lim_{k\to +\infty} \Big|V^f_{n_k}(\cdot) - V^f(\cdot) \Big|_{(0,T)}~=~0.
\end{eqnarray*}
Now fix $s\in [0,T]$ and $u\in[0,T]$, with $u\ge s$. Let $\varphi\in \mathcal{C}_b(\Phi,\R)$, $\Psi_s$ measurable, and let $f\in\mathcal{C}^{1,2}_b(\mathcal{J}_T)$. Using Lebesgue's theorem for weakly convergence measures (see Corollary \ref{cr Lebesgue faible} in Appendix), we have that:
\begin{eqnarray*}
\lim_{k\to + \infty}\mathbb{E}^{P^{n_k}} \croc{\varphi_s(V^{f}_{n_k}(u)-V^{f}_{n_k}(s))}~=~\mathbb{E}^{P}\croc{\varphi_s(V^{f}(u)-V^{f}(s))}.
\end{eqnarray*}
Hence, for any $f\in\mathcal{C}^{1,2}_b(\mathcal{J}_T)$ $(V^f(s))_{0\leq s\leq T}$ is a $(\Psi_s)_{0 \leq s \leq T}$ martingale under the measure of probability $P$, namely point (iii) of Theorem \ref{th: exist Spider } is satisfied.
\end{proof}

For the remaining of this paper, given $(x_\ast,i_\ast, l_\ast)\in \mathcal{J}\times \R^+$, we will denote by $\P^{x_\ast,i_\ast,l_\ast}$ a probability defined on the measurable space $(\Phi,\mathbb{B}(\Phi))$ solution of the martingale problem \emph{$\big(\mathcal{S}_{pi}-\mathcal{M}_{ar}\big)$} and satisfying $$\P^{x_\ast,i_\ast,l_\ast}\pare{(x(0), i(0), l(0)) = (x_\ast, i_\ast, l_\ast)} = 1.$$  

\subsection{The first component as a reflected process}
Next Proposition \ref{pr: existence Brow} characterizes the paths of the process $(x(t))_{0\leq t\leq T}$, by showing that its martingale part can be represented as a Brownian integral.
\begin{Proposition}\label{pr: existence Brow}
We have:
\begin{eqnarray*}
&\forall s\in[0,T],~~d\langle x(\cdot)\rangle_{s}~~=~~\sigma^2_{i(s)}(s,x(s))ds,\;\;\P^{x_\ast,i_\ast,l_\ast}-\text{ a.s.}
\end{eqnarray*}
Moreover, there exists a $(\Psi_{s})_{0\leq s\leq T}$ standard one dimensional Brownian motion $W$, such that $\P^{x_\ast,i_\ast,l_\ast} - \text{ a.s.}$:  
\begin{eqnarray*}\label{eq : diff x}
\nonumber&\forall s\in[0,T],~~x(s)=x+\displaystyle\int_{0}^{s}b_{i(u)}(u,x(u),l(u))du+\displaystyle\int_{0}^{s}\sigma_{i(u)}(u,x(u),l(u))dW(u)+l(s).
\end{eqnarray*}
\end{Proposition}
\begin{proof}
Relying on the property (iii) of Theorem \ref{th: exist Spider }, the proof of Proposition \ref{pr: existence Brow} follows exactly the same lines as in the proof of Proposition \ref{pr: existence Brow-concatenation}: we do not repeat it here.
\end{proof}

As consequence of Proposition \ref{pr: existence Brow} or using the martingale property, we are allowed to repeat all of the arguments in Lemma \ref{lem:non-stickness} to get:
\begin{Proposition}('Non stickiness' property)
\label{prop:non-stickness}
Assume assumption $(\mathcal{H})$. Then, there exists a constant $C>0$, depending only on the data $\Big(T,|b|,|\sigma|,c,x_\ast\Big)$ such that
\begin{eqnarray*}\forall \varepsilon>0,~~\mathbb{E}^{\P^{x_\ast,i_\ast,l_\ast}}\croc{\int_{0}^{T}\mathbf{1}_{\{x(s)<\varepsilon\}}ds}\leq C\varepsilon.
\end{eqnarray*}
In particular,
\begin{equation*}
\mathbb{E}^{\P^{x_\ast,i_\ast,l_\ast}}\croc{\int_0^T\mathbf{1}_{\{x(s)=0\}}ds}=0\;\;\text{\;and\;}~~ \int_0^T\mathbf{1}_{\{x(s)=0\}}ds,\;\;~\P^{x_\ast,i_\ast,l_\ast}-\text{a.s.}
\end{equation*}
\end{Proposition}

\subsection{Extension of the martingale property}

The aim of this paragraph is to extend the martingale property under $\P^{x_\ast,i_\ast,l_\ast}$ to test functions in $\mathcal{C}_b^{1,2,1}\pare{[0,T]\times \mathcal{J}\times \R^+;\R}$. We have the following result.
\begin{Proposition}
\label{pr: martingale-property-extension}
For any $f\in\mathcal{C}^{1,2,1}_b([0,T]\times \mathcal{J}\times \R^+;\R)$, the following process:
\begin{align}
\label{eq:def-V}
&\Bigg(f_{i(s)}(s,x(s),l(s))- f_{i_\ast}(0,x_\ast,l_\ast)\displaystyle\nonumber\\
&\displaystyle -\int_{0}^{s}\pare{\partial_tf_{i(u)}(u,x(u),l(u))+\displaystyle\frac{1}{2}\sigma_{i(u)}^2(u,x(u),l(u))\partial_{xx}^2f_{i(u)}(u,x(u),l(u))}du\nonumber\\
&\displaystyle -\int_{0}^{s}\pare{b_{i(u)}(u,x(u),l(u))\partial_xf_{i(u)}(u,x(u),l(u))}du\nonumber\\
&\displaystyle -\int_{0}^{s}\partial_{l}f(u,0,l(u))dl(u)- \displaystyle\sum_{j=1}^I\int_{0}^{s}\alpha_j(u,l(u))\partial_{x}f_{j}(u,0,l(u))dl(u)~\Bigg)_{0\leq s\leq T}\nonumber,
\end{align} 
is a $(\Psi_s)_{0 \leq s \leq T}$ martingale under $\P^{x_\ast,i_\ast,l_\ast}$.
\end{Proposition}
\begin{proof}

{\bf Step 1}

Let $f\in\mathcal{C}^{1,2,1}_b([0,T]\times \mathcal{J}\times \R^+;\R)$ with product form $f(t,x,l) = h(t,x)g(l)$ where $h\in \mathcal{C}_b^{1,2}(\mathcal{J}_T)$ and $g\in \mathcal{C}_b^{1}(\R^+;\R)$. 

From the result of Theorem \ref{th: exist Spider } (Existence part) and using the notation \eqref{eq:def-V}, we have that $(V^h(s))_{s\in [0,T]}$ is a martingale under $\P^{x_\ast,i_\ast,l_\ast}$. Applying the stochastic integration by parts formula to the $\R$ valued processes $(V^h(s))_{s\in [0,T]}$ and $g(l(s))_{s\in [0,T]}$ and because $(l(s))_{s\in [0,T]}$ is by construction of finite variation , we see that for all $s\in [0,T]$
\[
V^h(s)g(l(s)) - V^h(0)g(l_\ast) - \int_0^s V^h(s)g'(l(s))dl(s) = \int_0^s g(l(s))d(V^h(s))
\]
which defines a $(\Psi_s)_{0 \leq s \leq T}$ martingale under $\P^{x_\ast,i_\ast,l_\ast}$.

Using then Proposition \ref{prop:non-stickness} and identifying the martingale part in this equality, we retrieve the result of Proposition \ref{pr: martingale-property-extension} in this case.

Using the same identification, it is easy to see that the result of Proposition \ref{pr: martingale-property-extension} holds also true for functions $f\in\mathcal{C}^{1,2,1}_b([0,T]\times \mathcal{J}\times \R^+;\R)$ with a sum form $f(t,x,l) = h(t,x) + g(l)$ with $h\in \mathcal{C}_b^{1,2}(\mathcal{J}_T)$ and $g\in \mathcal{C}_b^{1}(\R^+;\R)$..

Let us introduce $\R[X]$ (resp. $\R[X,Y]$) the set of one-indeterminate polynomials on $\R$ (resp. two indeterminate polynomials $\R[X,Y]$ on $\R$) and
\begin{align*}
\mathcal{G} := \left \{
f\in\mathcal{C}^{1,2,1}([0,T]\times \mathcal{J}\times \R^+;\R)~\Big |~\exists ((P_i)_{i\in [I]},Q)\in \pare{(\R[X,Y])^I\times \R[X]},\right .\\
\hspace{0,4 cm}\left .f_i(t,x,l)=P_i(t,x)Q(l)\;\;\;\forall i\in [I]
\right \}.
\end{align*}

As a consequence of the above arguments, we claim that the result of Proposition \ref{pr: martingale-property-extension} holds true for all functions $f\in \mathcal{G}$.

{\bf Step 2}

Let $f\in\mathcal{C}^{0,0,0}_b([0,T]\times \mathcal{J}\times \R^+;\R)$. Fix $N\in \N^\ast$, we claim that there is $(f^{(n)})_{n\in \N^\ast}$ a sequence on $\mathcal{G}$ -- in particular $f^{(n)}\in \mathcal{C}^{1,2,1}(\mathcal{J}_T\times \R^+)$ -- such that 
\[
\sup_{(t,(x,i),l)\in [0,T]\times \mathcal{J}\times \R^+}|f_i(t,x\wedge N,l\wedge N) - f^{(n)}_i(t,x\wedge N,l\wedge N)|
\leq C/n\]
for a constant $C$ that depends only on $N$ on $f$ and the data: it suffices to take the Bernstein polynomial
\begin{align*}
&f_i^{(n)}(t,x,l)\,=\,\\
&\!\!\!\!\sum_{j,k,m=1}^n\!\!\binom{n}{j}\binom{n}{k}\binom{n}{m}f_i(jT/n, kN/n, mN/n)(t - jT/n)^{n-j}t^{j}(x - kN/n)^{n-k}x^{k}(l - mN/n)^{n-m}l^m.
\end{align*}

Hence, on the set $[0,T]\times ([I]\times [0,N])\times [0,N]$, any function $f\in\mathcal{C}^{1,2,1}_b([0,T]\times \mathcal{J}\times \R^+;\R)$ is the limit in $\mathcal{C}^{1,2,1}_b([0,T]\times ([I]\times [0,N])\times [0,N];\R)$ of a sequence of polynomial functions in $\mathcal{G}$. From Step 1, by stopping and using the ordinary and stochastic dominated convergence theorems (see for e.g. \cite{Revuz-Yor} Chapter IV Theorem (2.12)), the result of the proposition is established.
\end{proof}

\section{Uniqueness in law}
\label{sec:uniqueness}
The main purpose of this section is to prove the uniqueness of the weak solution
$$\Biggl(\Phi,\mathbb{B}(\Phi),(\Psi_t)_{0 \leq t \leq T}, \Big( x(\cdot),i(\cdot),l(\cdot)\Big),\P^{x_\ast,i_\ast,l_\ast}\Biggl)$$
of our central Theorem \ref{th: exist Spider }, whose existence has been proven in the last Section \ref{sec: preuve et paths properties}. For this purpose, we will use a PDE argument and more precisely the last recent results obtained in \cite{Martinez-Ohavi EDP}. We start first by recalling the central Theorem obtained in \cite{Martinez-Ohavi EDP} for our purpose, which ensures the well-posedness (existence and uniqueness) for a system of linear parabolic partial differential equations posed on the star-shaped network $\mathcal{J}$ (junction), satisfying a {\it local time Kirchhoff's transmission condition} at the junction point $\bf 0$ (see Theorem 2.8 in \cite{Martinez-Ohavi EDP}). 

\subsection{The backward parabolic differential equation associated to the martingale problem $\big(\mathcal{S}_{pi}-\mathcal{M}_{ar}\big)$}
\noindent

In the whole section we are given a fixed terminal condition $T>0$. 

We introduce the following two domains:
$$\mathcal{D}_{\infty}=\mathcal{J}\times [0,+\infty),~~\mathcal{D}_{T,\infty}=[0,T]\times\mathcal{D}_{\infty}.$$
Let us introduce the class of regularity corresponding to the unique solution of the PDE system.
\begin{Definition} of the class $\mathfrak{C}^{1,2,0}_{\{0\},b} \big(\mathcal{D}_{T,\infty}\big)$ (see \cite{Martinez-Ohavi EDP} Definition 2.7)\label{def: class regula EDP}.

We say that
$$f:=\begin{cases} \mathcal{D}_{T,\infty}\to \R,\\
\big(t,(x,i),l\big)\mapsto f_i(t,x,l)
\end{cases}.$$
is in the class $f\in \mathfrak{C}^{1,2,0}_{\{0\},b} \big(\mathcal{D}_{T,\infty}\big)$ if\\
(i) the following continuity condition holds at the junction point $\{0\}$: \\for all $(t,l)\in[0,T]\times[0,+\infty)$, for all $(i,j)\in [I]^2$, $f_i(t,0,l)=f_j(t,0,l)=f(t,0,l)$;\\
(ii) for all $i\in [I]$, the map $(t,x,l)\mapsto f_i(t,x,l)$ has a regularity in the class\\ $\mathcal{C}^{0,1,0}_b\big([0,T]\times [0,+\infty)^2,\R \big)$ ;\\
(iii) for all $i\in [I]$, the map $(t,x,l)\mapsto f_i(t,x,l)$ has regularity in the interior of each ray $\mathcal{R}_i$ in the class $\mathcal{C}^{1,2,0}_b\big((0,T)\times (0,+\infty)^2,\R \big)$;\\
(iv) at the junction point $\bf 0$, the map $(t,l)\mapsto f(t,0,l)$ has a regularity in the class\\ $\mathcal{C}^{0,1}_b\big([0,T]\times [0,+\infty),\R \big)$ ;\\
(v) for all $i\in [I]$, on each ray $\mathcal{R}_i$, $f$ admits a generalized locally integrable derivative with respect to the variable $l$ in $\displaystyle \bigcap\limits_{q\in (1,+\infty)}\!L^q_{loc}\big((0,T)\times (0,+\infty)^2\big)$.
\end{Definition}
As a consequence of Theorem 2.8 in \cite{Martinez-Ohavi EDP} we have:
\begin{Theorem}\label{th : exis para with l bord infini}
Assume assumption $(\mathcal{H})$. Let $g\in \mathcal{C}^{\infty}_b\big(\mathcal{D}_{\infty}\big)$ satisfying the following compatibility condition:
\begin{align*}
&\partial_lg(0,l)+\sum_{i=1}^I\alpha_i(\mathcal{T},l)\partial_xg_i(0,l)=0,~~l\in[0,+\infty).
\end{align*}
where $\mathcal{T}\in (0,T]$ is a terminal condition.
Then, the following system of backward linear parabolic partial differential equations involving a {\it local time Kirchhoff's transmission condition} posed on the domain $\mathcal{D}_{\mathcal{T},\infty}$:
\begin{eqnarray}\label{eq : pde with l bord infini Walsh}
\begin{cases}&\partial_tu_i(t,x,l)+\ds \frac{1}{2}\sigma^2_i(t,x,l)\partial_x^2u_i(t,x,l)\\
&\hspace{2,0 cm}+b_i(t,x,l)\partial_xu_i(t,x,l)=0,\,(t,x,l)\in (0,\mathcal{T})\times (0,+\infty)^2,\\
&\partial_lu(t,0,l)+\displaystyle \sum_{i=1}^I \alpha_i(t,l)\partial_xu_i(t,0,l)=0,~~(t,l)\in(0,\mathcal{T})\times(0,+\infty)\\
&\forall (i,j)\in[I]^2,~~u_i(t,0,l)=u_j(t,0,l)=u(t,0,l),~~(t,l)\in[0,\mathcal{T}]\times[0,+\infty)^2,\\
&\forall i\in[I],~~ u_i(\mathcal{T},x,l)=g_i(x,l),~~(x,l)\in[0,+\infty)^2,
\end{cases}
\end{eqnarray}
is uniquely solvable in the class $\mathfrak{C}^{1,2,0}_{\{0\},b} \big(\mathcal{D}_{\mathcal{T},\infty}\big)$. 
\end{Theorem}

Observe that the class of test function $\mathcal{C}_b^{1,2,1}([0,T]\times \mathcal{J}\times \R^+;\R)$ of Proposition \ref{pr: martingale-property-extension} is much smaller than the class $\mathfrak{C}^{1,2,0}_{\{0\},b} \big(\mathcal{D}_{T,\infty}\big)$. The issue of the next proposition is to show that the  martingale property of Theorem \ref{th: exist Spider } (iii) extends further to any test function in the class $\mathfrak{C}^{1,2,0}_{\{0\},b} \big(\mathcal{D}_{T,\infty}\big)$.
\begin{Proposition}\label{pr : weak Ito formula}
Assume assumption $(\mathcal{H})$, and let $$\Biggl(\Phi,\mathbb{B}(\Phi),(\Psi_t)_{0 \leq t \leq T}, \Big(x(\cdot),i(\cdot),l(\cdot)\Big),\P^{x_\ast,i_\ast,l_\ast}\Biggl)$$
be a weak solution solution of the spider martingale problem $\big(\mathcal{S}_{pi}-\mathcal{M}_{ar}\big)$ of Theorem \ref{th: exist Spider }. Then for all $f\in \mathfrak{C}^{1,2,0}_{\{0\},\infty} \big(\mathcal{D}_{T,\infty}\big)$:
\begin{align*}
&\Bigg(~~f_{i(t)}(t,x(t),l(t)) - f_{i_\ast}(0,x_\ast,l_\ast)\\
&-\displaystyle\int_{0}^{t}\Big(\partial_tf_{i(u)}(u,x(u),l(u))+\displaystyle\frac{1}{2}\sigma_{i(u)}^2(u,x(u),l(u))\partial_{x}^2f_{i(u)}(u,x(u),l(u))\\
&+b_{i(u)}(u,x(u),l(u))\partial_xf_{i(u)}(u,x(u),l(u))\Big)du\\
&-\ds\int_{0}^{t}\Big(\partial_lf(u,0,l(u))+\displaystyle\sum_{j=0}^{I}\alpha_j(u,l(u))\partial_{x}f_{i}(u,0,l(u))\Big)dl(u)~~\Bigg)_{t\in[0,T]},
\end{align*} 
is a $(\Psi_t)_{0 \leq t \leq T}$ martingale under the probability measure $\P^{x_\ast,i_\ast,l_\ast}$.
\end{Proposition}
\begin{proof}
Let $f\in \mathfrak{C}^{1,2,0}_{\{0\},\infty} \big(\mathcal{D}_{T,\infty}\big)$. The main idea is to start from the result of Proposition \ref{pr: martingale-property-extension} and to regularize $f$ by convolution on the extent domain $\R \times (\R \times [I])\times \R$. Due to the fact that we have to ensure continuity at the junction vertex $\bf 0$, we need to introduce a special regularization that is not standard. We begin to establish some convergence properties for this special regularization. Afterwards,  starting from the martingale property ($\big(\mathcal{S}_{pi}-\mathcal{M}_{ar}\big)-(iii)$) and using the dominated convergence theorem together with the {\it non-stickiness} condition given in proposition\ref{prop:non-stickness}, we show that these properties are enough to ensure the convergence at the level of martingales. 

{\bf Step 1} (Deterministic regularization procedure)

Let $\varepsilon>0$. Let us introduce $\rho_\varepsilon$ an infinite differentiable kernel converging weakly to the Dirac mass at $0$ on $\R^3$ in the sense of distribution, with $\ds \int_{\R^3}\rho_\varepsilon =1$ and compact support $[-\varepsilon,\varepsilon]^3$. For instance: $\forall \varepsilon>0$, $\forall (t,x,z)\in \R^3$:
\begin{align*}
&\nonumber~~\rho_\varepsilon(t,x,z)=\begin{cases}
\ds \frac{\varepsilon ^3 \mathcal{I}}{\exp\big(\|(\varepsilon t,\varepsilon x,\varepsilon z)\|_{\R^3}^2\big)-1},~~\text{if}~~\|(\varepsilon t,\varepsilon x,\varepsilon z)\|_{\R^3}\leq 1,\\
0,~~\text{else}
\end{cases},\\
& \text{with:}~~\ds \mathcal{I}=\int_{\R^3}\ds \frac{1}{\exp\big(\|(t, x, z)\|_{\R^3,}^2\big)-1}\mathbf{1}_{\{\|(t,x, z)\|_{\R^3}\leq 1\}}dtdxdz.~~~~~
\end{align*}
We set in the sequel, for all $i\in [I]$ and for all $(t,x,l)\in  [0,T]\times [0,+\infty)^2$:
\begin{align}\label{eq regul solu EDP}
\nonumber &f_i^{\varepsilon}(t,x,l):=\ds \int_{\R^3}\overline{f}_i(t-s,0,l-z)\rho_\varepsilon(s,y,z)dsdydz\\
&+\ds \int_{\R^3}\Big(\int_0^x\partial_x\overline{f}_i(t-s,u-y,l-z)du\Big)\rho_\varepsilon(s,y,z)dsdydz,~~
\end{align}
where $\overline{f}_i:\R^3\to \R$ is the 'symmetrization of $f_i$' given by:
\begin{align*}
\forall (s,y,z)\in \R^3,\;\;\;&\overline{f}_i:=f_i(\psi_T(s),|y|,|z|),\\
&\text{and}~~\psi_T(s)=|s|\mathbf{1}_{\{|s|\leq T\}}+T\mathbf{1}_{\{|s|> T\}}. 
\end{align*}
Observe that
\begin{align}
\label{eq:develop-regul}
&\int_{\R^3}\Big(\int_0^x\partial_x\overline{f}_i(t-s,u-y,l-z)du\Big)\rho_\varepsilon(s,y,z)dsdydz\nonumber\\
&= \int_{\R^3} \pare{\overline{f}_i(t-s,x-y,l-z) - \overline{f}_i(t-s,-y,l-z)}\rho_\varepsilon(s,y,z)dsdydz,
\end{align}
the classical arguments ensure $f_i^{\varepsilon}\in \mathcal{C}^{\infty}_b\big([0,T]\times [0,+\infty)^2,\R\big)$, for all $i\in [I]$. 

We fix now $\delta>0$ and $\theta>0$ two other parameters large enough from $\varepsilon$, ($\delta>>\varepsilon,~\theta>>\varepsilon )$, in order to obtain the following inclusion: 
$$\forall (t,l)\in [\delta,T-\delta]\times [\theta,+\infty),~~\big\{(s,z)\in \R^2~;~|t-s|\leq \varepsilon,~|l-z|\leq \varepsilon\big\}\subset [0,T]\times [0,+\infty). $$
This last property - together with the continuity of $f(t,.,l)$ at the junction point $\bf 0$ - allows us to state that 

$\forall (t,l)\in [\delta,T-\delta]\times [\theta,+\infty)$, $\forall (i,j)\in [I]^2$:
\begin{eqnarray*}
&\ds f_i^{\varepsilon}(t,0,l)=\int_{\R^3}f_i(t-s,0,l)\rho_\varepsilon(s,y,z)ds=\int_{\R^3}f_j(t-s,0,l)\rho_\varepsilon(s,y,z)ds=f_j^{\varepsilon}(t,0,l),
\end{eqnarray*}
which implies that $f^{\varepsilon}\in \mathcal{C}^{\infty}_b\big([\delta,T-\delta]\times \mathcal{J}\times [\theta,+\infty),\R\big)$ and $f^\varepsilon(t,.,l)$ is continuous at $\bf 0$ whenever $(t,l)\in [\delta,T-\delta]\times [\theta, \infty)$. In the sequel we will use the following convergence properties as $\varepsilon \xrightarrow{} 0$: 
\begin{align}
\nonumber&\forall\zeta>\theta,~~\forall i\in [I]:\\
&\label{conv unifo en 0}
\lim_{\varepsilon\to 0} \|f^{\varepsilon}(\cdot,0,\cdot)-f(\cdot,0,\cdot)\|_{\mathcal{C}^{0,1}([\delta,T-\delta]\times [\theta,\zeta])}=0,\\
&\label{conv unifo global}
\lim_{\varepsilon\to 0} \|f_i^{\varepsilon}-f_i\|_{\mathcal{C}^{0,1,0}([\delta,T-\delta]\times[0,\zeta]\times[\theta,\zeta])}=0,
\\
&\label{conv unifo laplac}
\lim_{\varepsilon\to 0} \|\partial_x^2f_i^{\varepsilon}-\partial_x^2f_i\|_{\mathcal{C}^{0,0,0}([\delta,T-\delta]\times[\theta,\zeta]\times[\theta,\zeta])}=0,\\
&\label{conv unifo deriv temps}
\forall (t,x,l)\in [\delta,T-\delta]\times[\theta,\zeta]\times[\theta,\zeta],~~\lim_{\varepsilon\to 0} \partial_tf_i^{\varepsilon}(t,x,l)-\partial_tf_i(t,x,l)=0
\end{align}
that we are going to prove in the upcoming lines. Recall that $f$ belongs to the class $ \mathfrak{C}^{1,2,0}_{\{0\},b} \big(\mathcal{D}_{T,\infty}\big)$ given in Definition \ref{def: class regula EDP}. 

Fix $\zeta>\theta$. Our choices of the parameters $\delta$ and $\theta$ imply that 

$\forall (t,x,l)\in [\delta,T-\delta]\times [\theta,\zeta]^2:$
$$~~\Big\{~(s,y,z)\in \R^3;~~|t-s|\leq \varepsilon,~~|x-y|\leq \varepsilon,~~|l-z|\leq \varepsilon~\Big\}\subset [0,T]\times [0,+\infty)^2.$$
Thus, we ensure that $\overline{f}=f$ in all the integral terms appearing in the definition of the $f_i^\varepsilon$ in \eqref{eq regul solu EDP}. Since $f(\cdot,0,\cdot)\in \mathcal{C}^{0,1}([\delta,T-\delta]\times [\theta,\zeta])$ and $\partial_x^2f_i\in \mathcal{C}^{0,0,0}([\delta,T-\delta]\times[\theta,\zeta]\times[\theta,\zeta])$, we obtain using classical arguments that \eqref{conv unifo en 0} and \eqref{conv unifo laplac} hold true.  

The convergence \eqref{conv unifo global} is more involved: note that it involves a $\mathcal{C}^1$ convergence in the second coordinate (spatial) and up to $0$. 
Because \eqref{conv unifo en 0} holds true, to obtain \eqref{conv unifo global}, it is enough to show that:
\begin{eqnarray*}
\forall i\in [I],~~\lim_{\varepsilon\to 0} \|\partial_xf_i^{\varepsilon}-\partial_xf_i\|_{C^{0,0,0}([\delta,T-\delta]\times[0,\zeta]\times[\theta,\zeta])}=0.
\end{eqnarray*}
For this purpose, we use the definition of $\overline{f}$ and once again our choices of $\delta$ and $\theta$ to obtain on the domain $\mathcal{O}:=[\delta,T-\delta]\times[0,\zeta]\times[\theta,\zeta]$:
\begin{align*}
&\|\partial_xf_i^{\varepsilon}(t,x,l)-\partial_xf_i(t,x,l)\|_{\mathcal{C}^{0}(\mathcal{O})}\\
&=\|\ds \int_{\R^3}\Big(\partial_x\overline{f}_i(t-s,x-y,l-z)-\partial_xf_i(t,x,l)\Big)\rho_\varepsilon(s,y,z)dsdydz\|_{\mathcal{C}^{0}(\mathcal{O})}\\
&\leq \|\ds \int_{\big\{(x,y,z)\in \R^3;\;\;0\leq x-y\leq \varepsilon \big\}}\Big(\partial_xf_i(t-s,x-y,l-z)-\partial_xf_i(t,x,l)\Big)\rho_\varepsilon(s,y,z)dsdydz\|_{\mathcal{C}^{0}(\mathcal{O})}\\
&\;\;\;+\|\ds \int_{\big\{(x,y,z)\in \R^3;\;\;-\varepsilon\leq x-y\leq 0 \big\}}\Big(\partial_xf_i(t-s,y-x,l-z)-\partial_xf_i(t,x,l)\Big)\rho_\varepsilon(s,y,z)dsdydz\|_{\mathcal{C}^{0}(\mathcal{O}))}\\
&\leq 2\|\sup|\partial_xf_i(t-s,x-y,l-z)-\partial_xf_i(t,x,l)|\|_{\mathcal{C}^{0}(\mathcal{O})}
\end{align*}
where the supreme is taken over the set $\{(s,y,z)\in \R^3;~|t-s|\leq \varepsilon,|x-y|\leq \varepsilon,|l-z|\leq \varepsilon\}$.

Because $\partial_xf_i\in C^{0,0,0}([\delta,T-\delta]\times[0,\zeta]\times[\theta,\zeta])$ $(\textrm{for any } i\in [I])$, by uniform continuity, we see that the last term above tends to $0$ as $\varepsilon \xrightarrow{}0$, which yields \eqref{conv unifo global}. 

Finally, for the convergence of the time derivative \eqref{conv unifo deriv temps}, fix $i\in [I]$ and $(t,x,l)\in [\delta,T-\delta]\times[\theta,\zeta]\times[\delta,\zeta]$. Since $f\in W^{1,\infty}\big([\delta,T-\delta]\times[0,\zeta]\times[\theta,\zeta]\big)$ we have from the definition \eqref{eq regul solu EDP} and \eqref{eq:develop-regul}: 
\begin{align*}&\partial_tf_i^{\varepsilon}(t,x,l)=\ds\int_{\R^3}\Big(\partial_tf(t-s,0,l-z)-\partial_tf_i(t,|y|,l-z)\Big)\rho_\varepsilon(s,y,z)dsdydz\\
&\ds + \int_{\R^3}\partial_tf_i(t-s,x-y,l-z)\rho_\varepsilon(s,y,z)dsdydz.
\end{align*}
When $\varepsilon \to 0$, using the weak convergence of $\rho_\epsilon$ to the Dirac measure $\delta_{{}_{0_{\R^3}}}$, we see that the first integral tends to $0$, whereas the second term converges to $\partial_tf_i(t,x,l)$ (because $\partial_t f_i$ is continuous on the domain $[\delta,T-\delta]\times[\theta,\zeta]\times[\theta,\zeta]$). In conclusion \eqref{conv unifo deriv temps} holds true. 

We conclude that all the properties \eqref{conv unifo en 0}, \eqref{conv unifo global}, \eqref{conv unifo laplac}, \eqref{conv unifo deriv temps} hold true.

\medskip

{\bf Step 2}

Choose moreover $\zeta>(x_\ast\vee l_\ast)\wedge \theta$. Keep in mind that $\theta$ and $\delta$ are meant to be small parameters that will be sent to zero, whereas $\zeta>(x_\ast\vee l_\ast)\wedge \theta$ is designed to tend to infinity.

Let us associate to $\theta$ and $\zeta$ the following stopping times $\tau^\theta$  and $\tau^\zeta$ defined by:
\begin{equation}
\label{def:temps-arret}
\tau^\theta:=\inf\big\{t\ge \delta,~~l(t)\ge \theta\big\},~~\tau^\zeta:=\inf\big\{t\ge \delta,~~x(t)\wedge l(t)\ge \zeta\big\}
\end{equation}
with the convention that $\inf(\emptyset) := \delta$.

Recall that $f^{\varepsilon}\in \mathcal{C}^{\infty}_b\big([\delta,T-\delta]\times \mathcal{J}\times [\theta,+\infty),\R\big)$. Hence, from the result of Proposition \ref{pr: martingale-property-extension}, the following process: 
\begin{align}\label{eq : martin tronqu} 
\nonumber &\Bigg(\mathcal{M}_t^{\delta,\theta,\zeta}\big(f^{\varepsilon}\big):=f^{\varepsilon}_{i(t\wedge \tau^\zeta)}(t\wedge \tau^\zeta,x(t\wedge \tau^\zeta),l(t\wedge \tau^\zeta)) - f^{\varepsilon}_{i(t\wedge \tau^\theta)}(t\wedge \tau^\theta,x(t\wedge \tau^\theta),l(t\wedge \tau^\theta))-\\
\nonumber &\displaystyle\int_{t\wedge \tau^\theta}^{t\wedge \tau^\zeta}\Big(\partial_tf^{\varepsilon}_{i(u)}(u,x(u),l(u))+\displaystyle\frac{1}{2}\sigma_{i(u)}^2(u,x(u),l(u))\partial_{x}^2f^{\varepsilon}_{i(u)}(u,x(u),l(u))\\
&\nonumber + b_{i(u)}(u,x(u),l(u))\partial_xf^{\varepsilon}_{i(u)}(u,x(u),l(u))\Big)du\\
&-\ds\int_{t\wedge \tau^\theta}^{t\wedge \tau^\zeta}\big(\partial_lf^{\varepsilon}(u,0,l(u))+\displaystyle\sum_{j=0}^{I}\alpha_j(u,l(u))\partial_{x}f_{i}^{\varepsilon}(u,0,l(u))\big)dl(u)\Bigg)_{t\in[\delta,T-\delta]},
\end{align} 
is a $(\Psi_t)_{\delta\leq t \leq T-\delta}$ martingale under $\P^{x_\ast,i_\ast,l_\ast}$.

To conclude this proof, we are going now to show that the martingale property given in the statement of this proposition holds true for any $f\in \mathfrak{C}^{1,2,0}_{\{0\},\infty} \big(\mathcal{D}_{T,\infty}\big)$, namely: 
\begin{eqnarray}\label{eq: martingal proper}
\E^{\P^{x_\ast,i_\ast,l_\ast}}\big[\mathcal{M}_t \big(f\big)|{\Psi}_s\big]=\mathcal{M}_s \big(f\big)~~\P^{x_\ast,i_\ast,l_\ast}-\text{a.s},\;\;\;\;\forall (t,s)\in [0,T]^2,~~t\ge s,
\end{eqnarray}
where the process $\Big(\mathcal{M}_t \big(f\big)\Big)_{t\in [0,T]}$ is given by:
\begin{align}\label{eq : martin} 
\nonumber &\Bigg(\mathcal{M}_t \big(f\big):=f_{i(t)}(t,x(t),l(t)) - f_{i_\ast}(0,x_\ast,l_\ast)-\\
\nonumber &\displaystyle\int_{0}^{t}\Big(\partial_tf_{i(u)}(u,x(u),\ell(u))+\displaystyle\frac{1}{2}\sigma_{i(u)}^2(u,x(u),l(u))\partial_{x}^2f_{i(u)}(u,x(u),l(u))\\
&\nonumber +b_{i(u)}(u,x(u),l(u))\partial_xf_{i(u)}(u,x(u),l(u))\Big)du\\
&-\ds\int_{0}^{t}\Big(\partial_lf(u,0,l(u))+\displaystyle\sum_{j=0}^{I}\alpha_j(u,l(u))\partial_{x}f_{i}(u,0,l(u))\Big)dl(u)\Bigg)_{t\in[0,T]},~~\P^{x_\ast,i_\ast,l_\ast}-\text{a.s}.
\end{align}
Recall that 
\begin{eqnarray}
\label{eq:martingale-regul-prop}
\E^{\P^{x_\ast,i_\ast,l_\ast}}\big[\mathcal{M}_t^{\delta,\theta,\zeta}\big(f^{\varepsilon}\big)|{\Psi}_{s}\big]=\mathcal{M}_s^{\delta,\theta,\zeta}\big(f^{\varepsilon}\big),~~\P^{x_\ast,i_\ast,l_\ast}-\text{a.s.}
\end{eqnarray}
for any $(t,s)\in [\delta,T-\delta]^2$ with $t\ge s$.

Let us introduce the following decomposition
\begin{eqnarray}
\label{eq:decomposition-mart-regul}
 \forall t\in [\delta,T-\delta],~~\mathcal{M}_t^{\delta,\theta,\zeta}\big(f^{\varepsilon}\big)=\mathcal{I}_{1,t}^{\delta,\theta,\zeta}\big(f^{\varepsilon}\big)+\mathcal{I}_{2,t}^{\delta,\theta,\zeta}\big(f^{\varepsilon}\big)+\mathcal{I}_{3,t}^{\delta,\theta,\zeta}\big(f^{\varepsilon}\big),~~\P^{x_\ast,i_\ast,l_\ast}-\text{a.s.}
\end{eqnarray}
where $\forall t\in [\delta,T-\delta]$:
\begin{align*}
&\mathcal{I}_{1,t}^{\delta,\theta,\zeta}\big(f^{\varepsilon}\big)=f^{\varepsilon}_{i(t\wedge \tau^\zeta)}(t\wedge \tau^\zeta,x(t\wedge \tau^\zeta),l(t\wedge \tau^\zeta)) - f^{\varepsilon}_{i(t\wedge \tau^\theta)}(t\wedge \tau^\theta,x(t\wedge \tau^\theta),l(t\wedge \tau^\theta)),\\
&\mathcal{I}_{2,t}^{\delta,\theta,\zeta}\big(f^{\varepsilon}\big)=\displaystyle\int_{t\wedge \tau^\theta}^{t\wedge \tau^\zeta}\Big(\partial_tf^{\varepsilon}_{i(u)}(u,x(u),l(u))+\displaystyle\frac{1}{2}\sigma_{i(u)}^2(u,x(u),l(u))\partial_{x}^2f^{\varepsilon}_{i(u)}(u,x(u),l(u))+\\
 &\hspace{2,4 cm}+b_{i(u)}(u,x(u),l(u))\partial_xf^{\varepsilon}_{i(u)}(u,x(u),l(u)\Big)du,\\
&\mathcal{I}_{3,t}^{\delta,\theta,\zeta}\big(f^{\varepsilon}\big)=\ds\int_{t\wedge \tau^\theta}^{t\wedge \tau^\zeta}\Big(\partial_lf^{\varepsilon}(u,0,l(u))+\displaystyle\sum_{j=0}^{I}\alpha_j(u,l(u))\partial_{x}f_{i}^{\varepsilon}(u,0,l(u))\Big)dl(u),~~~\P^{x_\ast,i_\ast,l_\ast}~~\text{a.s.}
\end{align*}
We show first that:
\begin{eqnarray}\label{eq premier cvg}
 \forall t\in [\delta,T-\delta],~~\lim_{\theta \to 0} \limsup_{\varepsilon \to 0}\E^{\P^{x_\ast,i_\ast,l_\ast}}\Big[|\mathcal{M}_t^{\delta,\theta,\zeta}\big(f^{\varepsilon}\big)-\mathcal{M}_t^{\delta,\zeta}\big(f\big)|\Big]=0,
\end{eqnarray}
where:
\begin{align}\label{eq : martin tronqu-2} 
\nonumber &\Bigg(\mathcal{M}_t^{\delta,\zeta}\big(f\big):=f_{i(t\wedge \tau^\zeta)}(t\wedge \tau^\zeta,x(t\wedge \tau^\zeta),l(t\wedge \tau^\zeta)) - f_{i(\delta)}(\delta,x(\delta),l(\delta))-\\
\nonumber &\displaystyle\int_{\delta}^{t\wedge \tau^\zeta}\big(\partial_tf_{i(u)}(u,x(u),l(u))+\displaystyle\frac{1}{2}\sigma_{i(u)}^2(u,x(u),l(u))\partial_{x}^2f_{i(u)}(u,x(u),l(u))\\
\nonumber &+b_{i(u)}(u,x(u),l(u))\partial_xf_{i(u)}(u,x(u),l(u))\big)du\\
&-\ds\int_{\delta}^{t\wedge \tau^\zeta}\big(\partial_lf(u,0,l(u))+\displaystyle\sum_{j=0}^{I}\alpha_j(u,l(u))\partial_{x}f_{i}^{\varepsilon}(u,0,l(u))\big)dl(u)\Bigg)_{t\in[\delta,T-\delta]}.
\end{align} 
(Observe that the limit in \eqref{eq premier cvg} is legitimate since we only assumed $\theta>>\varepsilon$ in the previous).

In order to prove \eqref{eq premier cvg}, we use the decomposition \eqref{eq:decomposition-mart-regul}. 

Using transparent notations, because of the uniform convergences given in \eqref{conv unifo en 0} \eqref{conv unifo global} the term $\mathcal{I}_{1,t}^{\delta,\theta,\zeta}\big(f^{\varepsilon}\big)$ (resp. $\mathcal{I}_{3,t}^{\delta,\theta,\zeta}\big(f^{\varepsilon}\big)$) is easily shown to converge in $L^{1}(\P^{i_\ast,x_\ast,l_\ast})$ to its corresponding limit $\mathcal{I}_{1,t}^{\delta,\theta, \zeta}\big(f\big)$ (resp. $\mathcal{I}_{3,t}^{\delta,\theta, \zeta}\big(f\big)$) as $\varepsilon$ tends to $0$. Next, from the definition of $\tau^\theta$ (given in \eqref{def:temps-arret}) and the continuity of $t\mapsto l(t)$, we have that $\tau^\theta \xrightarrow[\theta\searrow 0]{\P^{i_\ast,x_\ast,l_\ast}-\text{a.s.}} \delta$, which implies that $t\wedge \tau^\theta$ converges almost surely to $\delta$ since $t\in [\delta, T-\delta]$. In turn, the term $\mathcal{I}_{1,t}^{\delta,\theta,\zeta}\big(f\big)$ (resp. $\mathcal{I}_{3,t}^{\delta,\theta,\zeta}\big(f\big)$) is easily shown to converge in $L^{1}(\P^{i_\ast,x_\ast,l_\ast})$ to its corresponding limit term $\mathcal{I}_{1,t}^{\delta,\zeta}\big(f\big)$ (resp. $\mathcal{I}_{3,t}^{\delta,\zeta}\big(f\big)$) as $\theta$ tends to $0$.

Let us now turn to the term $\mathcal{I}_{2,t}^{\delta,\theta,\zeta}\big(f^{\varepsilon}\big)$. 

We write for $t\in [\delta,T-\delta]$:
\begin{align*}
&\E^{\P^{x_\ast,i_\ast,l_\ast}}\Big[\big|\mathcal{I}_{2,t}^{\delta,\theta,\zeta}\big(f^{\varepsilon}\big)-\ds\int_{t\wedge \tau^\theta}^{t\wedge \tau^\zeta}\big(\partial_tf_{i(u)}(u,x(u),l(u))+\displaystyle\frac{1}{2}\sigma_{i(u)}^2(u,x(u),l(u))\partial_{x}^2f_{i(u)}(u,x(u),l(u))+\\
&b_{i(u)}(u,x(u),l(u))\partial_xf_{i(u)}(u,x(u),l(u))\big)du)\big|\Big ] \leq C \Big( \E^{\P^{x_\ast,i_\ast,l_\ast}}\Big[\ds \int_{\delta}^{T-\delta}\mathbf{1}_{\{x(u)\leq \theta\}} du\Big]+\\
&\E^{\P^{x_\ast,i_\ast,l_\ast}}\Big[\ds \int_{t\wedge \tau^\theta}^{t\wedge \tau^\zeta}\Big(|\partial_tf^\varepsilon-\partial_tf|+|\partial_xf^\varepsilon-\partial_xf|+|\partial_x^2f^\varepsilon-\partial_x^2f|\Big)_{i(u)}(u,x(u),l(u))\mathbf{1}_{\{x(u) \ge \theta\}}ds\Big),
\end{align*}
for a uniform constant $C>0$ uniformly independent of the parameters $\delta,\theta,\zeta,\varepsilon$. We see then that \eqref{eq premier cvg} must hold true, by making use of the {\it non-stickiness} condition given in Proposition \ref{prop:non-stickness}, Lebesgue's theorem and the convergence properties established in \eqref{conv unifo global}, \eqref{conv unifo laplac} and \eqref{conv unifo deriv temps}.

 Fix $s,t\in [\delta, T-\delta]$ with $s\leq t$. From \eqref{eq premier cvg} (applied to $s$ and $t$) and passing to the $\P^{x_\ast,i_\ast,l_\ast}-\text{a.s.}$ limit in \eqref{eq:martingale-regul-prop} (using sub sequences denoted abusively by $\varepsilon$ and $\theta$), we have the following $\P^{x_\ast,i_\ast,l_\ast}-\text{a.s.}$ equality:
\begin{eqnarray*}
 \E^{\P^{x_\ast,i_\ast,l_\ast}}\big[\mathcal{M}_t^{\delta,\zeta}\big(f\big)|{\Psi}_{s}\big]=\mathcal{M}_s^{\delta,,\zeta}\big(f\big),~~\P^{x_\ast,i_\ast,l_\ast}~~\text{a.s.}
\end{eqnarray*}

To yield finally \eqref{eq: martingal proper} and conclude the proof, we send in the last equation $\zeta \to +\infty$ and $\delta \to 0$.  Applying Lebesgue's dominated convergence theorem, we see that the result of the proposition follows from the continuity of the process $\pare{\mathcal{M}_u\big(f\big)}_{u\in [0,T]}$.
\end{proof}
As it is classical in the literature for the characterization of the uniqueness in distribution of stochastic processes, we introduce a determining class functional space that is adapted to our purpose.

\begin{Lemma}\label{lm : class determining }
Denote $\mathcal{D}_{\infty}:=\mathcal{J}\times[0,+\infty)$. 

Let $\P$ and $\Q$ be two probability measures defined on the measurable space $\pare{\mathcal{D}_{\infty},\mathbb{B}\big(\mathcal{D}_{\infty}\big)}$ satisfying:
\begin{equation}
\label{eq:P-Q-jonction}
\forall\,V\textrm{open in }\B([0,\infty)),\;\;\;\P(\{{\bf 0}\}\times V) = \Q(\{{\bf 0}\}\times V) = 0.
\end{equation} 
Fix any $\mathcal{T}\in (0,T]$. Assume that for any $g^\mathcal{T}\in \mathcal{C}^{\infty}_b\big(\mathcal{D}_{\infty}\big)$ satisfying the following compatibility condition
\begin{align}
\label{eq:compatitbility-star}
&\partial_lg^\mathcal{T}(0,l)+\sum_{i=1}^I\alpha_i(\mathcal{T},l)\partial_xg_i^\mathcal{T}(0,l)=0,~~l\in[0,+\infty),\;\;\;\;(\ast)
\end{align}
we have:
$$\int_{\big((x,i),l\big)\in\mathcal{D}_{\infty}}g^\mathcal{T}_i(x,l)d\P=\int_{\big((x,i),l\big)\in\mathcal{D}_{\infty}}g^\mathcal{T}_i(x,l)d\Q.$$
Then, $\P=\Q$. 
In particular the class
\begin{equation}
\label{eq:determining-class}
\mathcal{G}:=\left \{g^\mathcal{T}\in \mathcal{C}^{\infty}_b\big(\mathcal{D}_{\infty}\big)~|~g^\mathcal{T}\,\,\textrm{satisfies }\,(\ast) \right \}
\end{equation}
is a determining class over the set of probabilities on $\pare{\mathcal{D}_{\infty},\mathbb{B}\big(\mathcal{D}_{\infty}\big)}$
\end{Lemma}
\begin{proof}
To obtain the result, it is enough to show that $\P$ and $\Q$ are equal for any open sets $U$ of the Polish space $\mathcal{D}_{\infty}$, endowed with the following metric:
$$\forall \Big(((x,i),l\big),((y,j),q\big)\Big)\in \mathcal{D}_{\infty}^2,~~d\Big(((x,i),l\big),((y,j),q\big)\Big)=d^{\mathcal{J}}\big((x,i),(y,j)\big)+|l-q|_{\R}.$$
Fix $U$ a topological set of $\mathcal{D}_{\infty}=\mathcal{J}\times[0,+\infty)$. Let $n\in \N^*$ and define the following continuous map $f_n^U$ by:
\begin{equation}
\label{eq:f-n-U}
f_n^U:=
\begin{cases}
   \mathcal{D}_{\infty}\to \R \\
   \big((x,i),l\big)=[nd\Big(\big((x,i),l\big),U^c\Big)]\wedge 1 := f_{i,n}^U(x,l)
\end{cases}
\end{equation}
where:
$$d\Big(\big((x,i),l\big),U^c\Big)=\inf \Big\{~d\Big(\big((x,i),l\big),\big((y,j),q\big)\Big),~((y,j),q\big)\in U^c~\Big\}.$$
Let $\varepsilon>0$. We introduce first $\rho_\varepsilon$ an infinite differentiable kernel converging weakly to the Dirac mass at $0$ on $\R^2$ in the sense of distribution, with $\ds \int_{\R^2}\rho_\varepsilon =1$, and compact support $[-\varepsilon,\varepsilon]^2$. 
Define now for all $i\in [I]$ the following map $f_{i,n,\varepsilon}^U$ on $[0,+\infty)^2$ such that $\forall (x,l)\in [0,+\infty)^2$:
$$f_{i,n,\varepsilon}^U(x,l)=\int_{\R^2}f_{i,n}^U(|y|,|z|) \rho_\varepsilon (x-y,l-z) dy dz$$
(using the notation \eqref{eq:f-n-U}).
We obtain clearly that $f_{i,n,\varepsilon}^U \in \mathcal{C}^\infty_b\big([0,+\infty)^2,\R\big)$ and converges locally uniformly on the compact sets of $[0,+\infty)^2$ to $f_{i,n}^U$ as $\varepsilon \to 0$ (see the arguments used to prove \eqref{conv unifo en 0} in Proposition \ref{pr : weak Ito formula}).
Now fix $\delta>0$ a parameter that we will be chosen large enough from $\varepsilon$ and introduce the following set:
$$U^\delta =\Big\{\big((x,i),l\big)\in U,~x\ge \delta\Big\}.$$
Remark now that as soon as we choose $\delta>3\varepsilon$, for any $x\in [0,\varepsilon/2]$ we have that:

$\forall w \in \{y\in \R,~|y-x|\leq \varepsilon\}$:  $$w\leq x+\varepsilon \leq 3\varepsilon/2< \delta/2.$$ 
This implies that: 
$\forall x\in [0,\varepsilon/2],\;\forall l\in [0,+\infty), \;\forall i\in [I],\;\;\;f_{i,n,\varepsilon}^{U^\delta} (x,l)=0.$
Therefore, the corresponding map $f_{n,\varepsilon}^{U^\delta}$ is in the class $\mathcal{C}^{\infty}_b\big(\mathcal{D}_{\infty}\big)$ and satisfies the compatibility condition:
\begin{align*}
&\partial_lf_{n,\varepsilon}^{U^\delta}(0,l)+\sum_{i=1}^I\alpha_i(\mathcal{T},l)\partial_xf_{i,n,\varepsilon}^{U^\delta}(0,l)=0,~~l\in[0,+\infty),
\end{align*}
for any fixed $\mathcal{T}\in (0,T]$.
Hence, by assumption, we have:
$$\int_{((x,i),l)\in\mathcal{D}_{\infty}}f_{i,n,\varepsilon}^{U^\delta}(x,l)d\P=\int_{((x,i),l)\in\mathcal{D}_{\infty}}f_{i,n,\varepsilon}^{U^\delta}(x,l)d\Q.$$
Sending first $\varepsilon \to 0$ we obtain:
$$\int_{((x,i),l)\in\mathcal{D}_{\infty}}f_{i,n}^{U^\delta}(x,l)d\P=\int_{((x,i),l)\in\mathcal{D}_{\infty}}f_{i,n}^{U^\delta}(x,l)d\Q.$$
As $f_{i,n}^{U^\delta}$ converges pointwise to $\mathbf{1}_{U^\delta}$ when $n\to +\infty$, we deduce from the last inequality that:
$$\forall \delta>0,~~\P(U^\delta)=\Q(U^\delta),$$
and finally as $\delta \to 0$, by monotone convergence, we obtain that
$$\P(U^\ast)=\Q(U^\ast)$$
where $U^\ast:=\{(x,i,l)\in U,\;(x,i)={\bf 0}\}$. In turn, from \eqref{eq:P-Q-jonction}, we obtain that $$\P(U)=\Q(U)$$
holds for any open set $U$ of the Polish space $D_\infty$. This concludes the proof.
\end{proof}

\subsection{Proof of Theorem \ref{th: exist Spider } II - Uniqueness}
\begin{proof}{\it{Theorem \ref{th: exist Spider } II - Uniqueness}}

Since the canonical space $\Phi$ is a Polish space, to ensure the uniqueness for the martingale problem \emph{$\big(\mathcal{S}_{pi}-\mathcal{M}_{ar}\big)$}, it is enough to show  that if two probability measures $\P^{x_\ast,i_\ast,l_\ast}$ and $\Q^{x_\ast,i_\ast,l_\ast}$ are solutions of \emph{$\big(\mathcal{S}_{pi}-\mathcal{M}_{ar}\big)$}, then we have for all $n\in \N^*$, for all $(t_1,\ldots,t_n)\in [0,T]^n$, and for all $(g^1\ldots g^n)$ maps belonging to a class of determining functions:
\begin{align}\label{eq: egalite proba marginal}
\nonumber &\E^{\P^{x_\ast,i_\ast,l_\ast}}\big[g^1_{i(t_1)}\big(x(t_1),l(t_1)\big)\times \ldots\times g^n_{i(t_n)}\big(x(t_n),l(t_n)\big) \big]=\\
&\E^{\Q^{x_\ast,i_\ast,l_\ast}}\big[g^1_{i(t_1)}\big(x(t_1),l(t_1)\big)\times \ldots\times g^n_{i(t_n)}\big(x(t_n),l(t_n)\big)\big].\end{align}
We will use in the sequel the class $\mathcal{G}$ of determining functions introduced in Lemma \ref{lm : class determining } (see \eqref{eq:determining-class}). Indeed, from the result of Lemma \ref{lm : class determining } and using the same arguments given in Theorem 6.2.3 in \cite{Stroock}, we see that proving \eqref{eq: egalite proba marginal} is equivalent to show that for all $\mathcal{T}\in (0,T]$, for any $g\in \mathcal{C}^{\infty}_b\big(\mathcal{D}_{\infty}\big)$ satisfying the compatibility condition $(\ast)$ given by \eqref{eq:compatitbility-star} in Lemma \ref{lm : class determining } we have:
\begin{align}
\label{eq: zero-proba-P-Q}
&\forall\,V\textrm{open in}\;\mathbb{B}\big([0,+\infty)\big),\nonumber\\
&\P^{x_\ast,i_\ast,l_\ast}\pare{(x(\mathcal{T}),i(\mathcal{T}),l(\mathcal{T}))\in \{{\bf 0}\}\times V} = \Q^{x_\ast,i_\ast,l_\ast}\pare{(x(\mathcal{T}),i(\mathcal{T}),l(\mathcal{T}))\in \{{\bf 0}\}\times V} = 0,
\end{align}
and
\begin{eqnarray}\label{eq: egalite proba marginal rafinee}
\E^{\P^{x_\ast,i_\ast,l_\ast}}\big[g_{i(\mathcal{T})}\big(x(\mathcal{T}),l(\mathcal{T})\big) \big]=\E^{\Q^{x_\ast,i_\ast,l_\ast}}\big[g_{i(\mathcal{T})}\big(x(\mathcal{T}),l(\mathcal{T})\big)\big]\end{eqnarray}
Note that \eqref{eq: zero-proba-P-Q} is a direct consequence of the non stickiness property of Proposition \ref{prop:non-stickness} and the continuity of the process $(x(t))_{t\in [0,T]}$.

Denote by $u$ the unique solution of the following backward parabolic system
\begin{eqnarray*}
\begin{cases}\partial_tu_i(t,x,l)+\ds \frac{1}{2}\sigma^2_i(t,x,l)\partial_x^2u_i(t,x,l)\\
\hspace{2,2 cm}+b_i(t,x,l)\partial_xu_i(t,x,l)=0,~~(t,x,l)\in (0,\mathcal{T})\times (0,+\infty)^2,\\
\partial_lu(t,0,l)+\displaystyle \sum_{i=1}^I \alpha_i(t,l)\partial_xu_i(t,0,l)=0,~~(t,l)\in(0,\mathcal{T})\times(0,+\infty)\\
\forall (i,j)\in[I]^2,~~u_i(t,0,l)=u_j(t,0,l)=u(t,0,l),~~(t,l)\in[0,\mathcal{T}]\times[0,+\infty)^2,\\
\forall i\in[I],~~ u_i(\mathcal{T},x,l)=g_i(x,l),~~(x,l)\in[0,+\infty)^2,
\end{cases}
\end{eqnarray*}
which, in view of Theorem \ref{th : exis para with l bord infini}, is uniquely solvable in the class $\mathfrak{C}^{1,2,0}_{\{0\},\infty} \big(\mathcal{D}_{\mathcal{T},\infty}\big)$. Now, applying the result of Proposition \ref{pr : weak Ito formula} with the solution $u\in \mathfrak{C}^{1,2,0}_{\{0\},\infty} \big(\mathcal{D}_{\mathcal{T},\infty}\big)$ as a test function, we have that the martingale property holds true between time $t=0$ and time $t=\mathcal{T}$. Using the PDE system satisfied by the solution $u$, this leads to get 
\begin{align*}\nonumber 
\E^{\P^{x_\ast,i_\ast,l_\ast}}\big[u_{i(\mathcal{T})}\big(x(\mathcal{T}),l(\mathcal{T}\big)\big] = u_{i_\ast}(0,x_\ast,l_\ast) = \E^{\P^{x_\ast,i_\ast,l_\ast}}\big[g_{i(\mathcal{T})}\big(x(\mathcal{T}),l(\mathcal{T})\big)\big]
\end{align*}
for the probability measure $\P^{x_\ast,i_\ast,l_\ast}$.

With the same arguments for the probability measure $\Q^{x_\ast,i_\ast,l_\ast}$, we have similarly that
\begin{eqnarray*}\nonumber \E^{\Q^{x_\ast,i_\ast,l_\ast}}\big[g_{i(\mathcal{T})}\big(x(\mathcal{T}),l(\mathcal{T})\big)\big]=\E^{\Q^{x_\ast,i_\ast,l_\ast}}\big[u_{i(\mathcal{T})}\big(x(\mathcal{T}),l(\mathcal{T}\big)\big]=u_{i_\ast}(0,x_\ast,l_\ast).\end{eqnarray*}
We obtain therefore that \eqref{eq: zero-proba-P-Q} and \eqref{eq: egalite proba marginal rafinee}, or equivalently \eqref{eq: egalite proba marginal}, hold true and the proof is completed.
\end{proof}

\section{Two deriving results}
\label{sec:deriving-results}
\subsection{A spider Brownian motion whose spinning measure depends on its own local time}
\label{section-spiderBM}

As an application of Theorem \ref{th: exist Spider } with the data $(\sigma_i)_{i\in \{1,\dots, I\}} \equiv 1$ and  $(b_i)_{i\in \{1,\dots, I\}}\equiv 0$, we can assert that there is weak existence of an {\it inhomogeneous Brownian spider} $\pare{S_t, i(t), \ell_t}_{t\geq 0}$ constructed on some filtered probability space
$(\Omega,\mathcal{F}, (\mathcal{F}_t)_{t \ge 0}, (\mathbb{P}^{x_\ast,i_\ast,l_\ast}))$ {\it ''whose probability of going at some branch from the junction point is driven by a spinning probability measure $t\mapsto \alpha(t, \ell_t)$ that depends almost surely both on the current running time time and its own local time spent at the junction''}. The aim of this section is to give {\it by hand} some insights on this inhomogeneous Brownian spider $\pare{(S_t, i(t)), \ell_t}_{t\geq 0}$. We believe that this elementary study shades some light on the seminal ideas that have been developed in this paper.

\subsubsection{Heuristics on the generator}
\label{subsection-generator}
Let us begin by the formal derivation of the generator of the inhomogeneous Brownian spider $\pare{\pare{S_t, i(t), \ell^0_t}_{t\geq 0}, \P}$ whose spinning probability measure depends both the current running time and its own local time at the junction. This paragraph may help the reader to understand where the {\it local time Kirchhoff's transmission condition} \eqref{eq:condition-transmission-intro} comes from.

In order to perform our computations, we will need to consider $(W_t)_{t \ge 0}$ an auxiliary standard $(\Omega,(\mathbb{P}^W_x)_{x\in \R})$ ($\R$-valued) Brownian motion.
\medskip

Let $f\in \mathfrak{C}^{1,2,0}_{\{0\},b} \big(\mathcal{D}_{T,\infty}\big)$ (see \eqref{def: class regula EDP} for the definition of this function space) : because there is no dependence on the time variable in the coefficients, let us simplify the exposition and assume that $f(s,.,.,.) = f(u,.,.,.)$ for any $s,u>0$ so that $f$ belong to the subset of $\mathfrak{C}^{1,2,0}_{\{0\},b} \big(\mathcal{D}_{T,\infty}\big)$ of constant function over the time variable. 

In the sequel we will write $L^0$ to denote the local time at $0$ of the auxiliary Brownian motion $W$.
\medskip

Let $t>0$ be fixed. 

{\bf First step}

Assume at first that $x_\ast > 0$. Introduce
$
\tau_0 = \inf(u\geq 0~:~S_u=0)$ and $T_0 = \inf(u\geq 0~:~W_t=0)$.
From Markov's property applied at time $t$
\begin{align*}
&\E\croc{f(S_{t+\varepsilon}, i(t+\varepsilon), \ell_{t+\varepsilon}) - f(t, S_t, i(t), \ell_t)~|~S_t = x_\ast, i(t) =i_\ast, \ell_t=l_\ast}\\
&=\E\croc{f({S}_{\varepsilon}, {i}(\varepsilon), l_\ast+ {\ell}_{\varepsilon}) - f(x_\ast, i_\ast, l_\ast)}\\
&=\E\croc{\pare{f({S}_{\varepsilon}, {i}(\varepsilon), l_\ast + {\ell}_{\varepsilon}) - f(x_\ast, i_\ast, l_\ast)}{\bf 1}_{\tau_0 > \varepsilon}}+ \E\croc{\pare{f(S_{\varepsilon}, {i}(\varepsilon), l_\ast+{\ell}_{\varepsilon}) - f(x_\ast, i_\ast, l_\ast)}{\bf 1}_{\tau_0 \leq \varepsilon}}\\&=\E^W_{x_\ast}\croc{\pare{f(|W|_{\varepsilon}, i, l_\ast) - f(x_\ast, i_\ast, l_\ast)}{\bf 1}_{T_0 > \varepsilon}}+ \E\croc{\pare{f(S_{\varepsilon}, {i}(\varepsilon), l_\ast+{\ell}_{\varepsilon}) -f(x_\ast, i_\ast, l_\ast)}{\bf 1}_{\tau_0\leq \varepsilon}}.
\end{align*}
\begin{align*}
\left .\begin{array}{r}
\ds \frac{1}{\varepsilon}\E\croc{\pare{f(|W|_{\varepsilon}, i, l_\ast) - f(x_\ast, i_\ast, l_\ast)}{\bf 1}_{T_0\leq \varepsilon}}\\
\ds \frac{1}{\varepsilon}\E\croc{\pare{f(S_{\varepsilon}, {i}(\varepsilon), l_\ast+{\ell}_{\varepsilon}) - f(x_\ast, i_\ast, l_\ast)}{\bf 1}_{\tau_0\leq \varepsilon}}
\end{array}\right \}&\leq 2|f|_\infty \frac{1}{\varepsilon}\P^W_{x_\ast}(T_0\leq \varepsilon)\\
&\xrightarrow{\varepsilon \searrow 0+} 0.
\end{align*}
So using the definition of the generator of the standard Brownian motion, we get
\begin{align*}
&\E\croc{f(S_{t+\varepsilon}, i(t+\varepsilon), \ell_{t+\varepsilon}) - f(S_t, i(t), \ell_t)~|~S_t = x_\ast, i(t) =i_\ast, \ell_t=l_\ast}\\
&\xrightarrow{\varepsilon \searrow 0+}\frac{1}{2}\partial^2_{xx}f(t, x_\ast,i_\ast,l_\ast)\;\;\;\textrm{as expected}.
\end{align*}

{\bf Second step}

We know set $x={\bf 0}$. For $u\geq 0$, we denote $g_u = \sup(s\leq u~:~S_{u} = 0)$ and $G_u = \sup(s\leq u~:~W_{u} = 0)$. We have
\begin{align*}
&\E\croc{f(S_{t+\varepsilon}, i(t+\varepsilon), \ell_{t+\varepsilon}) - f(S_t, i(t), \ell_t)~|~S_t = {\bf 0}, i(t)=1,\dots, I, \ell_t=l_\ast}\\
&=\E^{{\bf 0},l}\croc{f({S}_{\varepsilon}, {i}(\varepsilon), l+{\ell}_{\varepsilon}) - f(0,i,l_\ast)|~S_t = {\bf 0}, i(t)=1,\dots, I, \ell_t=l_\ast}\\
&=\E^{{\bf 0},l}\croc{\sum_{j=1}^I \alpha_j(t+g_{\varepsilon},l_\ast + {\ell}_{\varepsilon})f(S_\varepsilon, j, l_\ast+{\ell}_{{\varepsilon}}) - f(0,i,l_\ast)}\\
&=\sum_{j=1}^I \E^{W}_{0}\croc{\alpha_j(t+G_{\varepsilon},l_\ast + L^0_{\varepsilon})\pare{f(|W_\varepsilon|, j, l_\ast+ L^0_{{\varepsilon}}) - f(0,i,l_\ast)}}\\
&=\sum_{j=1}^I \E^{W}_{0}\croc{\alpha_j(t+\varepsilon G_{1},l_\ast + \sqrt{\varepsilon}L^0_1)\pare{f(\sqrt{\varepsilon}|W_1|, j, l_\ast+\sqrt{\varepsilon}L^0_{1}) - f(0,j,l_\ast)}}.
\end{align*}
Hence,
\begin{align}
\label{eq:generateur-at-junction-1}
&\frac{1}{\varepsilon}\E\croc{f(S_{t+\varepsilon}, i(t+\varepsilon), \ell_{t+\varepsilon}) - f(S_t, i(t), \ell_t)~|~S_t = {\bf 0}, i(t)=1,\dots, I, \ell_t=l_\ast}\nonumber\\
&=\sum_{j=1}^I \E^{W}_{0}\croc{\alpha_j(t+\varepsilon G_{1},l_\ast + \sqrt{\varepsilon}L^0_1)\frac{1}{\varepsilon}\pare{f(\sqrt{\varepsilon}|W_1|, j, l+\sqrt{\varepsilon}L^0_{1}) - f(0,j,l_\ast)}}.
\end{align}

By the joint continuity of $\alpha$ in both variables
\begin{align*}
\alpha_j(t+\varepsilon G_{1},l_\ast + \sqrt{\varepsilon}L^0_1) = \alpha_j(t,l_\ast) + h^j_{t,l_\ast}(\varepsilon G_{1} + \sqrt{\varepsilon}L^0_1)
\end{align*}
with $|h^j_{t,l_\ast}(\varepsilon G_{1} + \sqrt{\varepsilon}L^0_1)| \leq |\overline{\alpha}|(\varepsilon G_{1}, \sqrt{\varepsilon}L^0_1)$
by $\text{Assumption } (\mathcal{H})-(iii)$. 
Using a Taylor's expansion, we have
\begin{align*}
&\frac{1}{\varepsilon}\pare{f(\sqrt{\varepsilon}|W_1|, j, l_\ast+\sqrt{\varepsilon}L^0_{1}) - f(0,j,l_\ast)}\\
&= \frac{1}{\varepsilon}
\pare{\partial_xf(0,j,l_\ast)\sqrt{\varepsilon}|W_1| + \partial_lf(0,j,l_\ast)\sqrt{\varepsilon}L^0_1}\\
&\hspace{2.0 cm}+ \frac{1}{\varepsilon}\croc{\pare{\sqrt{\varepsilon}|W_1|, \sqrt{\varepsilon}L^0_1}^{\ast}{\mathbf H}_{x,l}[f(0,j,l_\ast)]\pare{\sqrt{\varepsilon}|W_1|, \sqrt{\varepsilon}L^0_1} + o(\varepsilon |W_1|^2 + \varepsilon (L^0_1)^2)},
\end{align*}
where ${\mathbf H}_{x,l}[f(0,j,l_\ast)]$ stands for the Hessian matrix  w.r.t variables $x,l$ taken at point $(0,j,l_\ast)$ of the function $f(.,j.)$.
So that
\begin{align*}
&\alpha_j(t+\varepsilon G_{1},l + \sqrt{\varepsilon}L^0_1)\frac{1}{\varepsilon}\pare{f(\sqrt{\varepsilon}|W_1|, j, l_\ast+\sqrt{\varepsilon}L^0_{1}) - f(0,j,l_\ast)}\\
&=\alpha_j(t,l_\ast)\frac{1}{\varepsilon}
\pare{\partial_xf(0,j,l_\ast)\sqrt{\varepsilon}|W_1| + \partial_lf(0,j,l_\ast)\sqrt{\varepsilon}L_1^0}\\
&\hspace{2.0 cm}+\alpha_j(t,l_\ast)\croc{\croc{\pare{|W_1|, L^0_1}^{\ast}{\mathbf H}_{x,l}[f(0,j,l_\ast)]\pare{|W_1|, L^0_1} + o(\varepsilon |W_1|^2 + \varepsilon (L^0_1)^2)}}\\
&\hspace{2.0 cm}+h^j_{t,l_\ast}(\varepsilon G_{1} + \sqrt{\varepsilon}L^0_1)\frac{1}{\sqrt{\varepsilon}}
\pare{\partial_xf(0,j,l_\ast)|W_1| + \partial_lf(0,j,l_\ast)L^0_1}\\
&\hspace{0,7 cm}+ h^j_{t,l_\ast}(\varepsilon G_{1} + \sqrt{\varepsilon}L^0_1)\croc{\croc{\pare{|W_1|, L^0_1}^{\ast}{\mathbf H}_{x,l}[f(0,j,l_\ast)]\pare{|W_1|, L^0_1} + o(\varepsilon |W_1|^2 + \varepsilon (L^0_1)^2)}}.
\end{align*}
Remember that $|h^j_{t,l_\ast}(\varepsilon G_{1} + \sqrt{\varepsilon}L^0_1)| \leq |\overline{\alpha}|(\varepsilon G_{1}, \sqrt{\varepsilon}L^0_1)$
by $\text{Assumption } (\mathcal{H})-(iii)$. Then taking expectations with the use of the fact that $\E_0^W(|W_1|) = \E_0^W (L^0_1)$ and plugin in \eqref{eq:generateur-at-junction-1}, we see that there is a $\limsup$ convergence of 
\begin{align*}
\frac{1}{\varepsilon}\E\croc{f(S_{t+\varepsilon}, i(t+\varepsilon), \ell_{t+\varepsilon}) - f(S_t, i(t), \ell_t)~|~S_t = {\bf 0}, i(t)=1,\dots, I, \ell_t=l_\ast}
\end{align*} 
as $\varepsilon \searrow 0+$ if and only if
\begin{equation}
\label{eq:bc-1}
\sum_{j=1}^I \alpha_j(t,l_\ast)\pare{\partial_xf(0,j,l_\ast) + \partial_lf(0,j,l_\ast)}=0.
\end{equation}
Due to the continuity of $f$ at the junction point, note that for any $(i,j)\in \{1,\dots, I\}^2$,
\begin{align*}
\partial_lf(0,i,\ell) &= \lim_{h\rightarrow 0}\frac{f(0,i,\ell + h) - f(0,i,\ell + h)}{h}=\lim_{h\rightarrow 0}\frac{f(0,j,\ell + h) - f(0,j,\ell + h)}{h}=\partial_lf(0,j,\ell)
\end{align*}
which shows that the partial derivative w.r.t variable $\ell$ does not depend on the branch label. Since $\sum_{j=1}^I \alpha_j(t,l_\ast) = 1$, we get in turn that \eqref{eq:bc-1} reads in fact
\[
\partial_lf(0,j,l_\ast) + \sum_{j=1}^I \alpha_j(t,l_\ast)\partial_xf(0,j,l_\ast) =0
\]
giving some insight on the appearance of {\it local time Kirchhoff's transmission condition} \eqref{eq:condition-transmission-intro} and the boundary condition in the formulation of the backward system parabolic PDE \eqref{eq : pde with l bord infini Walsh} associated to our martingale problem.

\subsubsection{The transition density function}
Remember that for all $t,s>0$, the image measure $\P_0^W[|W_t|\in dx,L_t^0(W)\in d\ell,G_t\in ds]$ is given by
\begin{equation}
\label{tridens-RY}
g(x,\ell,s)ds\,d\ell\,dx = {\bf 1}_{s\leq t}\frac{2}{\sqrt{2\pi s^3}}\exp\Big( -\frac{\ell^2}{2s} \Big)\frac{x}{\sqrt{2\pi(t-s)^3}}\exp\pare{-\frac{x^2}{2(t-s)}}{\bf 1}_{x>0}{\bf 1}_{\ell>0}ds\,d\ell \,dx
\end{equation}
(see \cite{Revuz-Yor} Chapter XII, exercise (3.8)).

Using the same ideas as in \cite{Etore-Martinez}, we may state the following intuitive result:
\begin{Proposition}
    For $x\neq {\bf 0}_J$, the transition density measure of $\pare{(S_t, i(t)), \ell_t^0(S)}_{t\geq 0}$ is given by the family $\{p^{{\bf \alpha}}(s,t;(x,i,l),(y,j, l+d\ell)~|~s<t,\, \pare{(x,i,l),(y,j,l+\ell)}\in ({\J}\times [0,\infty))^2\}$ defined by
\begin{equation}
\label{transition-density-complete-0}
p^{{\bf \alpha}}(s,t;({\bf 0}_J, k),(y,j, l+d\ell))dy:=\int_0^{t-s}du\;{\alpha_j(s,k+\ell)}g(y,\ell,u)\,dy\,d\ell.
\end{equation}
and
\begin{equation}
\label{transition-density-complete}
\begin{split}
&p^{{\bf \alpha}}(s,t;(x,i,l),(y,j,l+d\ell))dy\,\,\\
&:=\int_0^{t-s} du\frac{x {\rm e}^{-{x^2}/{u}} }{\sqrt{2\pi u^3}}\pare{\int_0^{t-(s+u)}dv\,\alpha_j(s+u+v,l+\ell)\,g(y,\ell,v)}\,dy\,d\ell\\
&\hspace{0,4 cm}\,+\,{\bf 1}_{i=j}\frac{1}{\sqrt{2\pi (t-s)}}\croc{\exp\pare{-\frac{(y-x)^2}{2(t-s)}} - \exp\pare{-\frac{(y+x)^2}{2(t-s)}} }dy\delta_{\{0\}}(d\ell).
\end{split}
\end{equation}
\end{Proposition}
\begin{proof} (Main idea of the proof)
We will only give the main idea of the proof.
It suffices to check that
the family 
$$\{p^{{\bf \alpha}}(s,t;(x,i,l),(y,j, l+d\ell)~|~s<t,\, \pare{(x,i,l),(y,j,l+\ell)}\in ({\J}\times [0,\infty))^2\}$$ defined by \eqref{transition-density-complete-0} and \eqref{transition-density-complete} yields the fundamental solution of
\begin{eqnarray}\label{eq : pde-spider}
\begin{cases}\partial_tu_i(t,x,l)+\ds \frac{1}{2}\partial_x^2u_i(t,x,l)=0,~~(t,x,l)\in (0,T)\times (0,+\infty)^2,\\
\partial_lu(t,0,l)+\displaystyle \sum_{i=1}^I \alpha_i(t,l)\partial_xu_i(t,0,l)=0,~~(t,l)\in(0,\mathcal{T})\times(0,+\infty)\\
\forall (i,j)\in[I]^2,~~u_i(t,0,l)=u_j(t,0,l)=u(t,0,l),~~(t,l)\in[0,T]\times[0,+\infty)^2,\\
\forall i\in[I],~~ u_i(T,x,l)=g_i(x,l),~~(x,l)\in[0,+\infty)^2,
\end{cases}
\end{eqnarray}
where $g$ is arbitrary in $\mathcal{C}^{\infty}_b\big(\mathcal{D}_{\infty}\big)$ satisfies the compatibility condition:
\begin{align*}
&\partial_lg(0,l)+\sum_{i=1}^I\alpha_i(T,l)\partial_xg_i(0,l)=0,~~l\in[0,+\infty).
\end{align*}
Namely, we check that the unique solution of the Cauchy problem \ref{eq : pde-spider} admits the following representation
\[
u_i(t,x,l) = \sum_{j=1}^I\int_{\R^+\times \R^+}g_{j}(y,\ell)\,p^{{\bf \alpha}}(t,T;(x,i,l),(y,j,l+d\ell)\,dy,\;\;\;i\in \{1,\dots, I\}.
\]
Since this fact checking computation is tremendously long and tedious, we have decided to leave it aside in the present paper.
\end{proof}

\medskip
\subsection{The case $I=2$ - Existence and uniqueness for the solution of a skew SDE at zero whose characteristic coefficients and skewness parameter depend on the local time of the unknown process}

This subsection is devoted to the study of the special case $I=2$. 

We are interested in the analysis of the following stochastic differential equation
\begin{equation}
\label{eq:sde-I-equals-2}
y_0 = y,\hspace{0,4 cm}dy(s) = \tilde{\sigma}(s,y(s),\ell_s^0(y))dW_s + \tilde{b}(s,y(s),\ell_s^0(y))ds + \beta(s,\ell^0_s(y))d\ell_s^0(y)
\end{equation}
where
\begin{equation}
\label{eq:beta}
\beta(s,l) = \frac{\alpha(s,l)\tilde{\sigma}(s,0+,l) - (1-\alpha(s,l)) \tilde{\sigma}(s,0-,l)}{\alpha(s,l)\tilde{\sigma}(s,0+,l) + (1-\alpha(s,l)) \tilde{\sigma}(s,0-,l)}.
\end{equation}

We assume that analogous assumptions as $\textbf{Assumption } (\mathcal{H})$ hold for the coefficients $\tilde{\sigma}, \tilde{b}$ and $\beta$.

We may deduce the following result from Theorem \ref{th: exist Spider }:
\begin{Proposition}
    Under the above assumptions, if there exists a weak solution to the stochastic differential equation \eqref{eq:sde-I-equals-2}, then the probability law $${\mathcal L}\pare{(|y(t)|, \frac{\sgn(y(t)) + 3}{2}, \ell_t^0(y)}_{t\geq 0}$$ is the unique solution of the spider martingale problem \emph{$\big(\mathcal{S}_{pi}-\mathcal{M}_{ar}\big)$} with $I=2$ starting from $X(0) = (y, (\sgn(y) + 3)/{2},0)$ and data given by
    \begin{align}
    &\label{eq:chgt-variable-sigma}\sigma_{2}(s,x,l) = \tilde{\sigma}(s,x,l){\bf 1}_{x\geq 0},\;\;\;\sigma_{1}(s,x,l) = \tilde{\sigma}(s,-x,l){\bf 1}_{x\geq 0},\\
    &\label{eq:chgt-variable-b}b_{2}(s,x,l) = \tilde{b}(s,x,l){\bf 1}_{x\geq 0},\;\;\;b_{1}(s,x,l) = -\tilde{b}(s,-x,l){\bf 1}_{x\geq 0},\\
    &   \label{eq:changt-variable-coeffs-alpha1}\alpha_1(s,l) = \frac{(1-\alpha(s,l)) \tilde{\sigma}(s,0-,l)}{\alpha(s,l)\tilde{\sigma}(s,0+,l) + (1-\alpha(s,l)) \tilde{\sigma}(s,0-,l)},\\
    &   \label{eq:changt-variable-coeffs-alpha2} \alpha_2(s,l) = \frac{\alpha(s,l)\tilde{\sigma}(s,0+,l)}{\alpha(s,l)\tilde{\sigma}(s,0+,l) + (1-\alpha(s,l)) \tilde{\sigma}(s,0-,l)} = 1 - \alpha_1(s,l).
 \end{align}
    Conversely, assume that $${\mathcal L}\pare{(x(t), i(t), \ell_t^0}_{t\geq 0})$$ denotes the unique solution of the spider martingale problem \emph{$\big(\mathcal{S}_{pi}-\mathcal{M}_{ar}\big)$} with $I=2$ for the above data.
    Then, $((2i(t)-3)x(t))_{t\geq 0}$ gives a weak solution of \eqref{eq:sde-I-equals-2}.
\end{Proposition}
\begin{proof}
{\bf Part I}

Assume that there exists a weak solution to the stochastic differential equation \eqref{eq:sde-I-equals-2}. Set $i(t):=({\sgn(y(t)) + 3})/{2}\in \{1,2\}$. 
Fix $\varepsilon>0$. 

 Using the same ideas as in the proof of Proposition \ref{prop:non-stickness}, it is not difficult to prove that
 \begin{equation}
 \label{eq:non-stickness-y}
 \lim_{\e\searrow 0+}\E\pare{\int_0^t {\bf 1}_{|y(s)|\leq \e}ds} = 0.
 \end{equation}
 In particular,
 \begin{equation}
 \label{eq:non-stickness-y-2}
 \P\pare{\lambda\pare{\{s\geq 0~|~y(s) = 0\}} = 0} = 1
 \end{equation}
(with $\lambda(ds)$ standing for the Lebesgue measure).

By the It\^o-Tanaka formula (with $\sgn(0) = 0$) and from our definitions of $\sigma_{1}, \sigma_2, b_1, b_2$, we have
\begin{align*}
(y(t))^- &= (y)^- - \int_0^t \tilde{\sigma}(s, y(s), \ell_s^0(y)){\bf 1}_{y(s)< 0}dW_s - \int_0^t \tilde{b}(s, y(s), \ell_s^0(y)){\bf 1}_{y(s)< 0}ds + \hat{\ell}_t^0(y)\\
&=(y)^- - \int_0^t {\sigma}_{1}(s, |y(s)|, \ell_s^0(y)){\bf 1}_{y(s)< 0}dW_s + \int_0^t {b}_{1}(s, |y(s)|, \ell_s^0(y)){\bf 1}_{y(s)< 0}ds + \hat{\ell}_t^0(y)
\end{align*}
where 
\begin{equation}
\label{eq:hat-l}
\hat{\ell}_t^0(y) = \frac{1}{2}\int_0^t\pare{1 - \beta(s,\ell_s^0(y))}d\ell_s^0(y).
\end{equation}
And analogously,
\begin{align*}
(y(t))^+ &= (y)^+ + \int_0^t \tilde{\sigma}(s, y(s), \ell_s^0(y)){\bf 1}_{y(s)> 0}dW_s + \int_0^t \tilde{b}(s, y(s), \ell_s^0(y)){\bf 1}_{y(s)> 0}ds + \check{\ell}_t^0(y)\\
&=(y)^+ + \int_0^t {\sigma}_{2}(s, |y(s)|, \ell_s^0(y)){\bf 1}_{y(s)> 0}dW_s + \int_0^t {b}_{2}(s, |y(s)|, \ell_s^0(y)){\bf 1}_{y(s)> 0}ds + \check{\ell}_t^0(y).
\end{align*}
where 
\begin{equation}
\label{eq:check-l}
\check{\ell}_t^0(y) = \frac{1}{2}\int_0^t\pare{1 + \beta(s,\ell_s^0(y))}d\ell_s^0(y).
\end{equation}
Let $f\in\mathcal{C}^{2}_b(\mathcal{J})$. 
We may apply a generalized It\^o's formula to $((y(t))^-)_{t\geq 0}$ and $f_1$ which is twice differentiable at zero. Observe that with our definition of $i(t)$, we have $\{t: y(t)< 0\} = \{t: i(t)=1\}$. By using \eqref{eq:non-stickness-y-2}, we have
\begin{align*}
&f_1((y(t))^-) - f_1((y)^-) = -\int_{0}^{t}\partial_xf_{1}((y(s))^-){\sigma}_{1}(s,|y(s)|,\ell_s^0(y)){\bf 1}_{y(s)\leq 0}dW_s\\
&\;+\int_{0}^{t}\pare{\partial_xf_{1}((y(s))^-){b}_{1}(s,|y(s)|,\ell_s^0(y))+ \frac{1}{2}\partial^2_{xx}f_{1}((y(s))^-){\sigma}_{1}^2(s,|y(s)|,\ell_s^0(y))}{\bf 1}_{y(s)\leq 0}ds\\
&\;+\frac{1}{2}\partial_xf_{1}(0+)\hat{\ell}_t^0(y) + \int_{0}^{t}\partial_xf_{1}(0+){\ell}_t^0((y)^-)\\
&= \int_{0}^{t}\partial_xf_{i(s)}(|y(s)|){\sigma}_{i(s)}(s,|y(s)|,\ell_s^0(y))(-{\bf 1}_{y(s)\leq 0})dW_s\\
&\;+\int_{0}^{t}\pare{\partial_xf_{i(s)}(|y(s)|){b}_{i(s)}(s,|y(s)|,\ell_s^0(y))+ \frac{1}{2}\partial^2_{xx}f_{i(s)}(|y(s)|){\sigma}_{1}^2(s,|y(s)|,\ell_s^0(y))}{\bf 1}_{i(s)=1}ds\\
&\;+\frac{1}{2}\partial_xf_{1}(0+)d\hat{\ell}_t^0(y).
\end{align*}
In the same manner, we may apply a generalized It\^o's formula to $((y(t))^+)_{t\geq 0}$ and $f_2$. Using $\{t: y(t)> 0\} = \{t: i(t)=2\}$ and \eqref{eq:non-stickness-y-2}, we have
\begin{align*}
&f_2((y(t))^+) - f_2((y)^+) = \int_{0}^{t}\partial_xf_{2}((y(s))^+){\sigma}_{2}(s,|y(s)|,\ell_s^0(y)){\bf 1}_{y(s)\geq 0}dW_s\\
&\;+\int_{0}^{t}\pare{\partial_xf_{2}((y(s))^+){b}_{2}(s,|y(s)|,\ell_s^0(y))+ \frac{1}{2}\partial^2_{xx}f_{2}((y(s))^+){\sigma}_{2}^2(s,|y(s)|,\ell_s^0(y))}{\bf 1}_{y(s)\geq 0}ds\\
&\;+\frac{1}{2}\partial_xf_{2}(0+)\check{\ell}_t^0(y) + \partial_xf_{2}(0+){\ell}_t^0((y)^+)\\
&=\int_{0}^{t}\partial_xf_{i(s)}(|y(s)|){\sigma}_{i(s)}(s,|y(s)|,\ell_s^0(y)){\bf 1}_{y(s)\geq 0}dW_s\\
&\;+\int_{0}^{t}\pare{\partial_xf_{i(s)}(|y(s)|){b}_{i(s)}(s,|y(s)|,\ell_s^0(y))+ \frac{1}{2}\partial^2_{xx}f_{i(s)}(|y(s)|){\sigma}_{i(s)}^2(s,|y(s)|,\ell_s^0(y))}{\bf 1}_{i(s)= 2}ds\\
&\;+\partial_xf_{2}(0+)\check{\ell}_t^0(y)).
\end{align*}
Adding the previous equalities and remembering the fact that $f_1(0) = f_2(0)$, we get (using once again \eqref{eq:non-stickness-y-2}):
\begin{align}
\label{eq:almost-pb-mart}
&f_{i(t)}(|y(t)|) - f_{i(0)}(|y|) = f_1((y(t))^-) - f_1((y)^-) + f_2((y(t))^+) - f_2((y)^+)\nonumber\\
&=\int_{0}^{t}\partial_xf_{i(s)}(|y(s)|){\sigma}_{i(s)}(s,|y(s)|,\ell_s^0(y))\sgn(y(s))dW_s\nonumber\\
&\;+\int_{0}^{t}\pare{\partial_xf_{i(s)}(|y(s)|){b}_{i(s)}(s,|y(s)|,\ell_s^0(y))+ \frac{1}{2}\partial^2_{xx}f_{i(s)}(|y(s)|){\sigma}_{i(s)}^2(s,|y(s)|,\ell_s^0(y))}ds\nonumber\\
&\;+\partial_xf_{1}(0+)\hat{\ell}_t^0(y) + \int_{0}^{t}\partial_xf_{2}(0+)d\check{\ell}_t^0(y).
\end{align}
Replacing with the explicit expression of $\beta$ \eqref{eq:beta} and from the definitions of $\alpha_1$ \eqref{eq:changt-variable-coeffs-alpha1} and and $\alpha_2$ \eqref{eq:changt-variable-coeffs-alpha2}  yields
\begin{align*}
&\partial_xf_{1}(0+)\hat{\ell}_t^0(y) + \partial_xf_{2}(0+)d\check{\ell}_t^0(y)\\
&=\pare{\int_0^t\frac{\pare{1-\alpha(s,\ell_s^0(y))}{\sigma}_1(s,0,\ell_s^0(y))}{\alpha(s,\ell_s^0(y)){\sigma}_2(s,0,\ell_s^0(y)) + (1-\alpha(s,\ell_s^0(y))) {\sigma}_1(s,0,\ell_s^0(y))}d\ell_s^0(y)}\partial_xf_{1}(0+)\\
&\;\;+\pare{\int_0^t\frac{\pare{\alpha(s,\ell_s^0(y))}{\sigma}_2(s,0,\ell_s^0(y))}{\alpha(s,\ell_s^0(y)){\sigma}_2(s,0,\ell_s^0(y)) + (1-\alpha(s,\ell_s^0(y))) {\sigma}_1(s,0,\ell_s^0(y))}d\ell_s^0(y)}\partial_xf_{2}(0+)\\
&=\partial_xf_{1}(0+)\int_0^t\alpha_1(s,\ell_s^0(y))d{\ell}_s^0(y) + \partial_xf_{2}(0+)\int_{0}^{t}\alpha_2(s,\ell_s^0(y))d{\ell}_s^0(y).
\end{align*}
Substituting this expression in \eqref{eq:almost-pb-mart} ensures that
\begin{align}
\label{eq:pb-mart-homo}
&f_{i(t)}(|y(t)|) - f_{i(0)}(|y|)\nonumber\\
&=\int_{0}^{t}\partial_xf_{i(s)}(|y(s)|){\sigma}_{i(s)}(s,|y(s)|,\ell^0_s(y))\sgn(y(s))dW_s\nonumber\\
&\;+\int_{0}^{t}\pare{\partial_xf_{i(s)}(|y(s)|){b}_{i(s)}(s,|y(s)|,\ell^0_s(y))+ \frac{1}{2}\partial^2_{xx}f_{i(s)}(|y(s)|){\sigma}_{i(s)}^2(s,|y(s)|,\ell^0_s(y))}ds\nonumber\\
&\;+\partial_xf_{1}(0+)\int_0^t\alpha_1(s,\ell^0_s(y))d{\ell}^0_s(y) + \partial_xf_{2}(0+)\int_{0}^{t}\alpha_2(s,\ell_s^0(y))d{\ell}^0_s(y).
\end{align}
This formula generalizes standardly to time dependent $f\in\mathcal{C}^{1,2}_b(\mathcal{J})$ (see for example the methodology of proof for the multidimensional It\^o formula in \cite{Revuz-Yor} Chapter IV). Finally, we have proved that for any $f\in\mathcal{C}^{1,2}_b(\mathcal{J}_T)$,
\begin{align*}
&\Bigg(~f_{i(t)}(t,x(t))- f_{\frac{\sgn(y) + 3}{2}}(0,|y|)\displaystyle-\int_{0}^{t}\pare{\partial_tf_{i(u)}(u,x(u))+\displaystyle\frac{1}{2}\sigma_{i(u)}^2(u,x(u),\ell_u)\partial_{xx}^2f_{i(u)}(u,x(u))}du\\
&\displaystyle+\int_{0}^{s}\pare{b_{i(u)}(t,x(u),\ell_u)\partial_xf_{i(u)}(u,x(u)}du\\
&~\displaystyle +\int_0^t\partial_xf_{1}(u,0)\alpha_1(u,\ell_u)d{\ell}_u + \int_{0}^{t}\partial_xf_{2}(u,0)\alpha_2(u,\ell_u)d{\ell}_u\Bigg)_{0\leq t\leq T},
\end{align*} 
is a $(\Psi_t)_{0 \leq t \leq T}$ martingale under the probability measure ${\mathcal L}\pare{(|y(t)|, \frac{\sgn(y(t)) + 3}{2}, \ell_t^0(y)}_{0\leq t\leq T}$ given on the canonical space. This proves the first assertion of the proposition.
\medskip
\medskip

{\bf Part II}

The converse is more involved and for the sake of conciseness we will only give the outline of the proof, leaving details aside. Before getting started, let us introduce the set ${\mathcal{C}^{1,2,1}_{\rm trans}}$ of continuous functions $f:~[0,T]\times \R\times \R^+\rightarrow \R$ such that $f\in {\mathcal C}_b^{1,2,1}([0,T]\times (\R\setminus \{0\})\times \R^+)$ and $f$ satisfies the following transmission conditions
\begin{align}
\forall (t,l)\in [0,T]\times \R^+:&\nonumber\\
&\label{trans-derivee-l-f}\partial_lf(t,0-,l) = \partial_lf(t,0+,l) = \partial_lf(t,0,l),\\
&\label{transmission}\partial_lf(t,0,l) - \frac{1-\beta(t,l)}{2}\partial_xf(t,0-,l) + \frac{1+\beta(t,l)}{2}\partial_xf(t,0+,l) = 0.
\end{align}
We form the product space 
\begin{align*}
\W =\mathcal{C}([0,T],\R)\times \mathcal{L}[0,T] 
\end{align*}
 considered as a measurable Polish space equipped with its Borel $\sigma$-algebra $\mathbb{B}(\Phi)$. 

The canonical process $\tilde{Y}$ on $\pare{\W, \mathbb{B}(\W)}$ is defined as the family of measurable mappings $\tilde{Y} = \{\tilde{Y}(t)~|~t\in [0,T]\}$ that are defined for each $t\in [0,T]$ as
\[\tilde{Y}(t):=\begin{array}{cll}
\W&\to\;\;\R\times \R^+\\
\big(y,l\big)&\mapsto\;\;(y(t),l(t)\big). \end{array}
\]
In the following, $(\Theta_t:=\sigma(\tilde{Y}(s), 0\leq s \leq t))_{0\leq t\leq T}$ stands for the canonical filtration associated to $\pare{\W, \mathbb{B}(\W)}$.

The crucial point here is that we may attach to the stochastic differential equation \eqref{eq:sde-I-equals-2} the following martingale problem
\medskip
\begin{center}
\emph{$\big({\mathcal{E}_{ds}}-\mathcal{M}_{ar}\big)$}
\end{center}
\medskip
{\it For $y\in \R$, find a probability $\P^{y}$ defined on the measurable space $(\W,\mathbb{B}(\W))$ such that:\\
-(i) $\tilde{Y}(0)=(y,0)$, $\P^{y}$-a.s.\\ 
-(ii) For each $s\in [0,T]$: 
\begin{eqnarray*}
\displaystyle\displaystyle\int_{0}^{s}\mathbf{1}_{\{y(u)>0\}}dl(u)=0,~~\P^{y}-\text{ a.s.}
\end{eqnarray*}
-(iii) 
For $y\in \R$, find a probability measure $\P^{y,0}$ on the measurable space $(\W,\mathbb{B}(\W))$ such that:\\
For any $f\in {\mathcal{C}^{1,2,1}_{\rm trans}}$, 
\begin{align*}
&\Bigg(~f(s,y(s),l(s))- f(0,y,0)\displaystyle-\int_{0}^{s}\pare{\partial_tf(u,y(u),l(u))+\displaystyle\frac{1}{2}\tilde{\sigma}^2(u,y(u),l(u))\partial_{xx}^2f(u,y(u),l(u))}du\\
&\displaystyle-\int_{0}^{s}\pare{\tilde{b}(u,y(u),l(u))\partial_xf(u,y(u)),l(u)}du\Bigg)_{0\leq s\leq T}
\end{align*} 
is a $(\Theta_s)_{0 \leq s \leq T}$ martingale under the probability measure $\P^{y}$.}

For conciseness, we will not provide a proof of this natural fact here. For a formal proof we invoke similar arguments as those exposed in this article. Since we are just looking at the natural correspondence between a possible solution of equation \eqref{eq:sde-I-equals-2} and our two branch spider martingale, we are not interested in uniqueness issues. Indeed, once the relationship between a weak solution of the stochastic differential equation \eqref{eq:sde-I-equals-2} and the martingale problem \emph{$\big({\mathcal{E}_{ds}}-\mathcal{M}_{ar}\big)$} is formally established, the uniqueness for the possible solutions of the martingale problem \emph{$\big({\mathcal{E}_{ds}}-\mathcal{M}_{ar}\big)$} derives directly from the well-posedeness for the martingale problem \emph{$\big({\mathcal{S}_{pi}}-\mathcal{M}_{ar}\big)$} and the first part of this proposition.
\medskip

Now let $${\mathcal L}\pare{(x(t), i(t), \ell_t^0}_{0\leq t\leq T})$$ denote the unique solution of the spider martingale problem \emph{$\big(\mathcal{S}_{pi}-\mathcal{M}_{ar}\big)$} with $I=2$ for the data given in \eqref{eq:changt-variable-coeffs-alpha1} and \eqref{eq:changt-variable-coeffs-alpha2} and initial condition 
$x(0) = y, i(0) = (\sgn(y) + 3)/{2}$. 

Our objective turns to show that ${\mathcal L}\pare{((2i(t)-3)x(t), \ell_t^0)_{0\leq t\leq T}}$ on $\pare{\W, \mathbb{B}(\W)}$ gives a solution of the martingale problem \emph{$\big({\mathcal{E}_{ds}}-\mathcal{M}_{ar}\big)$}.

For this purpose, set for $t\in [0,T]$
\begin{equation}
\label{eq:def-f1-f2}
f_1(t,y,l)=f(t,-y,l){\bf 1}_{y\geq 0}\;\;\;\;\;f_2(t,y,l)=f(t,y,l){\bf 1}_{y\geq 0}.
\end{equation}
With these notations (and since $x(.)\geq 0$) observe that for all $s\in [0,T]$
\begin{align*}
f(s, (2i(s)-3)x(s), \ell_s^0) &= f(s, (2i(s)-3)x(s), \ell_s^0){\bf 1}_{i(s)=1} + f(s, (2i(s)-3)x(s), \ell_s^0){\bf 1}_{i(s)=2}\\
&=f(s, -x(s), \ell_s^0){\bf 1}_{i(s)=1} + f(s, x(s), \ell_s^0){\bf 1}_{i(s)=2}\\
&=f_1(s, x(s), \ell_s^0){\bf 1}_{i(s)=1} + f_2(s, x(s), \ell_s^0){\bf 1}_{i(s)=2}\\
&=f_{i(s)}(s, x(s), \ell_s^0).
\end{align*}
Then $f(t,i,z,l) = f_i(t,z,l)$ ($i=1,2$) belongs to $\mathcal{C}^{1,2,1}_b(\mathcal{J}_T\times \R^+)$ and we are in position to apply the martingale property to the spider $(x(t),i(t),\ell_t^0)_{0\leq t\leq T}$. 

We have
\begin{align*}
&f(t, (2i(t)-3)x(t), \ell_t^0) - f(0, (2i(0)-3)x(0), 0) = f_{i(t)}(t,x(t),\ell_t^0)- f_{i(0)}(0,y,0)
\end{align*}
and
\begin{align*}
&\Big (f(t, (2i(t)-3)x(t), \ell_t^0) - f(0, (2i(0)-3)x(0), 0)\\
&\;\;\;\;-\displaystyle\int_{0}^{t}\pare{\partial_tf_{i(u)}(u,x(u),\ell_u^0)+\displaystyle\frac{1}{2}\sigma_{i(u)}^2(t,x(u),\ell_t^0)\partial_{xx}^2f_{i(u)}(u,x(u),\ell_u^0)}du\\
&\;\;\;\;\displaystyle-\int_{0}^{s}\pare{b_{i(u)}(t,x(u),\ell^0_u)\partial_xf_{i(u)}(u,x(u),\ell^0 _u}du\\
&\;\;\;\;\displaystyle -\int_0^t\partial_lf_{i(u)}(u,0,\ell^0_u)d\ell_u^0 -\int_0^t\partial_xf_{1}(u,0,\ell^0_u)\alpha_1(u,\ell^0_u)d{\ell}^0_u - \int_{0}^{t}\partial_xf_{2}(u,0)\alpha_2(u,\ell_u)d{\ell}^0_u\Big )_{0\leq t\leq T}
\end{align*}
is a $(\Psi_s)_{0\leq t\leq T}$ martingale.
Observe from our definitions of $f_1$ and $f_2$ \eqref{eq:def-f1-f2} and the formulas for $\alpha_1$ \eqref{eq:changt-variable-coeffs-alpha1} and $\alpha_2$ \eqref{eq:changt-variable-coeffs-alpha2} that 
\begin{align*}
&\int_0^t\partial_lf_{i(u)}(u,0,\ell^0_u)d\ell_u^0 +\int_0^t\partial_xf_{1}(u,0,\ell^0_u)\alpha_1(u,\ell^0_u)d{\ell}^0_u + \int_{0}^{t}\partial_xf_{2}(u,0)\alpha_2(u,\ell_u)d{\ell}^0_u\\
=&\int_0^t\pare{\partial_lf_{i(u)}(u,0,\ell^0_u) - \frac{1-\beta(u,\ell_u^0)}{2}\partial_xf(u,0,\ell^0_u)d{\ell}^0_u + \frac{1+\beta(u,\ell_u^0)}{2}\partial_xf(u,0)(u,\ell_u)}d{\ell}^0_u.
\end{align*}
Hence, since $f\in {\mathcal{C}^{1,2,1}_{\rm trans}}$, the transmission conditions \eqref{trans-derivee-l-f} and \eqref{transmission} imply
\begin{align*}
&\int_0^t\partial_lf_{i(u)}(u,0,\ell^0_u)d\ell_u^0 +\int_0^t\partial_xf_{1}(u,0,\ell^0_u)\alpha_1(u,\ell^0_u)d{\ell}^0_u + \int_{0}^{t}\partial_xf_{2}(u,0)\alpha_2(u,\ell_u)d{\ell}^0_u = 0.
\end{align*}
Consequently, by setting $\{y(s):=(2i(s)-3)x(s), s\in [0,T]\}$, we may conclude using $x(u) = (2i(u) - 3)y(u) = \sgn(y(u))y(u)$ ($u\in [0,T]$) that the process
\begin{align*}
&\Bigg(~f(s,y(s),l(s))- f(0,y,0)\displaystyle-\int_{0}^{s}\pare{\partial_tf(u,y(u),l(u))+\displaystyle\frac{1}{2}\tilde{\sigma}^2(u,y(u),l(u))\partial_{xx}^2f(u,y(u),l(u))}du\\
&\displaystyle-\int_{0}^{s}\pare{\tilde{b}(u,y(u),l(u))\partial_xf(u,y(u)),l(u)}du\Bigg)_{0\leq s\leq T}
\end{align*} 
is a $(\Theta_s)_{0 \leq s \leq T}$ martingale under the probability measure ${\mathcal L}\pare{((2i(t)-3)x(t), \ell_t^0)_{0\leq t\leq T}}$ over the canonical space $\pare{\W, \mathbb{B}(\W)}$.
\end{proof}
\begin{Corollary}
There exists a unique weak solution to the stochastic differential equation \eqref{eq:sde-I-equals-2}.
\end{Corollary}
Note that the solution of \eqref{eq:sde-I-equals-2} gives a new example of a process that belongs to the class $(\Sigma)$ extended to semimartingales: originally, the class $(\Sigma)$ was introduced for submartingale processes whose increasing finite variation process increases only on the zero-set of the submartingale (for an account on the class $(\Sigma)$, we refer the reader to \cite{Eyi-Obiang-al} and \cite{Eyi-Obiang-al-2}).
\medskip

\appendix 


\section{Some technical results}
\subsection{Fatou's Lemma for weakly convergence measures}
We first start this appendix by recalling some interesting results concerning the convergence for sequences of type: $\displaystyle \int f_n d\mu_n$, where the sequences $f_n$ and $\mu_n$ are respectively, measurable functions and measures defined on a metric space $S$, endowed with its Borel sigma algebra $\mathbb{B}(S)$. The following generalization of Fatou's Lemma, and Lebesgue's theorem proved in \cite{Feinberg}, give necessary conditions on the convergence of $\displaystyle \int f_n d\mu_n$ to $\displaystyle \int f d\mu$, where $\mu_n$ converges weakly to a measure $\mu$, and $f_n$ to a measurable function $f$. 

Let $(S,d)$ a metric space and $\mathbb{B}(S)$ its borel sigma algebra. We denote in the sequel by $\mathcal{M}(S)$ the set of all finite measures on the measurable space $(S,\mathbb{B}(S))$.
\begin{Corollary}\label{cr: Fatou}(Fatou's Lemma for weakly convergence measures, see Corollary 4.1 in \cite{Feinberg}.) Let $\mu_n$ be a sequence in $\mathcal{M}(S)$ converging weakly to $\mu$ in $\mathcal{M}(S)$. Let $f_n$ a sequence of lower semiequicontinuous sequence of real-valued function on $S$, namely:
\begin{align*}
&\forall s\in S,~~\forall \varepsilon>0,~~\exists \delta_\varepsilon>0,~~\forall n\in \mathbb{N},~~\forall y\in B_{\delta_\varepsilon}(s):\\
&f_n(y)> f_n(s)-\varepsilon,\,\,\text{where }\;B_{\delta_\varepsilon}(s):=\{y \in S:~d(y,s)\leq \delta_\varepsilon\}.    
\end{align*}
Assume that $f_n$ is asymptotic uniformly integrable with respect to the sequence of measures $\mu_n$, namely:
$$ \lim_{K \to +\infty} \limsup_{n \to +\infty } \int_Sf_n(s)\mathbf{1}_{\{|f_n(s)|\ge K\}}\mu_n(ds)=0,$$ then:
$$\int_S\liminf_{n\to + \infty}f_n(s)\mu(ds)\leq \liminf_{n \to +\infty}\int_Sf_n(s)\mu_n(ds).$$
\end{Corollary}
\begin{Corollary}\label{cr Lebesgue faible}(Lebesgue's Theorem for weakly convergence measures, see Corollary 5.1 in \cite{Feinberg}.) Let $\mu_n$ be a sequence in $\mathcal{M}(S)$ converging weakly to $\mu$ in $\mathcal{M}(S)$. Let $f_n$ a sequence of real-valued measurable function on $S$. Assume that $f_n$ is asymptotic uniformly integrable with respect to the sequence of measures $\mu_n$, namely:
$$\lim_{K \to +\infty} \limsup_{n \to +\infty} \int_S f_n(s)\mathbf{1}_{\{|f_n(s)|\ge K \}}\mu_n(ds)=0,$$
and $\lim_{n\to +\infty}\lim_{u\to s}f_n(u)$ exists $\mu$ almost everywhere $s\in S$. Then:
$$\lim_{n\to +\infty}\int_Sf_n(s)\mu_n(ds)=\int_S\lim_{n\to +\infty}\lim_{u\to s}f_n(u)\mu(ds)=\int_S\lim_{n\to +\infty}f_n(s)\mu(ds).$$
\end{Corollary}

\subsection{Proof of Lemma \ref{lem:modules-cont}}
\begin{proof}

\label{sk: ineg modulus inequality}

Fix $n\in \N^\ast$. 

Let $m\geq 1$ an integer.

From the definition of $P_{t_n}^{(x_n,i_n),l_n}$ (condition $(\mathcal{S})-i)$) we have $\text{for all }s\in [0,t_n]\;\;P_{t_n}^{(x_n,i_n),l_n} - {\rm a. s}$:
\begin{align}
\label{eq:entrezeroettn}
|x(s) - x_n|^{2m} = 0 &\;\;\text{ ; }\;\;|l(s) - l_n|^{2m} =0\;\;\text{ ; }\;\; |\ell_s|^{2m}= 0.
\end{align}

Define $\tau_K=\inf(s\geq t_n, |x(s)|\geq K)$ for $K\geq 1$.
This localizing sequence of stopping times is such that
$|x(s\wedge \tau_K)|\leq K$ and $\lim_{K\rightarrow +\infty}|x(s\wedge \tau_K)| = |x(s)|$ for all $s\in [0,T]$ (by continuity of $x$, which is guaranteed by the fact that we are working on $(\Phi, \B(\Phi))$).

From \eqref{eq:equation-freidlin-sheu} we are in position to apply It\^o's formula to $y\mapsto y^{2m}$ and $(x(s\wedge \tau_K))_{s\in [0,T]}$. Note that since the process $s\mapsto \ell(s)$ is increasing it is of bounded variation and its bracket is simply the null process. 
Using crucially condition $(\mathcal{S})-ii)$, we have for all $s\in [t_n, T]$:
\begin{align*}
\displaystyle
x(s\wedge \tau_K)^{2m} &= x_n^{2m} + 
2m\displaystyle\int_{t_n}^{s\wedge \tau_K}x(u\wedge \tau_K)|x(u\wedge \tau_K)|^{2m - 2}b_{i(u)}(t_n,x(u\wedge \tau_K),l_n)du\\
&\hspace{0,5 cm}+ m(2m-1)\int_{t_n}^{s\wedge \tau_K} |x(u)|^{2(m-1)}\sigma^2_{i(u)}(t_n,x(u\wedge \tau_K),l_n)du\\
&\hspace{0,5 cm}
+2m\int_{t_n}^{s\wedge \tau_K}|x(u\wedge \tau_K)|^{2m - 2}\sigma_{i(u)}(t_n,x(u\wedge \tau_K),l_n)x(u\wedge \tau_K)dW^{(t_n,l_n)}_u.
\end{align*}
The stochastic integral in the r.h.s. is a martingale and the stopping time $s\wedge \tau_K$ is bounded by $T$. 
We are in position to apply the optional stopping theorem for martingales, we obtain
\[
\mathbb{E}^{P_{t_n}^{(x_n,i_n),l_n}}\croc{\int_{t_n}^{s\wedge \tau_K}|x(u\wedge \tau_K)|^{2m - 2}\sigma_{i(u)}(t_n,x(u\wedge \tau_K),l_n)x(u\wedge \tau_K)dW^{(t_n,l_n)}_u}= 0.
\]

Observe that $|x|x|^{2m-2}y| = |x|^{2m-1}|y|\leq \max(x^{2m}, y^{2m}) \leq 2(x^{2m} + y^{2m})$ and
$||x|^{2m-2}y^2|\leq 2(x^{2m} + y^{2m})$.

Hence, we deduce that there exists a constant $C_m$ (independent of $n$) such that:
\begin{align*}
\displaystyle
\mathbb{E}^{P_{t_n}^{(x_n,i_n),l_n}}\croc{|x(s\wedge \tau_K)|^{2m}} &\leq x^{2m}_n + 
C_m\displaystyle\int_{t_n}^{s\wedge \tau_K}|b|^{2m}du\,+\, C_m\int_{t_n}^{s\wedge \tau_K} |\sigma|^{2m}du\\
&\hspace{0,5 cm}+C_m\int_{t_n}^{s\wedge \tau_K} \mathbb{E}^{P_{t_n}^{(x_n,i_n),l_n}}\croc{|x(u\wedge \tau_K)|^{2m}}du.
\end{align*}
Applying Gr\"onwall's inequality and using that $(s\wedge \tau_K -t_n) \leq T$ yields
\begin{align}
\label{eq:result-gronwall}
\displaystyle
\mathbb{E}^{P_{t_n}^{(x_n,i_n),l_n}}\croc{|x(s\wedge \tau_K)|^{2m}}  &\leq {\rm e}^{C_m T}\pare{x^{2m}_n + 
C_m T\pare{|b|^{2m} + |\sigma|^{2m}}}.
\end{align}
By application of Fatou's lemma, we deduce finally from \eqref{eq:entrezeroettn} and \eqref{eq:result-gronwall} that
\begin{align*}
\mathbb{E}^{P_{t_n}^{(x_n,i_n),l_n}}\croc{|x(s)|^{2m}} &
\leq \liminf\limits_{K\rightarrow +\infty}\mathbb{E}^{P_{t_n}^{(x_n,i_n),l_n}}\croc{|x(s\wedge \tau_K)|^{2m}} \\
&\leq {\rm e}^{C_m T}\pare{x^{2m}_n + 
C_m T\pare{|b|^{2m} + |\sigma|^{2m}}},
\end{align*}
holding for all $s\in [0,T]$.

From the equality \eqref{eq:equation-freidlin-sheu}, it not difficult to deduce that there exists a constant $C_m$, depending only on the data $(T,|b|,|\sigma|)$, such that
\begin{align}
\label{eq:control-moments-sup-ext}
\sup_{n\geq 1}\sup_{s\in [0,T]}\mathbb{E}^{P_{t_n}^{(x_n,i_n),l_n}}\croc{|x(s)|^{2m} + |l(s)|^{2m} + |\ell_s|^{2m}} \leq C_m\pare{1+x^{2m}_n + l_n^{2m}}.
\end{align}
We have to work a little bit more in order to get the supreme inside the expectation.

Let us introduce $(y(s))_{s\in [t_n,T]}$ defined for all $s\in [t_n,T]$ by
\begin{equation}
\label{def:y}
y(s) := x(s) - \ell_s = x_n + \int_{t_n}^s  \sigma_{i(u)}(t_n,x(u),l_n)dW^{(t_n,l_n)}_u + \int_{t_n}^s b_{i(u)}(t_n,x(u),l_n)du 
\end{equation}
(the last equality comes from \eqref{eq:equation-freidlin-sheu}).

The process $(y(s))_{t_n\leq s\leq T}$ is a classical It\^o process. Define $\xi_K=\inf(s\geq t_n, |y(s)|\geq K)$ for $K\geq 1$. 

Following the same kind of computations as in \cite{Karatzas Book} p.390 solution of Problem 3.15 (in the context of classical SDE), we have from \eqref{def:y}:
\begin{align*}
|y(s\wedge \xi_K)|^{2m} &\leq C_m\pare{x_n^{2m} + \Big |\int_{t_n}^{s\wedge \xi_k} b_{i(u)}(t_n,x(u\wedge \xi_K),l_n)du \Big |^{2m}}\\
&\hspace{2 cm}+C_m \Big |\int_{t_n}^{s\wedge \xi_K}  \sigma_{i(u)}(t_n,x(u\wedge \xi_K),l_n)dW^{(t_n,l_n)}_u\Big |^{2m}
\end{align*}
and H\"older's inequality provides
\begin{align*}
|y(s\wedge \xi_K)|^{2m} &\leq C_m\pare{x_n^{2m} + ({s\wedge \xi_K} - t_n)^{2m-1}\int_{t_n}^{s\wedge \xi_K} \big |b_{i(u)}(t_n,x(u\wedge \xi_K),l_n)\Big |^{2m}du }\\
&\hspace{2 cm}+C_m \sup_{r\in [t_n,s]}\Big |\int_{t_n}^{r\wedge \xi_K}  \sigma_{i(u)}(t_n,x(u\wedge \xi_K),l_n)dW^{(t_n,l_n)}_u\Big |^{2m}.
\end{align*}
Combining Burkh\"older-Davies-Gundy and H\^older inequalities gives
\begin{align*}
&\mathbb{E}^{P_{t_n}^{(x_n,i_n),l_n}}\croc{\sup_{r\in [t_n,s\wedge \xi_K]}\Big |\int_{t_n}^{r\wedge \xi_K}  \sigma_{i(u)}(t_n,x(u\wedge \xi_K),l_n)dW^{(t_n,l_n)}_u\Big |^{2m}}\\
&\hspace{0,5 cm}\leq C_m \mathbb{E}^{P_{t_n}^{(x_n,i_n),l_n}}\croc{\Big |\int_{t_n}^{s\wedge \xi_K}  \sigma_{i(u)}(t_n,x(u\wedge \xi_K),l_n)\Big |^{2}du\Big |^m}\\
&\hspace{0,5 cm}\leq C_m ({s\wedge \xi_K} - t_n)^{2m-1}\mathbb{E}^{P_{t_n}^{(x_n,i_n),l_n}}\croc{\Big |\int_{t_n}^{s\wedge \xi_K}  \sigma^{2m}_{i(u)}(t_n,x(u\wedge \xi_K),l_n)du\Big |^{2m}}.
\end{align*}
By using the bounded character of coefficients $(b,\sigma)$, we retrieve directly
\begin{align*}
\mathbb{E}^{P_{t_n}^{(x_n,i_n),l_n}}\croc{\sup_{r\in [t_n,s\wedge \xi_K]}|y(r\wedge \xi_K)|^{2m}} \leq C_m\pare{x_n^{2m} + (|b|^{2m} + |\sigma|^{2m})\pare{s - t_n}^{2m}}.
\end{align*}
By choosing $s=T$ and letting $K$ tend to infinity, Fatou's lemma ensures
\begin{align}
\label{ineq:supy}
\mathbb{E}^{P_{t_n}^{(x_n,i_n),l_n}}\croc{\sup_{r\in [t_n,T]}|y(r)|^{2m}} \leq C_m\pare{x_n^{2m} + (|b|^{2m} + |\sigma|^{2m})\pare{T - t_n}^{2m}}.
\end{align}

Now from the definition of $(y(s))_{s\in [0,T]}$ combined with the inequality $(a+b)^{2m}\leq 2m(a^{2m} + b^{2m})$, and remembering that $s\mapsto \ell_s$ is an increasing process, we see that
\begin{align*}
&\mathbb{E}^{P_{t_n}^{(x_n,i_n),l_n}}\croc{\sup_{r\in [t_n,T]}|x(r)|^{2m}}\\ 
&\hspace{0,5 cm}\leq 2m\pare{\mathbb{E}^{P_{t_n}^{(x_n,i_n),l_n}}\croc{\sup_{r\in [t_n,T]}|y(r)|^{2m}} + \mathbb{E}^{P_{t_n}^{(x_n,i_n),l_n}}\croc{|\ell_T|^{2m}}}.
\end{align*}
Hence from \eqref{ineq:supy} and \eqref{eq:control-moments-sup-ext}, we prove that there exists a constant $C_m$ depending only on the data $(T,|b|,|\sigma|)$, such that
\begin{align*}
\sup_{n\geq 0}\,\mathbb{E}^{P_{t_n}^{(x_n,i_n),l_n}}\croc{|x|_{(0,T)}^{2m} + |l|_{(0,T)}^{2m} + |\ell|_{(0,T)}^{2m}}\leq 
C_m\pare{1 + x_n^{2m} + l_n^{2m}}.
\end{align*}
\medskip

Let us now turn to the control of the moduli of continuity.

Note that $(y(s))_{s\in [t_n,T]}$ is an It\^o process. Our assumptions on the coefficients allow to apply the results of \cite{Fischer} Theorem 1 (with $p=2$) to our context. We prove easily that 
\begin{equation}
\label{eq:modulus-y}
\forall \theta\in (0,1),~~\sup_{n\ge 0}E^{(x_n,i_n),l_n}_{t_n}\croc{\omega\pare{y,\theta}^2 }~~\leq~~C\theta \ln\pare{\displaystyle{2T}/{\theta}}.
\end{equation}
Our problem is now to pass from process $y$ to the process $x$ in the modulus.

Note that standard methods (meaning for e.g. as in \cite{Karatzas Book} p.390 solution of Problem 3.15) cannot be directly applied here. In particular, we do not know whether $s\mapsto \ell_s$ is the local time of an It\^o process or not: in fact we have strong insights that this is not the case due to the fact that $({X}(t) = (x(t), i(t)))_{t\in [0,T]}$ is not adapted to a Brownian filtration.

Nonetheless, $x$ takes only positive values and $\ell$ is an increasing process starting form $0$ at time $t_n$ with corresponding measure carried -- almost surely under $P^{(x_n,i_n),l_n}_{t_n}$ -- only by $\{u\in [t_n, T]: x(u)= 0\}$. Hence, we deduce from \eqref{eq:equation-freidlin-sheu} that $\pare{(x(s))_{s\in [t_n, T]}, (\ell_s)_{s\in [t_n, T]}}$ is solution of the Skorokhod problem associated to $(y(s))_{s\in [t_n,T]}$ in the sense of Lemma 2.1 p.239 in \cite{Revuz-Yor}. In particular, we have the representation
\begin{equation}
\label{eq:skorokhod-representation}
\ell_s = \sup_{u\in [t_n,s]}\pare{-y(u)\vee 0}\;\;\;\;\;\;s\in [t_n, T]
\end{equation}
that holds almost surely under $P^{(x_n,i_n),l_n}_{t_n}$.

Let $s',s\in [t_n, T]$ with $s'<s$ and $s-s'\leq \theta$. Assume that $\ell_{s'}< \ell_{s}$. From the fact that $s\mapsto -y_s\vee 0$ is continuous on $[t_n, T]$, the representation \eqref{eq:skorokhod-representation} implies that there exists $\xi'\in [s',s]$ and $\xi\in [s',s]$ with $\xi'\leq \xi$ such that
\begin{align*}
0\leq  \ell_s - \ell_{s'}&= \sup_{u\in [t_n,s]}\pare{-y(u)\vee 0} - \sup_{u\in [t_n,s']}\pare{-y(u)\vee 0} = (-y(\xi)\vee 0) - (-y(\xi')\vee 0)\\
&\leq y(\xi') - y(\xi)\leq w(y,\theta).
\end{align*}
from which we deduce that
\[
\omega\pare{\ell,\theta} \leq
\omega\pare{y,\theta}.
\]
By sub-linearity of the modulus of continuity
\[
\omega\pare{x,\theta} = \omega\pare{y+\ell,\theta}\leq \omega\pare{y,\theta} + \omega\pare{\ell,\theta}\leq 2\omega\pare{y,\theta}.
\]
Inequality \eqref{eq:modulus-continuity-x-l} is deduced directly as a consequence of \eqref{eq:modulus-y}. The inequality \ref{eq:modulus-continuity-X} follows then by definition of the distance $d^{\mathcal J}$ (using the fact that $i(s)\neq i(t)$ ($s<t$) ensures that there exists $u\in [s,t]$ such that $x(u)=0$).


\end{proof} 


\end{document}